\documentclass[11pt]{amsart}
\usepackage{amssymb}

\newtheorem{lemma}{Lemma}[section]
\newtheorem{theorem}{Theorem}[section]

\newtheorem{construction}{Construction}[section]

\newenvironment{adfenumerate}{
\begin{enumerate}
\setlength{\itemsep}{0.5mm}
\setlength{\parskip}{0mm}
\setlength{\parsep}{0mm}
}{
\end{enumerate}
}

\newcommand{\adfmod}[1]{~(\mathrm{mod}~#1)}    % This is like \pmod except that it works properly.
\newcommand{\adfPENT}{\mathop{\mathrm{PENT}} } % For pentagonal geometries in theorems and math mode.
\newcommand{\ADFvfyParStart}[1]{}
\newcommand{\adfDgap}{\vskip 0mm}              % Start of a design.
\newcommand{\adfLgap}{\vskip 0mm}              % Start of the actual blocks of a design.
\newcommand{\adfsplit}{\par}                   % Split the list of blocks at these points.
\newcommand{\adfBfont}{\normalsize}                 % Font for blocks of a design
\newcommand{\adfhide}[1]{}                     % For stuff we don't want to appear.
\begin{document}
\title{Pentagonal geometries with block sizes 3, 4 and 5}
\author{A. D. Forbes}
\address{School of Mathematics and Statistics\\
        The Open University\\
        Walton Hall, Milton Keynes MK7 6AA, UK}
\email{anthony.d.forbes@gmail.com}
\date{\today}             % USE THIS FOR FINAL SUBMISSION TO A JOURNAL
% Version 6.15 to ArXiv 27 Jul 2020
\subjclass[2010]{05B25, 51E12}
\keywords{pentagonal geometry, group divisible design}

\maketitle
\begin{abstract}
A pentagonal geometry PENT($k$, $r$) is a partial linear space, where
every line, or block, is incident with $k$ points,
every point is incident with $r$ lines, and
for each point $x$, there is a line incident with precisely those points that are not collinear with $x$.
An opposite line pair in a pentagonal geometry consists of two parallel lines such that each point
on one of the lines is not collinear with precisely those points on the other line.

We give a direct construction for an infinite sequence of pentagonal geometries with block size 3 and connected deficiency graphs.
Also we present 39 new pentagonal geometries with block size 4 and five with block size 5,
all with connected deficiency graphs.
Consequentially we determine the existence spectrum up to a few possible exceptions for
PENT(4, $r$) that do not contain opposite line pairs and for
PENT(4, $r$) with one opposite line pair.
More generally, given $j$ we show that there exists a PENT(4, $r$) with $j$ opposite line pairs
for all sufficiently large admissible $r$.
Using some new group divisible designs with block size 5 (including types $2^{35}$, $2^{71}$ and $10^{23}$)
we significantly extend the known existence spectrum for PENT(5, $r$).
\end{abstract}

%%%%%%%%%%%%%%%%%%%%%%%%%%%%%%%%%%%%%%%%%%%%%%%%%%%%%%%%%%%%%%%%%%%%%%%%%%%%%%%%%%%%%%%%%%
%%%%%%%%%%%%%%%%%%%%%%%%%%%%%%%%%%%%%%%%%%%%%%%%%%%%%%%%%%%%%%%%%%%%%%%%%%%%%%%%%%%%%%%%%%
%%%%%%%%%%%%%%%%%%%%%%%%%%%%%%%%%%%%%%%%%%%%%%%%%%%%%%%%%%%%%%%%%%%%%%%%%%%%%%%%%%%%%%%%%%
%%%%%%%%%%%%%%%%%%%%%%%%%%%%%%%%%%%%%%%%%%%%%%%%%%%%%%%%%%%%%%%%%%%%%%%%%%%%%%%%%%%%%%%%%%

\section{Introduction}\label{sec:Introduction}
% PLS: (i) Any line is incident with at least two points; (ii) any pair of distinct points is incident with at most one line.
A {\em pentagonal geometry} is a uniform, regular partial linear space such that
for each point $x$, there is a line which is incident with precisely those points that are not collinear with $x$.
If $k$ is the number of points incident with a given line and
$r$ is the number of lines incident with a given point, then $k$ and $r$ are constants,
and we denote a pentagonal geometry with these parameters by $\adfPENT(k,r)$.
We always assume that $k \ge 2$ and $r \ge 1$.
For terminological convenience, we will break the point--line symmetry by regarding lines as sets of points.
Thus two points are collinear if they are on (i.e.\ elements of) the same line.
Lines are also called blocks, and we often refer to the parameter $k$ of a $\adfPENT(k,r)$ as the block size.

The concept of a pentagonal geometry was introduced in \cite{BallBambergDevillersStokes2013} to provide a generalization
of the pentagon and thus fill a gap in the theory of generalized polygons.
Feit \& Higman proved that a generalized $n$-gon of order $(s, t)$ with $s, t > 1$ exists only if $n \in \{2, 3, 4, 6, 8\}$,
\cite[Theorem 1]{FeitHigman1964}.
The geometry described by Ball, Bamberg, Devillers \& Stokes in \cite{BallBambergDevillersStokes2013}
is based on the observation that for each vertex $x$ of a pentagon, the two vertices that are not collinear with $x$ form a line.
According to the above definition the pentagon is a $\adfPENT(2,2)$.

The line consisting of the points that are not collinear with point $x$ is
called the {\em opposite line} to $x$ and is denoted by $x^\mathrm{opp}$.
An {\em opposite line pair} in a pentagonal geometry is a pair of non-intersecting lines such that for each point
$x$ on one of the lines, $x^\mathrm{opp}$ is the other line.
For any integer $k \ge 2$, there exists a $\adfPENT(k,1)$, which consists of two lines that form
an opposite line pair, and we say that such a pentagonal geometry is {\em degenerate}.
Conversely, a $\adfPENT(k,r)$ with $r > 1$ is {\em non-degenerate}.
As might be suggested by these terms, pentagonal geometries that do not contain opposite line pairs
are considered to be more interesting than those that do.

For the general theory as well as further background material, we refer the reader to \cite{BallBambergDevillersStokes2013},
and for convenience we reproduce some of its results.
\begin{lemma} \label{lem:v, b}
A pentagonal geometry $\adfPENT(k,r)$ has $r(k-1) + k + 1$ points and $(r(k-1) + k + 1)r/k$ lines.
A $\adfPENT(k,r)$ exists only if $r(r - 1) \equiv 0 \adfmod{k}$.
\end{lemma}
\begin{proof}
See \cite[Lemma 2.1]{BallBambergDevillersStokes2013}.
\end{proof}
Given $k$, we say that $r$ is {\em admissible} if $r(r - 1) \equiv 0 \adfmod{k}$.
\begin{lemma} \label{lem:r >= k}
For a non-degenerate $\adfPENT(k,r)$, we have $r \ge k$.
\end{lemma}
\begin{proof}
See \cite[Lemma 2.2]{BallBambergDevillersStokes2013}.
\end{proof}

The {\em deficiency graph} of a pentagonal geometry has as its vertices the points of the geometry, and
there is an edge $x \sim y$ precisely when $x$ and $y$ are not collinear.
\begin{lemma} \label{lem:deficiency graphvb}
The deficiency graph of a pentagonal geometry $\adfPENT(k,r)$ containing $m$ opposite line pairs
is the disjoint union of $m$ copies of the complete bipartite graph $K_{k,k}$ together with,
if $r(k-1) + k + 1 > 2mk$, a $k$-regular graph that has girth at least $5$.
\end{lemma}
\begin{proof}
See \cite[Lemma 2.5]{BallBambergDevillersStokes2013}.
\end{proof}
Hence the deficiency graph of a $\adfPENT(k,r)$ without an opposite line pair
has $r(k - 1) + k + 1$ vertices, is $k$-regular and has girth at least 5.
\begin{lemma} \label{lem:(k,k), (k,k+1)}
There exists a pentagonal geometry $\adfPENT(k,k)$ only when $k = 2$, $3$, $7$ and possibly $57$.
There exists a pentagonal geometry $\adfPENT(k-1,k)$ only when $k = 3$, $7$ and possibly $57$.
\end{lemma}
\begin{proof}
See \cite[Sections 3 and 4]{BallBambergDevillersStokes2013}.
\end{proof}
The pentagonal geometries of Lemma~\ref{lem:(k,k), (k,k+1)} are related to Moore graphs---it turns out that
the deficiency graphs of the $\adfPENT(k,k)$ for $k = 2$, $3$ and $7$ are
the pentagon, the Petersen graph and the Hoffman--Singleton graph respectively.
Removing a point and its opposite line from a $\adfPENT(k,k)$, $k > 2$, yields a $\adfPENT(k-1,k)$,
\cite[Theorem 4.4]{BallBambergDevillersStokes2013}.

Section~\ref{sec:Block size 3} of this paper deals with pentagonal geometries with block size 3.
In Theorem~\ref{thm:PENT-3-connected} we describe a direct construction for $\adfPENT(3,r)$, $r \ge 5$, $r \equiv 3 \adfmod{6}$, where
the deficiency graph is connected.

In Section~\ref{sec:Block size 4} we present 39 new $\adfPENT(4,r)$ with connected deficiency graphs.
We use them to determine the existence spectrum up to a few possible exceptions
for $\adfPENT(4,r)$ geometries without opposite line pairs
(Theorems~\ref{thm:PENT-4-0OLP-constructed-weak} and \ref{thm:PENT-4-0OLP-constructed})
as well as
for $\adfPENT(4,r)$ geometries with precisely one opposite line pair
(Theorem~\ref{thm:PENT-4-1OLP-constructed-weak}).
More generally, we show that given integer $j \ge 2$, there exists a $\adfPENT(4, r)$ with precisely $j$ opposite line pairs
for all sufficiently large $r \equiv 0 \textrm{~or~} 1 \adfmod{4}$;
this is Theorem~\ref{thm:PENT-4-jOLP-constructed-weak}.

Finally, in Section~\ref{sec:Block size 5} we present five new $\adfPENT(5,r)$ systems (also with connected deficiency graphs) with which,
together with some new group divisible designs, we significantly extend the known existence spectrum of
pentagonal geometries with block size 5.

For the creation of new pentagonal geometries from existing ones, we use Theorem~\ref{thm:GDD-basic}, below,
a straightforward adaptation of Wilson's Fundamental Construction, \cite{WilsonRM1972}, \cite[Theorem IV.2.5]{GreigMullen2007}, and a simple extension of \cite[Theorems 7 and 8]{GriggsStokes2016}.
First we require some definitions.

A {\em group divisible design}, $k$-GDD, of type $g_1^{u_1} g_2^{u_2} \dots g_n^{u_n}$ is
an ordered triple ($V,\mathcal{G},\mathcal{B}$) such that:
\begin{adfenumerate}
\item[(i)]{$V$
is a base set of cardinality $u_1 g_1 + u_2 g_2 + \dots + u_n g_n$;}
\item[(ii)]{$\mathcal{G}$
is a partition of $V$ into $u_i$ subsets of cardinality $g_i$, $i = 1, 2, \dots, n$, called \textit{groups};}
\item[(iii)]{$\mathcal{B}$
is a non-empty collection of $k$-subsets of $V,$ called \textit{blocks}; and}
\item[(iv)]{each pair of elements from distinct groups occurs in precisely one block but no pair of
elements from the same group occurs in any block.}
\end{adfenumerate}
A $k$-GDD of type $q^k$ is also called a {\em transversal design}, TD$(k,q)$.
A {\em parallel class} in a group divisible design is a subset of the block set in which each element of the base set appears exactly once.
A $k$-GDD is called {\em resolvable}, and is denoted by $k$-RGDD, if the entire set of blocks can be partitioned into parallel classes.
If there exist $k$ mutually orthogonal Latin squares (MOLS) of side $q$, then there exists
a $(k+2)$-GDD of type $q^{k+2}$ and
a $(k+1)$-RGDD of type $q^{k+1}$, \cite[Theorem III.3.18]{AbelColbournDinitz2007}.
Furthermore, as is well known, there exist $q - 1$ MOLS of side $q$ whenever $q$ is a prime power.

\begin{theorem}
\label{thm:GDD-basic}
Let $k \ge 2$ be an integer.
For $i = 1$, $2$, \dots, $n$, let $r_i$ be a positive integer, let $v_i = (k-1) r_i + k + 1$,
and suppose there exists a pentagonal geometry $\adfPENT(k, r_i)$ with exactly $p_i$ opposite line pairs.
Suppose also that there exists a $k$-$\mathrm{GDD}$ of type $v_1^{u_1} v_2^{u_2} \dots v_n^{u_n}$.
Let $N = u_1 + u_2 + \dots + u_n$,
$R = u_1 r_1 + u_2 r_2 + \dots + u_n r_n$ and
$P = u_1 p_1 + u_2 p_2 + \dots + u_n p_n$.
Then there exists a  pentagonal geometry
$\adfPENT(k, R + (N - 1)(k + 1)/(k - 1))$ with exactly $P$ opposite line pairs.
\end{theorem}
\begin{proof}
Overlay each group of size $v_i$ with a $\adfPENT(k,r_i)$, $i = 1$, 2, \dots, $n$.
Clearly the blocks of the group divisible design contain no opposite line pair.
Therefore any opposite line pairs in the $\adfPENT(k, R + (N - 1)(k + 1)/(k - 1))$
must arise from the $\adfPENT(k,r_i)$.
\end{proof}

%%%%%%%%%%%%%%%%%%%%%%%%%%%%%%%%%%%%%%%%%%%%%%%%%%%%%%%%%%%%%%%%%%%%%%%%%%%%%%%%%%%%%%%%%%
%%%%%%%%%%%%%%%%%%%%%%%%%%%%%%%%%%%%%%%%%%%%%%%%%%%%%%%%%%%%%%%%%%%%%%%%%%%%%%%%%%%%%%%%%%
%%%%%%%%%%%%%%%%%%%%%%%%%%%%%%%%%%%%%%%%%%%%%%%%%%%%%%%%%%%%%%%%%%%%%%%%%%%%%%%%%%%%%%%%%%
%%%%%%%%%%%%%%%%%%%%%%%%%%%%%%%%%%%%%%%%%%%%%%%%%%%%%%%%%%%%%%%%%%%%%%%%%%%%%%%%%%%%%%%%%%

\section{Block size 3}\label{sec:Block size 3}
By Lemma~\ref{lem:v, b}, a pentagonal geometry with block size 3 has $2r + 4$ points,
and $r$ is admissible if and only if $r \equiv 0 \textrm{~or~} 1 \adfmod{3}$.
The existence spectrum was determined in \cite[Theorem 9]{GriggsStokes2016},
where it is proved that $\adfPENT(3,r)$ geometries exist for all admissible $r$ except $r \in \{4, 6\}$.
Also it is proved in \cite{ForbesGriggsStokes2020} that a $\adfPENT(3,r)$ containing no opposite line pair
exists for each admissible $r \not\in \{1, 4, 6, 7\}$.

Here we address the construction of pentagonal geometries with connected deficiency graphs.
The authors of \cite{BallBambergDevillersStokes2013} remark that
only a finite number of such $\adfPENT(k,r)$ with $k \ge 3$ are known to exist.
Furthermore, in \cite{GriggsStokes2016}, \cite{ForbesGriggsStokes2020} and in
Sections~\ref{sec:Block size 4} and \ref{sec:Block size 5} of this paper,
every pentagonal geometry created by a construction involving a group divisible design with $u$ groups
has a deficiency graph consisting of at least $u > 1$ connected components.

We now give a direct construction for infinitely many $\adfPENT(3,r)$ that have connected deficiency graphs.

\begin{lemma}
\label{lem:Triples}
Let $m$ be an integer such that $m \ge 5$.
Then there exists a set $T$ of $m$ ordered triples of the form $(0, x, z)$ such that
\begin{align*}&\bigcup\{\{x, z-x, z\}: (0, x, z) \in T\} \\
  & ~~~~~ = \left\{\begin{array}{ll}\{3, 4, \dots, 3m + 2\} & \textrm{~if~} m \equiv 0,1 \adfmod{4},\\
                                       \{3, 4, \dots, 3m + 1, 3m + 3\} & \textrm{~if~} m \equiv 2,3 \adfmod{4}.
\end{array}\right.
\end{align*}
\end{lemma}
\begin{proof}
Let $M = \{1, 2, \dots, 2m\}$ if $m \equiv 0 \textrm{~or~} 1 \adfmod{4}$,
$M = \{1, 2, \dots, 2m-1, 2m+1\}$ if $m \equiv 2 \textrm{~or~} 3 \adfmod{4}$.
Create a partition of $M$ into a set of $m$ mutually disjoint ordered pairs
$$L = \{(a_i, b_i): i = 1, 2, \dots, m\}$$
such that the difference set
$\{b_i - a_i: i = 1, 2, \dots, m\}$ is $\{3, 4, \dots, m + 2\}$.
Since $m \ge 5$ this is possible by \cite{BermondBrouwerGermaAbe1978} and \cite{Simpson1983},
where explicit constructions are given, which we have adapted as shown in Table~\ref{tab:Triples}.
For the two cases not covered by Table~\ref{tab:Triples}, we have from
\cite{BermondBrouwerGermaAbe1978} and \cite{Simpson1983} respectively:
\begin{align*}
m = 5: &~ L = \{(1,8),(4,10),(2,7),(5,9),(3,6)\}, \\
m = 6: &~ L = \{(6,9),(1,5),(3,8),(7,13),(4,11),(2,10)\}.
\end{align*}
The triples are given by
$$T = \{(0, b_i - a_i, b_i + m + 2): (a_i, b_i) \in L,~~ i = 1, 2, \dots, m\}.$$
It is easily verified that $T$ has the stated property.
\end{proof}

% Check using Check-Lemma-2.1-Tables; table rows must end with \\ or %\\
\begin{table}[h]
\begin{tabular}{ll}
\multicolumn{2}{c}{$m = 4 t$, $m \ge 8$} \\
%TabStart
($2t - j - 2,  2t + j + 2$), & $0 \le j \le t - 3$ \\
($ t - j,      3t + j + 1$), & $0 \le j \le t - 1$ \\
($6t - j - 1,  6t + j + 2$), & $0 \le j \le t - 2$ \\
($5t - j,      7t + j + 2$), & $0 \le j \le t - 2$ \\
($2t + 1,      4t + 1$) & %\\
($3t,          7t + 1$)    \\
($2t - 1,      6t + 1$) & %\\
($2t,          6t$)       \\
%TabEnd
\multicolumn{2}{c}{$m = 4 t+1$, $m \ge 9$} \\
%TabStart
$(j,           4t - j + 2)$, & $1 \le j \le t$ \\
$(j +  t + 1,  3t - j + 2)$, & $1 \le j \le t - 1$ \\
$(j + 4t + 2,  8t - j + 3)$, & $1 \le j \le t$ \\
$(j + 5t + 3,  7t - j + 3)$, & $1 \le j \le t - 2$ \\
$(2t + 1,      6t + 4)$ & %\\
$(2t + 2,      6t + 3)$    \\
$(4t + 2,      6t + 2)$ & %\\
$( t + 1,      5t + 3)$    \\
%TabEnd
\multicolumn{2}{c}{$m = 4 t+10$, $m \ge 10$} \\
%TabStart
$(2t - j + 1,   2t + j + 5)$,  & $0 \le j \le t$ \\
$( t - j,       3t + j + 7)$,  & $0 \le j \le t - 1$ \\
$(6t - j + 10,  6t + j + 15)$, & $0 \le j \le t - 1$ \\
$(5t - j + 10,  7t + j + 16)$, & $0 \le j \le t$ \\
$(3t + 6,       7t + 15)$ & %\\
$(4t + 7,       8t + 19)$    \\
$(2t + 4,       6t + 11)$ & %\\
$(4t + 9,       8t + 17)$    \\
$(4t + 8,       6t + 13)$ & %\\
$(2t + 2,       6t + 12)$    \\
$(2t + 3,       6t + 14)$ & %\\
$(8t + 18,      8t + 21)$    \\
%TabEnd
\multicolumn{2}{c}{$m = 4 t+7$, $m \ge 7$} \\
%TabStart
$(2t - j + 1,  j+2 t+4)$,  & $0 \le j \le t$ \\
$(t - j,       j+3 t+6)$,  & $0 \le j \le t - 1$ \\
$(6t - j + 8,  j+6 t+12)$, & $0 \le j \le t - 1$ \\
$(5t - j + 8,  j+7 t+13)$, & $0 \le j \le t$ \\
$(3t + 5,      7t + 12)$ & %\\
$(4t + 6,      8t + 15)$    \\
$(2t + 3,      6t +  9)$ & %\\
$(4t + 7,      6t + 11)$    \\
$(2t + 2,      6t + 10)$   %\\
%TabEnd
\end{tabular}

\vskip 2mm
\caption{Lemma~\ref{lem:Triples}: The set $L$}
\label{tab:Triples}
\end{table}

\begin{theorem}
\label{thm:PENT-3-connected}
If $m \ge 5$, there exists a pentagonal geometry $\adfPENT(3,6m + 3)$ with connected deficiency graph.
\end{theorem}
\begin{proof}
Assume $m \ge 5$, let $Q = Z_{6m + 5}$ and let $q = |Q| = 6m + 5$.
We wish to construct a $\adfPENT(3, 6m + 3)$.
The number of points is $v = 12m + 10$ and the number of blocks is $b = (4m + 2)q$.

Let $V = Q \times Z_2$ and write $x_j$ for element $(x, j)$ of $V$.
We construct $4m+2$ base blocks for a $\adfPENT(3, 6m + 3)$ on point set $V$ with
automorphism
\begin{equation}\label{eqn:Automorphism}
x_j \mapsto (x + 1)_j,~~ x \in Q,~~ j = 0, 1,
\end{equation}
as follows.

The base blocks for the opposite lines are
$$\{0_1, 2_0, (q-2)_0\} = (0_0)^{\mathrm{opp}} \textrm{~~and~~} \{0_0, 1_1, (q-1)_1\} = (0_1)^{\mathrm{opp}}.$$
The corresponding deficiency graph is a generalized Petersen graph which consists of two $q$-cycles,
$(0_1$, $1_1$, \dots, $(q-1)_1)$ and $(0_0$, $2_0$, \dots, $(q-1)_0$, $1_0$, $3_0$, \dots, $(q-2)_0)$, linked by edges
$\{0_0, 0_1\}$, $\{1_0, 1_1\}$, \dots, $\{(q-1)_0, (q-1)_1\}.$ The graph has girth 5.

Another $3m$ base blocks are given by
$$ \{\{0_0, i_0, (i/2)_1\}: i = 1, 3, 5, 6, 7, \dots, 3m + 2\}.$$
Note that $i = 2$, $4$ are omitted and that $(i/2)_1$ is well defined since $q$ is odd.
For the final $m$ base blocks, we have
$$\{\{0_1, x_1, z_1\}: (0, x, z) \in T\},$$
where $T$ is the set of Lemma~\ref{lem:Triples}.

To see why the construction works, we define the {\em difference} of a pair
$\{x_i, y_j\}$, $x_i \neq y_j$, $0 \le x \le y < q$, $i, j \in \{0,1\}$, as
$$\begin{array}{ll}
(y - x)_{j,i}       & \textrm{~if~} 0 < y - x < q/2, \\
(q - (y - x))_{i,j} & \textrm{~if~} y - x > q/2, \\
0_{0,1} = 0_{1,0}    & \textrm{~if~} x = y.
\end{array}$$
The differences generated by a triple are those arising from the three pairs contained in it.
The difference of a pair is preserved under the automorphism (\ref{eqn:Automorphism}), and
the set of differences generated by the lines of the geometry together with the edges of the deficiency graph is given by
$$\{0_{0,1}\} \cup \{d_{i,j}: d = 1, 2, \dots, 3m+2,~~ i, j = 0, 1\}.$$
Now observe that the differences corresponding to the opposite lines are
\begin{equation}\label{eqn:DiffLines}
2_{0,1}, 2_{1,0}, 4_{0,0}, 1_{1,0}, 1_{0,1}, 2_{1,1},
\end{equation}
and the differences of $\{0_0, 0_1\}$, $\{0_1, 1_1\}$ and $\{0_0, 2_0\}$,
edges that generate the deficiency graph under the automorphism (\ref{eqn:Automorphism}), are
\begin{equation}\label{eqn:DiffDef}
0_{0,1}, 1_{1,1}, 2_{0,0}.
\end{equation}
As $i$ runs through the set $I = \{1, 2, \dots, 3m + 2\}$ the triples $\{0_0, i_0, (i/2)_1\}$
generate the differences
$$i_{0,0}, (i/2)_{0,1}, (i/2)_{1,0},~ i \in I, \textrm{~~equivalently~~} i_{0,0}, i_{0,1}, i_{1,0},~ i \in I.$$
Moreover, the two missing blocks $\{0_0, 2_0, 1_1\}$ and $\{0_0, 4_0, 2_1\}$  together with
$0_{0,1}$, $1_{1,1}$ and $2_{1,1}$,
account for the differences (\ref{eqn:DiffLines}) and (\ref{eqn:DiffDef}).
The remaining differences, $i_{1,1}$, $i = 3, 4, \dots, 3m + 2$, are provided by the triples of $T$.
Note that $(3m+3)_{1,1} = (3m+2)_{1,1}$.
\end{proof}

%%%%%%%%%%%%%%%%%%%%%%%%%%%%%%%%%%%%%%%%%%%%%%%%%%%%%%%%%%%%%%%%%%%%%%%%%%%%%%%%%%%%%%%%%%
%%%%%%%%%%%%%%%%%%%%%%%%%%%%%%%%%%%%%%%%%%%%%%%%%%%%%%%%%%%%%%%%%%%%%%%%%%%%%%%%%%%%%%%%%%
%%%%%%%%%%%%%%%%%%%%%%%%%%%%%%%%%%%%%%%%%%%%%%%%%%%%%%%%%%%%%%%%%%%%%%%%%%%%%%%%%%%%%%%%%%
%%%%%%%%%%%%%%%%%%%%%%%%%%%%%%%%%%%%%%%%%%%%%%%%%%%%%%%%%%%%%%%%%%%%%%%%%%%%%%%%%%%%%%%%%%

\section{Block size 4}\label{sec:Block size 4}
By Lemma~\ref{lem:v, b}, a pentagonal geometry with block size 4, $\adfPENT(4,r)$, has $3r + 5$ points,
and a necessary existence condition is that $r \equiv 0 \textrm{~or~} 1 \adfmod{4}$.
Using the basic construction of Theorem~\ref{thm:GDD-basic} with the degenerate $\adfPENT(4,1)$ and
4-GDDs of type $8^{3t+1}$ it is proved in \cite[Theorem 11]{GriggsStokes2016} that
there exist pentagonal geometries $\adfPENT(4,r)$ for all $r \equiv 1 \adfmod{8}$.
With the availability of new pentagonal geometries with block size 4,
Lemmas~\ref{lem:PENT-4-direct-basic} and \ref{lem:PENT-4-direct-extra}, further progress has been possible.
In \cite{ForbesGriggsStokes2020} the existence spectrum problem for $\adfPENT(4,r)$ is solved
completely when $r \equiv 1 \adfmod{4}$ and with a small number of possible exceptions when $r \equiv 0 \adfmod{4}$.
However, all of the pentagonal geometries constructed in \cite{GriggsStokes2016} and \cite{ForbesGriggsStokes2020}
contain opposite line pairs.
In this section we show that pentagonal geometries $\adfPENT(4,r)$ without opposite line pairs exist
for all admissible $r \ge 13$ with a few possible exceptions.

In Lemmas~\ref{lem:PENT-4-direct-basic}, \ref{lem:PENT-4-direct-extra} and \ref{lem:PENT-5-r-direct}
we present specific pentagonal geometries with block sizes 4 and 5.
For each $\adfPENT(k,r)$, the point set is $Z_v = \{0, 1, \dots, v - 1\}$, where $v = r(k - 1) + k + 1$, and
the lines are generated from a set of base blocks by a mapping of the form $x \mapsto x + d \adfmod{v}$.
The first $d$ base blocks are $0^{\mathrm{opp}}$, $1^{\mathrm{opp}}$, \dots, $(d - 1)^{\mathrm{opp}}$.
%%%%%%%%%%%%%%%%%%%%%%%%%%%%%%%%%%%%%%%%%%%%%%%%%%%%%%%%%%%%%%%%%%%%%%%%%%%%%%%%%%%%%%%%%%
%%%%%%%%%%%%%%%%%%%%%%%%%%%%%%%%%%%%%%%%%%%%%%%%%%%%%%%%%%%%%%%%%%%%%%%%%%%%%%%%%%%%%%%%%%
\begin{lemma}
\label{lem:PENT-4-direct-basic}
There exist pentagonal geometries $\adfPENT(4,r)$ with no opposite line pairs for
$r \in $ \{$13$, $17$, $20$, $21$, $24$\}.
In each case the deficiency graph is connected and has girth at least $5$.
\end{lemma}
\begin{proof}

{\noindent\boldmath $\adfPENT(4, 13)$}~ With point set $Z_{44}$, the 143 lines are generated from

{\normalsize
%adfPENTstart:1:4:13:4:4:
 $\{3, 38, 41, 42\}$, $\{4, 27, 31, 34\}$, $\{4, 8, 13, 23\}$, $\{0, 17, 21, 26\}$,\adfsplit
 $\{0, 6, 7, 16\}$, $\{1, 14, 22, 38\}$, $\{1, 18, 30, 43\}$, $\{0, 8, 18, 37\}$,\adfsplit
 $\{0, 2, 31, 32\}$, $\{0, 11, 22, 39\}$, $\{0, 1, 20, 33\}$, $\{1, 3, 9, 29\}$,\adfsplit
 $\{1, 15, 23, 35\}$
%adfPENTend
}

\noindent under the action of the mapping $x \mapsto x + 4 \adfmod{44}$.
The deficiency graph is connected and has girth 6.
% The opposite lines are generated by the first 4 blocks: True
% End of PENT(4, 13)
%%%%%%%%%%%%%%%%%%%%%%%%%%%%%%%%%%%%%%%%%%%%%%%%%%%%%%%%%%%%%%%%%%%%%%%%%%%%%%%%%%%%%%%%%%

{\noindent\boldmath $\adfPENT(4, 17)$}~ With point set $Z_{56}$, the 238 lines are generated from

{\normalsize
%adfPENTstart:1:4:17:8:8:
 $\{6, 39, 52, 55\}$, $\{11, 28, 39, 42\}$, $\{4, 17, 19, 30\}$, $\{19, 42, 43, 49\}$,\adfsplit
 $\{2, 8, 13, 33\}$, $\{13, 23, 52, 53\}$, $\{0, 14, 34, 54\}$, $\{8, 24, 25, 45\}$,\adfsplit
 $\{5, 20, 2, 11\}$, $\{0, 2, 28, 36\}$, $\{0, 7, 12, 44\}$, $\{0, 4, 9, 20\}$,\adfsplit
 $\{1, 5, 36, 43\}$, $\{3, 4, 22, 46\}$, $\{3, 15, 29, 36\}$, $\{2, 6, 12, 18\}$,\adfsplit
 $\{1, 20, 22, 53\}$, $\{1, 4, 35, 38\}$, $\{2, 21, 53, 54\}$, $\{3, 39, 45, 54\}$,\adfsplit
 $\{0, 19, 26, 53\}$, $\{0, 22, 29, 31\}$, $\{1, 13, 25, 54\}$, $\{2, 10, 23, 45\}$,\adfsplit
 $\{0, 35, 43, 45\}$, $\{0, 13, 30, 41\}$, $\{0, 3, 38, 48\}$, $\{1, 6, 7, 31\}$,\adfsplit
 $\{2, 14, 31, 43\}$, $\{0, 15, 24, 51\}$, $\{1, 2, 47, 49\}$, $\{0, 18, 23, 42\}$,\adfsplit
 $\{1, 15, 19, 23\}$, $\{0, 10, 33, 49\}$
%adfPENTend
}

\noindent under the action of the mapping $x \mapsto x + 8 \adfmod{56}$.
The deficiency graph is connected and has girth 6.
% The opposite lines are generated by the first 8 blocks: True
% End of PENT(4, 17)
%%%%%%%%%%%%%%%%%%%%%%%%%%%%%%%%%%%%%%%%%%%%%%%%%%%%%%%%%%%%%%%%%%%%%%%%%%%%%%%%%%%%%%%%%%

{\noindent\boldmath $\adfPENT(4, 20)$}~ With point set $Z_{65}$, the 325 lines are generated from

{\normalsize
%adfPENTstart:1:4:20:5:5:
 $\{11, 20, 45, 48\}$, $\{12, 44, 48, 55\}$, $\{13, 44, 56, 63\}$, $\{7, 20, 21, 57\}$,\adfsplit
 $\{26, 27, 29, 44\}$, $\{44, 20, 33, 32\}$, $\{30, 0, 34, 29\}$, $\{0, 2, 9, 60\}$,\adfsplit
 $\{0, 6, 8, 50\}$, $\{0, 10, 26, 42\}$, $\{0, 18, 39, 59\}$, $\{0, 36, 41, 53\}$,\adfsplit
 $\{0, 43, 44, 46\}$, $\{0, 17, 38, 62\}$, $\{0, 47, 57, 63\}$, $\{0, 19, 33, 49\}$,\adfsplit
 $\{0, 27, 51, 61\}$, $\{3, 9, 38, 48\}$, $\{2, 7, 33, 58\}$, $\{1, 26, 53, 58\}$,\adfsplit
 $\{1, 34, 43, 62\}$, $\{1, 14, 24, 27\}$, $\{1, 22, 49, 57\}$, $\{1, 7, 29, 47\}$,\adfsplit
 $\{1, 9, 16, 36\}$
%adfPENTend
}

\noindent under the action of the mapping $x \mapsto x + 5 \adfmod{65}$.
The deficiency graph is connected and has girth 6.
% The opposite lines are generated by the first 5 blocks: True
% End of PENT(4, 20)
%%%%%%%%%%%%%%%%%%%%%%%%%%%%%%%%%%%%%%%%%%%%%%%%%%%%%%%%%%%%%%%%%%%%%%%%%%%%%%%%%%%%%%%%%%

{\noindent\boldmath $\adfPENT(4, 21)$}~ With point set $Z_{68}$, the 357 lines are generated from

{\normalsize
%adfPENTstart:1:4:21:4:4:
 $\{13, 17, 19, 37\}$, $\{32, 43, 52, 56\}$, $\{26, 35, 46, 51\}$, $\{22, 29, 38, 52\}$,\adfsplit
 $\{13, 32, 50, 62\}$, $\{14, 13, 53, 43\}$, $\{57, 60, 15, 0\}$, $\{16, 26, 22, 44\}$,\adfsplit
 $\{0, 1, 9, 16\}$, $\{0, 29, 34, 42\}$, $\{0, 12, 39, 47\}$, $\{0, 7, 14, 51\}$,\adfsplit
 $\{0, 22, 32, 63\}$, $\{0, 5, 26, 62\}$, $\{0, 25, 41, 66\}$, $\{0, 2, 3, 67\}$,\adfsplit
 $\{0, 21, 33, 43\}$, $\{1, 19, 33, 66\}$, $\{1, 39, 54, 67\}$, $\{1, 15, 35, 47\}$,\adfsplit
 $\{1, 18, 46, 63\}$
%adfPENTend
}

\noindent under the action of the mapping $x \mapsto x + 4 \adfmod{68}$.
The deficiency graph is connected and has girth 6.
% The opposite lines are generated by the first 4 blocks: True
% End of PENT(4, 21)
%%%%%%%%%%%%%%%%%%%%%%%%%%%%%%%%%%%%%%%%%%%%%%%%%%%%%%%%%%%%%%%%%%%%%%%%%%%%%%%%%%%%%%%%%%

{\noindent\boldmath $\adfPENT(4, 24)$}~ With point set $Z_{77}$, the 462 lines are generated from

{\normalsize
%adfPENTstart:1:4:24:7:7:
 $\{4, 10, 55, 59\}$, $\{23, 27, 30, 46\}$, $\{23, 50, 57, 58\}$, $\{19, 21, 34, 70\}$,\adfsplit
 $\{0, 36, 39, 46\}$, $\{40, 41, 47, 66\}$, $\{28, 47, 52, 57\}$, $\{47, 25, 55, 71\}$,\adfsplit
 $\{54, 34, 51, 2\}$, $\{44, 17, 8, 28\}$, $\{76, 62, 39, 33\}$, $\{46, 55, 34, 44\}$,\adfsplit
 $\{46, 13, 62, 70\}$, $\{0, 43, 54, 58\}$, $\{48, 71, 50, 75\}$, $\{75, 43, 29, 26\}$,\adfsplit
 $\{58, 20, 18, 59\}$, $\{0, 1, 6, 48\}$, $\{0, 9, 27, 34\}$, $\{1, 3, 17, 20\}$,\adfsplit
 $\{0, 8, 14, 76\}$, $\{1, 13, 18, 31\}$, $\{1, 24, 34, 45\}$, $\{2, 31, 55, 68\}$,\adfsplit
 $\{1, 19, 41, 72\}$, $\{0, 15, 50, 74\}$, $\{1, 29, 66, 67\}$, $\{1, 40, 53, 74\}$,\adfsplit
 $\{0, 2, 22, 35\}$, $\{1, 52, 59, 65\}$, $\{0, 25, 37, 73\}$, $\{2, 10, 45, 53\}$,\adfsplit
 $\{0, 11, 31, 40\}$, $\{0, 7, 30, 52\}$, $\{0, 17, 47, 67\}$, $\{0, 5, 38, 61\}$,\adfsplit
 $\{0, 3, 18, 65\}$, $\{0, 12, 51, 56\}$, $\{2, 11, 12, 74\}$, $\{2, 47, 59, 61\}$,\adfsplit
 $\{0, 32, 60, 68\}$, $\{2, 16, 33, 60\}$
%adfPENTend
}

\noindent under the action of the mapping $x \mapsto x + 7 \adfmod{77}$.
The deficiency graph is connected and has girth 6.
% The opposite lines are generated by the first 7 blocks: True
% End of PENT(4, 24)
%%%%%%%%%%%%%%%%%%%%%%%%%%%%%%%%%%%%%%%%%%%%%%%%%%%%%%%%%%%%%%%%%%%%%%%%%%%%%%%%%%%%%%%%%%
\end{proof}
Note, by the way, that the $\adfPENT(4,13)$ provides the missing ingredient for Constructions~1 and 2 of \cite{GriggsStokes2016}.

%%%%%%%%%%%%%%%%%%%%%%%%%%%%%%%%%%%%%%%%%%%%%%%%%%%%%%%%%%%%%%%%%%%%%%%%%%%%%%%%%%%%%%%%%%
%%%%%%%%%%%%%%%%%%%%%%%%%%%%%%%%%%%%%%%%%%%%%%%%%%%%%%%%%%%%%%%%%%%%%%%%%%%%%%%%%%%%%%%%%%
\begin{lemma}
\label{lem:PENT-4-direct-extra}
There exist pentagonal geometries $\adfPENT(4,r)$ with no opposite line pairs for
$r \in $
\{$29$, $33$, $37$, $40$, $45$, $49$, $52$, $53$, $60$, $61$, $65$, $69$, $77$, $80$, $81$, $85$, $93$, $97$, $100$, $101$, $108$, $109$, $117$, $120$, $125$, $133$, $140$, $141$, $149$, $157$, $160$, $165$, $173$, $180$\}.
In each case the deficiency graph is connected and has girth at least $5$.
\end{lemma}
\begin{proof}
The pentagonal geometries are presented in Appendix~\ref{app:PENT-4-direct-extra}.
\end{proof}

In the following we repeatedly use Theorem~\ref{thm:GDD-basic} to create the $\adfPENT(4,r)$ systems
for $r$ in each of the 22 admissible residue classes modulo 44 from a $\adfPENT(4,13)$ and $\adfPENT(4,s)$ systems for small $s$.
We prove two theorems concerning pentagonal geometries {\em without opposite line pairs}.
The first, Theorem~\ref{thm:PENT-4-0OLP-constructed-weak},
uses only a handful of $\adfPENT(4,s)$ systems, those of Lemma~\ref{lem:PENT-4-direct-basic}.
However, if we use in addition the $\adfPENT(4,s)$ systems of Lemma~\ref{lem:PENT-4-direct-extra}, we
obtain a stronger result, Theorem~\ref{thm:PENT-4-0OLP-constructed}.
For the existence of the 4-GDDs of type $g^u m^1$, used throughout the proofs, we refer to
\cite[Theorem 1.2]{Forbes2019B} when $m \neq g$, or
\cite{BrouwerSchrijverHanani1977} when $m = g$.
%%%%%%%%%%%%%%%%%%%%%%%%%%%%%%%%%%%%%%%%%%%%%%%%%%%%%%%%%%%%%%%%%%%%%%%%%%%%%%%%%%%%%%%%%%
%%%%%%%%%%%%%%%%%%%%%%%%%%%%%%%%%%%%%%%%%%%%%%%%%%%%%%%%%%%%%%%%%%%%%%%%%%%%%%%%%%%%%%%%%%
\begin{theorem}
\label{thm:PENT-4-0OLP-constructed-weak}
There exist pentagonal geometries $\adfPENT(4,r)$ with no opposite line pairs
for all positive integers $r > 920$, $r \equiv 0 \textrm{~or~} 1 \adfmod{4}$.
\end{theorem}
\begin{proof}
The constructions use Theorem \ref{thm:GDD-basic} with
$3w$ copies of the $\adfPENT(4,13)$ of Lemma~\ref{lem:PENT-4-direct-basic},
a $\adfPENT(4, s)$ without an opposite line pair and
a 4-GDD of type $44^{3w} (3s+5)^1$.
If the $\adfPENT(4, s)$ is not included in Lemma~\ref{lem:PENT-4-direct-basic},
we construct it separately.
The details are presented in Table~\ref{tab:PENT-4-0OLP-constructed-weak-main}.
%
% PUT ALL THE TABLES HERE SO THEY DO NOT TRESPASS INTO THE NEXT SECTION
%
{\small
\begin{table}[h]
\begin{tabular}{cccccc}
$r$ & & & & & \\
mod 44 & $s$ & 4-GDD & $w$ & $\adfPENT(4,r)$ & $t$ \\
\hline
0 & 132 & $ 44^{3w} 401^1 $ & $ w \ge 7 $ & $\adfPENT(4,44t) $ \rule{0mm}{4mm} & $t \ge 10 $ \\
1 & 133 & $ 44^{3w} 404^1 $ & $ w \ge 7 $ & $\adfPENT(4,44t + 1) $ & $t \ge 10 $ \\
4 & 136 & $ 44^{3w} 413^1 $ & $ w \ge 7 $ & $\adfPENT(4,44t + 4) $ & $t \ge 10 $ \\
5 & 181 & $ 44^{3w} 548^1 $ & $ w \ge 9 $ & $\adfPENT(4,44t + 5) $ & $t \ge 13 $ \\
8 & 228 & $ 44^{3w} 689^1 $ & $ w \ge 11 $ & $\adfPENT(4,44t + 8) $ & $t \ge 16 $ \\
9 & 185 & $ 44^{3w} 560^1 $ & $ w \ge 9 $ & $\adfPENT(4,44t + 9) $ & $t \ge 13 $ \\
12 & 188 & $ 44^{3w} 569^1 $ & $ w \ge 9 $ & $\adfPENT(4,44t + 12) $ & $t \ge 13 $ \\
13 & 13 & $ 44^{3w} 44^1 $ & $ w \ge 1 $ & $\adfPENT(4,44t + 13) $ & $t \ge 1 $ \\
16 & 192 & $ 44^{3w} 581^1 $ & $ w \ge 10 $ & $\adfPENT(4,44t + 16) $ & $t \ge 14 $ \\
17 & 17 & $ 44^{3w} 56^1 $ & $ w \ge 2 $ & $\adfPENT(4,44t + 17) $ & $t \ge 2 $ \\
20 & 20 & $ 44^{3w} 65^1 $ & $ w \ge 2 $ & $\adfPENT(4,44t + 20) $ & $t \ge 2 $ \\
21 & 21 & $ 44^{3w} 68^1 $ & $ w \ge 2 $ & $\adfPENT(4,44t + 21) $ & $t \ge 2 $ \\
24 & 24 & $ 44^{3w} 77^1 $ & $ w \ge 2 $ & $\adfPENT(4,44t + 24) $ & $t \ge 2 $ \\
25 & 157 & $ 44^{3w} 476^1 $ & $ w \ge 8 $ & $\adfPENT(4,44t + 25) $ & $t \ge 11 $ \\
28 & 160 & $ 44^{3w} 485^1 $ & $ w \ge 8 $ & $\adfPENT(4,44t + 28) $ & $t \ge 11 $ \\
29 & 293 & $ 44^{3w} 884^1 $ & $ w \ge 14 $ & $\adfPENT(4,44t + 29) $ & $t \ge 20 $ \\
32 & 296 & $ 44^{3w} 893^1 $ & $ w \ge 14 $ & $\adfPENT(4,44t + 32) $ & $t \ge 20 $ \\
33 & 297 & $ 44^{3w} 896^1 $ & $ w \ge 14 $ & $\adfPENT(4,44t + 33) $ & $t \ge 20 $ \\
36 & 300 & $ 44^{3w} 905^1 $ & $ w \ge 15 $ & $\adfPENT(4,44t + 36) $ & $t \ge 21 $ \\
37 & 125 & $ 44^{3w} 380^1 $ & $ w \ge 7 $ & $\adfPENT(4,44t + 37) $ & $t \ge 9 $ \\
40 & 304 & $ 44^{3w} 917^1 $ & $ w \ge 15 $ & $\adfPENT(4,44t + 40) $ & $t \ge 21 $ \\
41 & 129 & $ 44^{3w} 392^1 $ & $ w \ge 7 $ & $\adfPENT(4,44t + 41) $ & $t \ge 9 $ \\
\end{tabular}

\vskip 2mm
\caption{Theorem~\ref{thm:PENT-4-0OLP-constructed-weak}: constructions}
\label{tab:PENT-4-0OLP-constructed-weak-main}
\end{table}
\begin{table}[h]

\begin{tabular}{ccccc|ccccc}
$r$ & $u$ & $p$ & $q$ & 4-GDD & $r$ & $u$ & $p$ & $q$ & 4-GDD \\
\hline
132 & 6 & 17 & 20 & $56^{6} 65^1$   &  133 & 6 & 17 & 21 & $56^{6} 68^1$ \rule{0mm}{4mm}\\
136 & 6 & 17 & 24 & $56^{6} 77^1$   &  181 & 9 & 17 & 13 & $56^{9} 44^1$ \\
228 & 9 & 21 & 24 & $68^{9} 77^1$   &  185 & 9 & 17 & 17 & $56^{9} 56^1$ \\
188 & 9 & 17 & 20 & $56^{9} 65^1$   &  192 & 9 & 17 & 24 & $56^{9} 77^1$ \\
157 & 6 & 21 & 21 & $68^{6} 68^1$   &  160 & 6 & 21 & 24 & $68^{6} 77^1$ \\
293 & 15 & 17 & 13 & $56^{15} 44^1$   &  296 & 12 & 21 & 24 & $68^{12} 77^1$ \\
297 & 15 & 17 & 17 & $56^{15} 56^1$   &  300 & 15 & 17 & 20 & $56^{15} 65^1$ \\
125 & 6 & 17 & 13 & $56^{6} 44^1$   &  304 & 15 & 17 & 24 & $56^{15} 77^1$ \\
129 & 6 & 17 & 17 & $56^{6} 56^1$  & & & & & \\
\end{tabular}

\vskip 2mm
\caption{Theorem~\ref{thm:PENT-4-0OLP-constructed-weak}: $\adfPENT(4,r)$ required for the constructions}
\label{tab:PENT-4-0OLP-constructed-weak-required}
\end{table}
\begin{table}[h]

\begin{tabular}{cccccc}
$r$ & & & & & \\
mod 44 & $s$ & 4-GDD & $w$ & $\adfPENT(4,r)$ & $t$ \\
\hline
0 & 132 & $ 44^{3w} 401^1 $ & $ w \ge 7 $ & $\adfPENT(4,44t) $ \rule{0mm}{4mm} & $t \ge 10 $ \\
1 & 45 & $ 44^{3w} 140^1 $ & $ w \ge 3 $ & $\adfPENT(4,44t + 1) $ & $t \ge 4 $ \\
4 & 136 & $ 44^{3w} 413^1 $ & $ w \ge 7 $ & $\adfPENT(4,44t + 4) $ & $t \ge 10 $ \\
5 & 49 & $ 44^{3w} 152^1 $ & $ w \ge 3 $ & $\adfPENT(4,44t + 5) $ & $t \ge 4 $ \\
8 & 52 & $ 44^{3w} 161^1 $ & $ w \ge 3 $ & $\adfPENT(4,44t + 8) $ & $t \ge 4 $ \\
9 & 53 & $ 44^{3w} 164^1 $ & $ w \ge 3 $ & $\adfPENT(4,44t + 9) $ & $t \ge 4 $ \\
12 & 100 & $ 44^{3w} 305^1 $ & $ w \ge 5 $ & $\adfPENT(4,44t + 12) $ & $t \ge 7 $ \\
13 & 13 & $ 44^{3w} 44^1 $ & $ w \ge 1 $ & $\adfPENT(4,44t + 13) $ & $t \ge 1 $ \\
16 & 60 & $ 44^{3w} 185^1 $ & $ w \ge 4 $ & $\adfPENT(4,44t + 16) $ & $t \ge 5 $ \\
17 & 17 & $ 44^{3w} 56^1 $ & $ w \ge 2 $ & $\adfPENT(4,44t + 17) $ & $t \ge 2 $ \\
20 & 20 & $ 44^{3w} 65^1 $ & $ w \ge 2 $ & $\adfPENT(4,44t + 20) $ & $t \ge 2 $ \\
21 & 21 & $ 44^{3w} 68^1 $ & $ w \ge 2 $ & $\adfPENT(4,44t + 21) $ & $t \ge 2 $ \\
24 & 24 & $ 44^{3w} 77^1 $ & $ w \ge 2 $ & $\adfPENT(4,44t + 24) $ & $t \ge 2 $ \\
25 & 69 & $ 44^{3w} 212^1 $ & $ w \ge 4 $ & $\adfPENT(4,44t + 25) $ & $t \ge 5 $ \\
28 & 160 & $ 44^{3w} 485^1 $ & $ w \ge 8 $ & $\adfPENT(4,44t + 28) $ & $t \ge 11 $ \\
29 & 29 & $ 44^{3w} 92^1 $ & $ w \ge 2 $ & $\adfPENT(4,44t + 29) $ & $t \ge 2 $ \\
32 & 120 & $ 44^{3w} 365^1 $ & $ w \ge 6 $ & $\adfPENT(4,44t + 32) $ & $t \ge 8 $ \\
33 & 33 & $ 44^{3w} 104^1 $ & $ w \ge 2 $ & $\adfPENT(4,44t + 33) $ & $t \ge 2 $ \\
36 & 80 & $ 44^{3w} 245^1 $ & $ w \ge 5 $ & $\adfPENT(4,44t + 36) $ & $t \ge 6 $ \\
37 & 37 & $ 44^{3w} 116^1 $ & $ w \ge 3 $ & $\adfPENT(4,44t + 37) $ & $t \ge 3 $ \\
40 & 40 & $ 44^{3w} 125^1 $ & $ w \ge 3 $ & $\adfPENT(4,44t + 40) $ & $t \ge 3 $ \\
41 & 85 & $ 44^{3w} 260^1 $ & $ w \ge 5 $ & $\adfPENT(4,44t + 41) $ & $t \ge 6 $ \\
\end{tabular}

\vskip 2mm
\caption{Theorem~\ref{thm:PENT-4-0OLP-constructed}: constructions}
\label{tab:PENT-4-0OLP-constructed-main}
\end{table}
\begin{table}[h]

\begin{tabular}{ccccc|ccccc}
$r$ & $u$ & $p$ & $q$ & 4-GDD & $r$ & $u$ & $p$ & $q$ & 4-GDD\\
\hline
129 & 6 & 17 & 17 & $56^{6} 56^1$  & 176 & 6 & 21 & 40 & $68^{6} 125^1$ \rule{0mm}{4mm}\\
188 & 9 & 17 & 20 & $56^{9} 65^1$  & 192 & 9 & 17 & 24 & $56^{9} 77^1$ \\
201 & 6 & 29 & 17 & $92^{6} 56^1$  & 204 & 6 & 29 & 20 & $92^{6} 65^1$ \\
208 & 6 & 29 & 24 & $92^{6} 77^1$  & 217 & 9 & 21 & 13 & $68^{9} 44^1$ \\
220 & 9 & 17 & 52 & $56^{9} 161^1$  & 224 & 9 & 21 & 20 & $68^{9} 65^1$ \\
232 & 6 & 33 & 24 & $104^{6} 77^1$  & 248 & 12 & 17 & 24 & $56^{12} 77^1$ \\
252 & 6 & 37 & 20 & $116^{6} 65^1$  & 256 & 6 & 37 & 24 & $116^{6} 77^1$ \\
261 & 6 & 37 & 29 & $116^{6} 92^1$  & 264 & 12 & 17 & 40 & $56^{12} 125^1$ \\
268 & 6 & 33 & 60 & $104^{6} 185^1$  & 276 & 12 & 17 & 52 & $56^{12} 161^1$ \\
292 & 12 & 21 & 20 & $68^{12} 65^1$  & 296 & 9 & 29 & 20 & $92^{9} 65^1$ \\
312 & 12 & 21 & 40 & $68^{12} 125^1$  & 336 & 9 & 33 & 24 & $104^{9} 77^1$ \\
340 & 15 & 17 & 60 & $56^{15} 185^1$  & 352 & 9 & 33 & 40 & $104^{9} 125^1$ \\
356 & 18 & 17 & 20 & $56^{18} 65^1$  & 380 & 15 & 21 & 40 & $68^{15} 125^1$ \\
396 & 18 & 17 & 60 & $56^{18} 185^1$  & 400 & 9 & 37 & 52 & $116^{9} 161^1$ \\
424 & 6 & 65 & 24 & $200^{6} 77^1$  & 468 & 24 & 17 & 20 & $56^{24} 65^1$ \\
\end{tabular}

\vskip 2mm
\caption{Theorem~\ref{thm:PENT-4-0OLP-constructed}: missing values}
\label{tab:PENT-4-0OLP-constructed-missing}

\end{table}
\begin{table}[h]

\begin{tabular}{cccccc}
$r$ & & & & & \\
mod 44 & $s$ & 4-GDD & $w$ & $\adfPENT(4,r)$ & $t$ \\
\hline
0 & 1496 & $ 44^{3w} 4493^1 $ & $ w \ge 69 $ & $\adfPENT(4,44t) $ \rule{0mm}{4mm} & $t \ge 103 $ \\
1 & 1 & $ 44^{3w} 8^1 $ & $ w \ge 2 $ & $\adfPENT(4,44t + 1) $ & $t \ge 2 $ \\
4 & 1236 & $ 44^{3w} 3713^1 $ & $ w \ge 57 $ & $\adfPENT(4,44t + 4) $ & $t \ge 85 $ \\
5 & 137 & $ 44^{3w} 416^1 $ & $ w \ge 7 $ & $\adfPENT(4,44t + 5) $ & $t \ge 10 $ \\
8 & 976 & $ 44^{3w} 2933^1 $ & $ w \ge 45 $ & $\adfPENT(4,44t + 8) $ & $t \ge 67 $ \\
9 & 185 & $ 44^{3w} 560^1 $ & $ w \ge 9 $ & $\adfPENT(4,44t + 9) $ & $t \ge 13 $ \\
12 & 1156 & $ 44^{3w} 3473^1 $ & $ w \ge 53 $ & $\adfPENT(4,44t + 12) $ & $t \ge 79 $ \\
13 & 233 & $ 44^{3w} 704^1 $ & $ w \ge 11 $ & $\adfPENT(4,44t + 13) $ & $t \ge 16 $ \\
16 & 2876 & $ 44^{3w} 8633^1 $ & $ w \ge 132 $ & $\adfPENT(4,44t + 16) $ & $t \ge 197 $ \\
17 & 281 & $ 44^{3w} 848^1 $ & $ w \ge 14 $ & $\adfPENT(4,44t + 17) $ & $t \ge 20 $ \\
20 & 3056 & $ 44^{3w} 9173^1 $ & $ w \ge 140 $ & $\adfPENT(4,44t + 20) $ & $t \ge 209 $ \\
21 & 417 & $ 44^{3w} 1256^1 $ & $ w \ge 20 $ & $\adfPENT(4,44t + 21) $ & $t \ge 29 $ \\
24 & 2796 & $ 44^{3w} 8393^1 $ & $ w \ge 128 $ & $\adfPENT(4,44t + 24) $ & $t \ge 191 $ \\
25 & 113 & $ 44^{3w} 344^1 $ & $ w \ge 6 $ & $\adfPENT(4,44t + 25) $ & $t \ge 8 $ \\
28 & 1876 & $ 44^{3w} 5633^1 $ & $ w \ge 86 $ & $\adfPENT(4,44t + 28) $ & $t \ge 128 $ \\
29 & 205 & $ 44^{3w} 620^1 $ & $ w \ge 10 $ & $\adfPENT(4,44t + 29) $ & $t \ge 14 $ \\
32 & 2276 & $ 44^{3w} 6833^1 $ & $ w \ge 104 $ & $\adfPENT(4,44t + 32) $ & $t \ge 155 $ \\
33 & 209 & $ 44^{3w} 632^1 $ & $ w \ge 10 $ & $\adfPENT(4,44t + 33) $ & $t \ge 14 $ \\
36 & 2016 & $ 44^{3w} 6053^1 $ & $ w \ge 93 $ & $\adfPENT(4,44t + 36) $ & $t \ge 138 $ \\
37 & 169 & $ 44^{3w} 512^1 $ & $ w \ge 9 $ & $\adfPENT(4,44t + 37) $ & $t \ge 12 $ \\
40 & 1756 & $ 44^{3w} 5273^1 $ & $ w \ge 81 $ & $\adfPENT(4,44t + 40) $ & $t \ge 120 $ \\
41 & 305 & $ 44^{3w} 920^1 $ & $ w \ge 15 $ & $\adfPENT(4,44t + 41) $ & $t \ge 21 $ \\
\end{tabular}

\vskip 2mm
\caption{Theorem~\ref{thm:PENT-4-1OLP-constructed-weak}: constructions}
\label{tab:PENT-4-1OLP-constructed-weak-main}
\end{table}
\begin{table}[h]

\begin{tabular}{cccc|cccc}
$s$ & $u$ & $p$ & 4-GDD & $s$ & $u$ & $p$ & 4-GDD \\
\hline
1496 & 69 & 20 & $65^{69} 8^1$   &  1236 & 57 & 20 & $65^{57} 8^1$ \rule{0mm}{4mm}\\
137 & 6 & 21 & $68^{6} 8^1$   &  976 & 45 & 20 & $65^{45} 8^1$ \\
185 & 6 & 29 & $92^{6} 8^1$   &  1156 & 45 & 24 & $77^{45} 8^1$ \\
233 & 6 & 37 & $116^{6} 8^1$   &  2876 & 69 & 40 & $125^{69} 8^1$ \\
281 & 15 & 17 & $56^{15} 8^1$   &  3056 & 141 & 20 & $65^{141} 8^1$ \\
417 & 12 & 33 & $104^{12} 8^1$   &  2796 & 129 & 20 & $65^{129} 8^1$ \\
113 & 6 & 17 & $56^{6} 8^1$   &  1876 & 45 & 40 & $125^{45} 8^1$ \\
205 & 9 & 21 & $68^{9} 8^1$   &  2276 & 105 & 20 & $65^{105} 8^1$ \\
209 & 6 & 33 & $104^{6} 8^1$   &  2016 & 93 & 20 & $65^{93} 8^1$ \\
169 & 9 & 17 & $56^{9} 8^1$   &  1756 & 81 & 20 & $65^{81} 8^1$ \\
305 & 6 & 49 & $152^{6} 8^1$  & & & & \\
\end{tabular}

\vskip 2mm
\caption{Theorem~\ref{thm:PENT-4-1OLP-constructed-weak}: $\adfPENT(4,r)$ required for the constructions}
\label{tab:PENT-4-1OLP-constructed-weak-required}
\end{table}
}
For instance, in the first row of Table~\ref{tab:PENT-4-0OLP-constructed-weak-main}
we are to construct pentagonal geometries for residue class 0 modulo 44, $\adfPENT(4,44t)$.
We take $3w$ copies of a $\adfPENT(4,13)$ (which has 44 points), a $\adfPENT(4,132)$ (401 points),
and a 4-GDD of type $44^{3w} 401^1$, which exists if and only if $w \ge 7$, \cite[Theorem 1.2]{Forbes2019B}.
Therefore, by Theorem~\ref{thm:GDD-basic}, there exists a
$\adfPENT(4,44w + 132)$ for $w \ge 7$;
i.e.\ there exists a $\adfPENT(4,44t)$ for $t \ge 10$.

For the constructions to work, we require $\adfPENT(4,s)$ systems for
those values of $s$ in column 2 of Table~\ref{tab:PENT-4-0OLP-constructed-weak-main}, namely
$s \in $
\{$13$, $17$, $20$, $21$, $24$, $125$, $129$, $132$, $133$, $136$, $157$, $160$, $181$, $185$, $188$, $192$, $228$, $293$, $296$, $297$, $300$, $304$\}.
The $\adfPENT(4,s)$ that are not given by Lemma~\ref{lem:PENT-4-direct-basic} are constructed by Theorem~\ref{thm:GDD-basic} from
$u$ copies of a $\adfPENT(4,p)$, one copy of a $\adfPENT(4,q)$ and a 4-GDD of type $(3p+5)^u (3q+5)^1$,
as presented in Table~\ref{tab:PENT-4-0OLP-constructed-weak-required}.
See Lemma~\ref{lem:PENT-4-direct-basic} for the existence of the $\adfPENT(4,p)$ and the $\adfPENT(4,q)$.

The largest $r$ for which the constructions do not create a $\adfPENT(4,r)$ is $r = 920$
corresponding to $t = 20$ in entry 40 of
Table~\ref{tab:PENT-4-0OLP-constructed-weak-main}.
% Basic=5 Extra=19 Exceptions=3 PossibleExceptions=37 E=22 X=227 Done=168
\end{proof}
%%%%%%%%%%%%%%%%%%%%%%%%%%%%%%%%%%%%%%%%%%%%%%%%%%%%%%%%%%%%%%%%%%%%%%%%%%%%%%%%%%%%%%%%%%
%%%%%%%%%%%%%%%%%%%%%%%%%%%%%%%%%%%%%%%%%%%%%%%%%%%%%%%%%%%%%%%%%%%%%%%%%%%%%%%%%%%%%%%%%%
\begin{theorem}
\label{thm:PENT-4-0OLP-constructed}
There exist pentagonal geometries $\adfPENT(4,r)$ with no opposite line pairs
for all positive integers $r \equiv 0 \textrm{~or~} 1 \adfmod{4}$
except for $r \in $
\{$1$, $4$, $5$\}
and except possibly for
$r \in $
\{$8$, $9$, $12$, $16$, $25$, $28$, $32$, $36$, $41$, $44$, $48$, $56$, $64$, $68$, $72$, $73$, $76$, $84$, $88$, $89$, $92$, $96$, $104$, $113$, $116$, $124$, $128$, $137$, $144$, $148$, $164$, $168$, $212$, $308$\}.
\end{theorem}
\begin{proof}

The only $\adfPENT(4,1)$ consists of just an opposite line pair.
By Lemma~\ref{lem:(k,k), (k,k+1)}, $\adfPENT(4, 4)$ and $\adfPENT(4, 5)$ do not exist.

The main constructions are similar to those of Theorem~\ref{thm:PENT-4-0OLP-constructed-weak}.
For each admissible residue class modulo 44,
we use Theorem \ref{thm:GDD-basic} with $3w$ copies of the $\adfPENT(4,13)$ of Lemma~\ref{lem:PENT-4-direct-basic},
an opposite-line-pair-free $\adfPENT(4, s)$ and
a 4-GDD of type $44^{3w} (3s+5)^1$.
The details are presented in Table~\ref{tab:PENT-4-0OLP-constructed-main}.

For the main constructions, we require $\adfPENT(4,s)$ systems for
those values of $s$ in column 2 of Table~\ref{tab:PENT-4-0OLP-constructed-main}, namely $s \in E$, where
$E =$
\{$13$, $17$, $20$, $21$, $24$, $29$, $33$, $37$, $40$, $45$, $49$, $52$, $53$, $60$, $69$, $80$, $85$, $100$, $120$, $132$, $136$, $160$\}.
The $\adfPENT(4,s)$ for $s \in E$ that are not given by Lemmas~\ref{lem:PENT-4-direct-basic} and \ref{lem:PENT-4-direct-extra},
namely $s = 132$, 136 and 160, are constructed as indicated in
Table~\ref{tab:PENT-4-0OLP-constructed-weak-required} of Theorem~\ref{thm:PENT-4-0OLP-constructed-weak}.

The values of $r = 44 t + \rho$ for which the main constructions do not produce a $\adfPENT(4,44t + \rho)$, i.e.\ where
$r \notin E$ and $t$ is less than the
minimum stated in Table~\ref{tab:PENT-4-0OLP-constructed-main}, are given by $r \in X$ =
\{$1$, $4$, $5$, $8$, $9$, $12$, $16$, $25$, $28$, $32$, $36$, $41$, $44$, $48$, $56$, $61$, $64$, $65$, $68$, $72$, $73$, $76$, $77$, $81$, $84$, $88$, $89$, $92$, $93$, $96$, $97$, $104$, $113$, $116$, $124$, $125$, $128$, $129$, $133$, $137$, $140$, $141$, $144$, $148$, $157$, $164$, $168$, $173$, $176$, $180$, $188$, $192$, $201$, $204$, $208$, $212$, $217$, $220$, $224$, $232$, $248$, $252$, $256$, $261$, $264$, $268$, $276$, $292$, $296$, $308$, $312$, $336$, $340$, $352$, $356$, $380$, $396$, $400$, $424$, $468$\}.

For $r \in$
\{$61$, $65$, $77$, $81$, $93$, $97$, $125$, $133$, $140$, $141$, $157$, $173$, $180$\},
see Lemma~\ref{lem:PENT-4-direct-extra}.
For the remaining $r \in X$ that are not listed as exceptions and possible exceptions in the statement of the theorem,
a $\adfPENT(4,r)$ is constructed by Theorem~\ref{thm:GDD-basic} from
$u$ copies of a $\adfPENT(4, p)$, one copy of a $\adfPENT(4, q)$ and a 4-GDD of type $(3p+5)^u (3q+5)^1$,
as shown in Table~\ref{tab:PENT-4-0OLP-constructed-missing}.
See Lemmas~\ref{lem:PENT-4-direct-basic} and \ref{lem:PENT-4-direct-extra} for the
existence of the $\adfPENT(4,p)$ and the $\adfPENT(4,q)$.
%
%     Basic=19 Extra=19 Exceptions=3 PossibleExceptions=34 E=22 X=80 Done=30
% v3: Basic=19 Extra=5  Exceptions=3 PossibleExceptions=37 E=22 X=80 Done=35
%%%%%%%%%%%%%%%%%%%%%%%%%%%%%%%%%%%%%%%%%%%%%%%%%%%%%%%%%%%%%%%%%%%%%%%%%%%%%%%%%%%%%%%%%%
\end{proof}
%
%%%%%%%%%%%%%%%%%%%%%%%%%%%%%%%%%%%%%%%%%%%%%%%%%%%%%%%%%%%%%%%%%%%%%%%%%%%%%%%%%%%%%%%%%%
%%%%%%%%%%%%%%%%%%%%%%%%%%%%%%%%%%%%%%%%%%%%%%%%%%%%%%%%%%%%%%%%%%%%%%%%%%%%%%%%%%%%%%%%%%
\begin{theorem}
\label{thm:PENT-4-1OLP-constructed-weak}
If $r > 9172$ and $r \equiv 0 \textrm{~or~} 1 \adfmod{4}$,
there exists a pentagonal geometry $\adfPENT(4,r)$ with precisely one opposite line pair.
\end{theorem}
\begin{proof}
The proof is similar to that of Theorem~\ref{thm:PENT-4-0OLP-constructed-weak}.
For each admissible residue class modulo 44,
we use Theorem \ref{thm:GDD-basic} with $3w$ copies of the $\adfPENT(4,13)$ of Lemma~\ref{lem:PENT-4-direct-basic},
a $\adfPENT(4, s)$ with exactly one opposite line pair, and
a 4-GDD of type $44^{3w} (3s+5)^1$, \cite[Theorem 1.2]{Forbes2019B},
as shown in Table~\ref{tab:PENT-4-1OLP-constructed-weak-main}.

For $s \neq 1$, the $\adfPENT(4,s)$ required for the constructions are built from
$u$ copies of a $\adfPENT(4,p)$ from Theorem~\ref{thm:PENT-4-0OLP-constructed},
a $\adfPENT(4,1)$ and a 4-GDD of type $(3p+5)^u 8^1$,
as indicated in Table~\ref{tab:PENT-4-1OLP-constructed-weak-required}.
For the 4-GDDs of type $g^u 8^1$, see \cite[Theorem 1.2]{Forbes2019B} when $g$ is even and \cite[Theorem 7.1]{WeiGe2015} when $g$ is odd.

The largest $r$ for which the constructions fail is $r = 9172$
corresponding to $t = 208$ in entry 20 of Table~\ref{tab:PENT-4-1OLP-constructed-weak-main}.
\end{proof}

Finally in this section we prove a more general theorem.
%%%%%%%%%%%%%%%%%%%%%%%%%%%%%%%%%%%%%%%%%%%%%%%%%%%%%%%%%%%%%%%%%%%%%%%%%%%%%%%%%%%%%%%%%%
%%%%%%%%%%%%%%%%%%%%%%%%%%%%%%%%%%%%%%%%%%%%%%%%%%%%%%%%%%%%%%%%%%%%%%%%%%%%%%%%%%%%%%%%%%
\begin{theorem}
\label{thm:PENT-4-jOLP-constructed-weak}
Given non-negative integer $j$, for all sufficiently large $r \equiv 0 \textrm{~or~} 1 \adfmod{4}$,
there exists a pentagonal geometry $\adfPENT(4,r)$ with precisely $j$ opposite line pairs.
\end{theorem}
\begin{proof}
For each $\rho \in \{0, 4, \dots, 40\} \cup \{1, 5, \dots, 41\}$, the set of admissible residues modulo 44,
perform the following constructions.

Choose integers $a$ and $t$ such that:
\begin{enumerate}
\item[]$t \ge 1$,~~ $12t + 1 \ge j$,~~ $\gcd(12t + 1, 44) = 1$,
\item[]$(12t + 1)a + 20t \equiv \rho \adfmod{44}$,
\item[]there exists a $\adfPENT(4,a)$, $A$, with no opposite line pair,
\item[]there exists a $\adfPENT(4,a)$, $A'$, with exactly one opposite line pair.
\end{enumerate}
For $a$ and $t$ chosen as indicated, there exists a 4-GDD of type $(3a + 5)^{12t+1}$,
\cite{BrouwerSchrijverHanani1977}.
Also $a \equiv \rho \adfmod{4}$ and therefore the existence of the pentagonal geometries $A$ and $A'$
follows from Theorems~\ref{thm:PENT-4-0OLP-constructed} and \ref{thm:PENT-4-1OLP-constructed-weak}
provided $a$ is sufficiently large.
Using Theorem~\ref{thm:GDD-basic} with $12t + 1 - j$ copies of $A$, $j$ copies of $A'$ and the 4-GDD,
create a $\adfPENT(4, s)$ with exactly $j$ opposite line pairs, where
$$s = (12t + 1)a + 20t \equiv \rho \adfmod{44}.$$
Now use Theorem~\ref{thm:GDD-basic} with $3w$ copies of the $\adfPENT(4,13)$ from Lemma~\ref{lem:PENT-4-direct-basic},
one copy of the $\adfPENT(4, s)$,
and a 4-GDD of type $44^{3w} (3s + 5)^1$, \cite[Theorem 1.2]{Forbes2019B},
to create for all sufficiently large $w$ a
$\adfPENT(4, 44 w + s)$, i.e.\ a $\adfPENT(4, 44 (w + \lfloor s/44\rfloor) + \rho)$, with exactly $j$ opposite line pairs.
\end{proof}

%%%%%%%%%%%%%%%%%%%%%%%%%%%%%%%%%%%%%%%%%%%%%%%%%%%%%%%%%%%%%%%%%%%%%%%%%%%%%%%%%%%%%%%%%%
%%%%%%%%%%%%%%%%%%%%%%%%%%%%%%%%%%%%%%%%%%%%%%%%%%%%%%%%%%%%%%%%%%%%%%%%%%%%%%%%%%%%%%%%%%
%%%%%%%%%%%%%%%%%%%%%%%%%%%%%%%%%%%%%%%%%%%%%%%%%%%%%%%%%%%%%%%%%%%%%%%%%%%%%%%%%%%%%%%%%%
%%%%%%%%%%%%%%%%%%%%%%%%%%%%%%%%%%%%%%%%%%%%%%%%%%%%%%%%%%%%%%%%%%%%%%%%%%%%%%%%%%%%%%%%%%

\section{Block size 5}\label{sec:Block size 5}
By Lemma~\ref{lem:v, b}, a pentagonal geometry with block size 5, $\adfPENT(5,r)$, has $4r + 6$ points,
and a necessary existence condition is that $r \equiv 0$ or 1 (mod~5).
We begin by presenting some $5$-GDDs, which we believe to be new.

%%%%%%%%%%%%%%%%%%%%%%%%%%%%%%%%%%%%%%%%%%%%%%%%%%%%%%%%%%%%%%%%%%%%%%%%%%%%%%%%%%%%%%%%%%
%%%%%%%%%%%%%%%%%%%%%%%%%%%%%%%%%%%%%%%%%%%%%%%%%%%%%%%%%%%%%%%%%%%%%%%%%%%%%%%%%%%%%%%%%%
\begin{lemma}
\label{lem:5-GDDs-direct}
There exist $5$-$\mathrm{GDD}$s of types
$ 2^{35} $,
$ 2^{71} $,
$ 10^{23} $,
$ 2^{40} 6^1 $,
$ 4^{40} 12^1 $,
$ 10^{10} 18^1 $ and
$ 20^{10} 36^1 $.
\end{lemma}
\begin{proof}
We write $Z_v$ for the set $\{0, 1,\dots, v - 1\}$.
The expression $a \adfmod{b}$ denotes the integer $n$ that satisfies $0 \le n < b$ and $b|n - a$.
%
%ADFvfySectionStart {5-GDDs}

%%%%%%%%%%%%%%%%%%%%%%%%%%%%%%%%%%%%%%%%%%%%%%%%%%%%%%%%%%%%%%%%%%%%%%%%%%%%%%%%%%%%%%%%%%
%%%%%%%%%%%%%%%%%%%%%%%%%%%%%%%%%%%%%%%%%%%%%%%%%%%%%%%%%%%%%%%%%%%%%%%%%%%%%%%%%%%%%%%%%%
\adfhide{
$ 2^{35} $,
$ 2^{71} $,
$ 10^{23} $,
$ 2^{40} 6^1 $,
$ 4^{40} 12^1 $,
$ 10^{10} 18^1 $ and
$ 20^{10} 36^1 $.
}

% Charlotte:Gen-PENT-4-5-GDD:HITS-fun:5.5
\adfDgap
%ADFvfyBlocksStart {2,2,2,2,2,2,2,2,2,2,2,2,2,2,2,2,2,2,2,2,2,2,2,2,2,2,2,2,2,2,2,2,2,2,2}
\noindent{\boldmath $ 2^{35} $}~
With the point set $Z_{70}$ partitioned into
 residue classes modulo $34$ for $\{0, 1, \dots, 67\}$, and
 $\{68, 69\}$,
 the design is generated from

\adfLgap {\adfBfont
%ADFvfyDesignStart
$\{45, 4, 46, 61, 7\}$,
$\{53, 36, 21, 47, 59\}$,
$\{14, 12, 68, 33, 7\}$,\adfsplit
$\{0, 1, 10, 67, 69\}$,
$\{1, 5, 13, 38, 59\}$,
$\{0, 9, 30, 33, 38\}$,\adfsplit
$\{0, 5, 45, 48, 62\}$,
$\{0, 13, 32, 54, 61\}$,
$\{0, 37, 39, 55, 59\}$,\adfsplit
$\{1, 11, 14, 46, 51\}$,
$\{0, 6, 18, 31, 58\}$,
$\{0, 26, 27, 46, 50\}$,\adfsplit
$\{0, 7, 15, 51, 66\}$,
$\{0, 4, 12, 28, 47\}$

}
%ADFvfyBlocksEnd
\adfLgap \noindent by the mapping:
$x \mapsto x + 4 j \adfmod{68}$ for $x < 68$,
$x \mapsto x$ for $x \ge 68$,
$0 \le j < 17$.
\ADFvfyParStart{(70, ((14, 17, ((68, 4), (2, 2)))), ((2, 34), (2, 1)))} %ADFvfyParEnd
% End of 2^35
%%%%%%%%%%%%%%%%%%%%%%%%%%%%%%%%%%%%%%%%%%%%%%%%%%%%%%%%%%%%%%%%%%%%%%%%%%%%%%%%%%%%%%%%%%

% Charlotte:Gen-PENT-4-5-GDD:HITS-fun:5.5
\adfDgap
%ADFvfyBlocksStart {2,2,2,2,2,2,2,2,2,2,2,2,2,2,2,2,2,2,2,2,2,2,2,2,2,2,2,2,2,2,2,2,2,2,2,2,2,2,2,2,2,2,2,2,2,2,2,2,2,2,2,2,2,2,2,2,2,2,2,2,2,2,2,2,2,2,2,2,2,2,2}
\noindent{\boldmath $ 2^{71} $}~
With the point set $Z_{142}$ partitioned into
 residue classes modulo $71$ for $\{0, 1, \dots, 141\}$,
 the design is generated from

\adfLgap {\adfBfont
%ADFvfyDesignStart
$\{50, 79, 70, 2, 62\}$,
$\{111, 131, 43, 136, 52\}$,
$\{83, 52, 87, 140, 98\}$,\adfsplit
$\{34, 7, 117, 62, 10\}$,
$\{110, 60, 126, 21, 81\}$,
$\{0, 1, 3, 22, 41\}$,\adfsplit
$\{0, 4, 14, 61, 112\}$,
$\{0, 6, 38, 78, 101\}$,
$\{0, 2, 67, 75, 111\}$,\adfsplit
$\{0, 11, 18, 33, 45\}$,
$\{0, 5, 51, 81, 129\}$,
$\{1, 7, 17, 59, 87\}$,\adfsplit
$\{0, 7, 56, 99, 125\}$,
$\{0, 13, 36, 116, 141\}$

}
%ADFvfyBlocksEnd
\adfLgap \noindent by the mapping:
$x \mapsto x + 2 j \adfmod{142}$,
$0 \le j < 71$.
\ADFvfyParStart{(142, ((14, 71, ((142, 2)))), ((2, 71)))} %ADFvfyParEnd
% End of 2^71
%%%%%%%%%%%%%%%%%%%%%%%%%%%%%%%%%%%%%%%%%%%%%%%%%%%%%%%%%%%%%%%%%%%%%%%%%%%%%%%%%%%%%%%%%%

% Charlotte:Gen-PENT-4-5-GDD:HITS-fun:5.5
\adfDgap
%ADFvfyBlocksStart {10,10,10,10,10,10,10,10,10,10,10,10,10,10,10,10,10,10,10,10,10,10,10}
\noindent{\boldmath $ 10^{23} $}~
With the point set $Z_{230}$ partitioned into
 residue classes modulo $23$ for $\{0, 1, \dots, 229\}$,
 the design is generated from

\adfLgap {\adfBfont
%ADFvfyDesignStart
$\{121, 4, 20, 51, 64\}$,
$\{168, 186, 144, 93, 63\}$,
$\{15, 137, 1, 140, 118\}$,\adfsplit
$\{16, 100, 138, 212, 101\}$,
$\{181, 13, 63, 70, 84\}$,
$\{141, 228, 54, 170, 35\}$,\adfsplit
$\{198, 101, 20, 187, 155\}$,
$\{113, 74, 180, 38, 102\}$,
$\{165, 26, 61, 66, 99\}$,\adfsplit
$\{0, 2, 5, 164, 226\}$,
$\{0, 7, 160, 185, 189\}$,
$\{0, 21, 61, 89, 198\}$,\adfsplit
$\{0, 26, 76, 176, 221\}$,
$\{0, 8, 134, 144, 199\}$,
$\{0, 17, 72, 131, 151\}$,\adfsplit
$\{0, 27, 148, 197, 215\}$,
$\{0, 20, 103, 177, 203\}$,
$\{0, 15, 165, 181, 182\}$,\adfsplit
$\{0, 41, 51, 123, 129\}$,
$\{0, 147, 169, 171, 205\}$,
$\{0, 12, 110, 140, 153\}$,\adfsplit
$\{1, 9, 85, 129, 141\}$

}
%ADFvfyBlocksEnd
\adfLgap \noindent by the mapping:
$x \mapsto x + 2 j \adfmod{230}$,
$0 \le j < 115$.
\ADFvfyParStart{(230, ((22, 115, ((230, 2)))), ((10, 23)))} %ADFvfyParEnd
% End of 10^23
%%%%%%%%%%%%%%%%%%%%%%%%%%%%%%%%%%%%%%%%%%%%%%%%%%%%%%%%%%%%%%%%%%%%%%%%%%%%%%%%%%%%%%%%%%

% Charlotte:Gen-PENT-4-5-GDD:HITS-fun:5.5
\adfDgap
%ADFvfyBlocksStart {2,2,2,2,2,2,2,2,2,2,2,2,2,2,2,2,2,2,2,2,2,2,2,2,2,2,2,2,2,2,2,2,2,2,2,2,2,2,2,2,6}
\noindent{\boldmath $ 2^{40} 6^{1} $}~
With the point set $Z_{86}$ partitioned into
 residue classes modulo $40$ for $\{0, 1, \dots, 79\}$, and
 $\{80, 81, \dots, 85\}$,
 the design is generated from

\adfLgap {\adfBfont
%ADFvfyDesignStart
$\{51, 31, 63, 50, 27\}$,
$\{4, 7, 69, 34, 71\}$,
$\{0, 2, 7, 73, 80\}$,\adfsplit
$\{0, 6, 49, 75, 82\}$,
$\{0, 20, 41, 51, 79\}$,
$\{0, 11, 45, 62, 81\}$,\adfsplit
$\{0, 8, 17, 23, 64\}$,
$\{0, 19, 27, 52, 77\}$,
$\{0, 4, 14, 46, 58\}$

}
%ADFvfyBlocksEnd
\adfLgap \noindent by the mapping:
$x \mapsto x + 2 j \adfmod{80}$ for $x < 80$,
$x \mapsto (x - 80 + 3 j \adfmod{6}) + 80$ for $x \ge 80$,
$0 \le j < 40$.
\ADFvfyParStart{(86, ((9, 40, ((80, 2), (6, 3)))), ((2, 40), (6, 1)))} %ADFvfyParEnd
% End of 2^40 6^1
%%%%%%%%%%%%%%%%%%%%%%%%%%%%%%%%%%%%%%%%%%%%%%%%%%%%%%%%%%%%%%%%%%%%%%%%%%%%%%%%%%%%%%%%%%

% Charlotte:Gen-PENT-4-5-GDD:HITS-fun:5.5
\adfDgap
%ADFvfyBlocksStart {4,4,4,4,4,4,4,4,4,4,4,4,4,4,4,4,4,4,4,4,4,4,4,4,4,4,4,4,4,4,4,4,4,4,4,4,4,4,4,4,12}
\noindent{\boldmath $ 4^{40} 12^{1} $}~
With the point set $Z_{172}$ partitioned into
 residue classes modulo $40$ for $\{0, 1, \dots, 159\}$, and
 $\{160, 161, \dots, 171\}$,
 the design is generated from

\adfLgap {\adfBfont
%ADFvfyDesignStart
$\{126, 41, 170, 39, 80\}$,
$\{76, 48, 127, 11, 75\}$,
$\{75, 27, 20, 5, 52\}$,\adfsplit
$\{0, 3, 9, 134, 160\}$,
$\{0, 18, 61, 115, 162\}$,
$\{0, 10, 21, 34, 93\}$,\adfsplit
$\{0, 5, 19, 76, 129\}$,
$\{0, 4, 12, 42, 104\}$,
$\{0, 16, 33, 82, 102\}$

}
%ADFvfyBlocksEnd
\adfLgap \noindent by the mapping:
$x \mapsto x +  j \adfmod{160}$ for $x < 160$,
$x \mapsto (x - 160 + 3 j \adfmod{12}) + 160$ for $x \ge 160$,
$0 \le j < 160$.
\ADFvfyParStart{(172, ((9, 160, ((160, 1), (12, 3)))), ((4, 40), (12, 1)))} %ADFvfyParEnd
% End of 4^40 12^1
%%%%%%%%%%%%%%%%%%%%%%%%%%%%%%%%%%%%%%%%%%%%%%%%%%%%%%%%%%%%%%%%%%%%%%%%%%%%%%%%%%%%%%%%%%

% Charlotte:Gen-PENT-4-5-GDD:HITS-fun:5.5
\adfDgap
%ADFvfyBlocksStart {10,10,10,10,10,10,10,10,10,10,18}
\noindent{\boldmath $ 10^{10} 18^{1} $}~
With the point set $Z_{118}$ partitioned into
 residue classes modulo $9$ for $\{0, 1, \dots, 89\}$,
 $\{90, 91, \dots, 99\}$, and
 $\{100, 101, \dots, 117\}$,
 the design is generated from

\adfLgap {\adfBfont
%ADFvfyDesignStart
$\{23, 19, 29, 81, 16\}$,
$\{24, 20, 30, 82, 17\}$,
$\{53, 101, 36, 48, 99\}$,\adfsplit
$\{54, 102, 37, 49, 90\}$,
$\{109, 48, 4, 64, 15\}$,
$\{0, 8, 43, 90, 111\}$,\adfsplit
$\{0, 26, 59, 75, 94\}$,
$\{1, 15, 35, 61, 115\}$,
$\{0, 21, 22, 61, 113\}$,\adfsplit
$\{0, 1, 20, 76, 114\}$,
$\{0, 2, 50, 69, 100\}$,
$\{0, 23, 31, 66, 95\}$,\adfsplit
$\{0, 37, 79, 97, 106\}$,
$\{0, 29, 51, 53, 108\}$

}
%ADFvfyBlocksEnd
\adfLgap \noindent by the mapping:
$x \mapsto x + 2 j \adfmod{90}$ for $x < 90$,
$x \mapsto (x + 2 j \adfmod{10}) + 90$ for $90 \le x < 100$,
$x \mapsto (x - 100 + 2 j \adfmod{18}) + 100$ for $x \ge 100$,
$0 \le j < 45$.
\ADFvfyParStart{(118, ((14, 45, ((90, 2), (10, 2), (18, 2)))), ((10, 9), (10, 1), (18, 1)))} %ADFvfyParEnd
% End of 10^10 18^1
%%%%%%%%%%%%%%%%%%%%%%%%%%%%%%%%%%%%%%%%%%%%%%%%%%%%%%%%%%%%%%%%%%%%%%%%%%%%%%%%%%%%%%%%%%

% Charlotte:Gen-PENT-4-5-GDD:HITS-fun:5.5
\adfDgap
%ADFvfyBlocksStart {20,20,20,20,20,20,20,20,20,20,36}
\noindent{\boldmath $ 20^{10} 36^{1} $}~
With the point set $Z_{236}$ partitioned into
 residue classes modulo $9$ for $\{0, 1, \dots, 179\}$,
 $\{180, 181, \dots, 199\}$, and
 $\{200, 201, \dots, 235\}$,
 the design is generated from

\adfLgap {\adfBfont
%ADFvfyDesignStart
$\{30, 181, 74, 177, 99\}$,
$\{85, 54, 70, 158, 119\}$,
$\{135, 134, 235, 40, 64\}$,\adfsplit
$\{197, 34, 32, 233, 87\}$,
$\{199, 70, 225, 59, 13\}$,
$\{0, 3, 7, 13, 181\}$,\adfsplit
$\{0, 5, 22, 197, 220\}$,
$\{0, 23, 137, 196, 233\}$,
$\{0, 14, 35, 93, 119\}$,\adfsplit
$\{0, 19, 60, 116, 213\}$,
$\{0, 8, 121, 151, 229\}$,
$\{0, 40, 91, 138, 200\}$,\adfsplit
$\{0, 20, 48, 148, 223\}$,
$\{0, 12, 74, 142, 226\}$

}
%ADFvfyBlocksEnd
\adfLgap \noindent by the mapping:
$x \mapsto x +  j \adfmod{180}$ for $x < 180$,
$x \mapsto (x +  j \adfmod{20}) + 180$ for $180 \le x < 200$,
$x \mapsto (x - 200 +  j \adfmod{36}) + 200$ for $x \ge 200$,
$0 \le j < 180$.
\ADFvfyParStart{(236, ((14, 180, ((180, 1), (20, 1), (36, 1)))), ((20, 9), (20, 1), (36, 1)))} %ADFvfyParEnd
% End of 20^10 36^1
%%%%%%%%%%%%%%%%%%%%%%%%%%%%%%%%%%%%%%%%%%%%%%%%%%%%%%%%%%%%%%%%%%%%%%%%%%%%%%%%%%%%%%%%%%
%
%ADFvfySectionEnd 
\end{proof}

The next two lemmas are special cases of well-known GDD constructions.
%%%%%%%%%%%%%%%%%%%%%%%%%%%%%%%%%%%%%%%%%%%%%%%%%%%%%%%%%%%%%%%%%%%%%%%%%%%%%%%%%%%%%%%%%%
%%%%%%%%%%%%%%%%%%%%%%%%%%%%%%%%%%%%%%%%%%%%%%%%%%%%%%%%%%%%%%%%%%%%%%%%%%%%%%%%%%%%%%%%%%
\begin{lemma}
\label{lem:5-GDD-g-to-gh}
Suppose there exists a $5$-$\mathrm{GDD}$ of type $g_1^{u_1} g_2^{u_2} \dots g_n^{u_n}$.
Then for positive integer $h \notin \{2, 3, 6, 10\}$,
there exists a $5$-$\mathrm{GDD}$ of type $(g_1 h)^{u_1} (g_2 h)^{u_2} \dots (g_n h)^{u_n}$.
\end{lemma}
\begin{proof}
Inflate each point of the 5-GDD by a factor of $h$ and replace the blocks with 5-GDDs of type $h^5$.
There exists a 5-GDD of type $h^5$ for $h \ge 1$, $h \not\in \{2, 3, 6, 10\}$, \cite[Theorem IV.4.16]{Ge2007}.
%%%%%%%%%%%%%%%%%%%%%%%%%%%%%%%%%%%%%%%%%%%%%%%%%%%%%%%%%%%%%%%%%%%%%%%%%%%%%%%%%%%%%%%%%%
\end{proof}
%%%%%%%%%%%%%%%%%%%%%%%%%%%%%%%%%%%%%%%%%%%%%%%%%%%%%%%%%%%%%%%%%%%%%%%%%%%%%%%%%%%%%%%%%%
%%%%%%%%%%%%%%%%%%%%%%%%%%%%%%%%%%%%%%%%%%%%%%%%%%%%%%%%%%%%%%%%%%%%%%%%%%%%%%%%%%%%%%%%%%
\begin{lemma}
\label{lem:4-RGDD-to-5-GDD}
Suppose there exists a $4$-$\mathrm{RGDD}$ of type $g^u$.
Then there exists a $5$-$\mathrm{GDD}$ of type $g^u (g(u - 1)/3)^1$.
\end{lemma}
\begin{proof}
Let $d = g(u - 1)/3$.
The blocks of the 4-RGDD are partitioned into $d$ parallel classes, $P_1$, $P_2$, \dots, $P_d$, say.
Create a new group of size $d$ consisting of points $p_1$, $p_2$, \dots, $p_d$ and augment each block of $P_i$ with $p_i$, $i = 1$, 2, \dots, $d$.
%%%%%%%%%%%%%%%%%%%%%%%%%%%%%%%%%%%%%%%%%%%%%%%%%%%%%%%%%%%%%%%%%%%%%%%%%%%%%%%%%%%%%%%%%%
\end{proof}

With the 5-GDD of type $2^{35}$  we are able to offer a small improvement over \cite[Theorem 12]{GriggsStokes2016}.
%%%%%%%%%%%%%%%%%%%%%%%%%%%%%%%%%%%%%%%%%%%%%%%%%%%%%%%%%%%%%%%%%%%%%%%%%%%%%%%%%%%%%%%%%%
%%%%%%%%%%%%%%%%%%%%%%%%%%%%%%%%%%%%%%%%%%%%%%%%%%%%%%%%%%%%%%%%%%%%%%%%%%%%%%%%%%%%%%%%%%
\begin{theorem}
\label{thm:PENT-5-1mod5-constructed}
There exists a pentagonal geometry $\adfPENT(5,r)$ for all $r \equiv 1 \adfmod{5}$, $r \neq 6$, except possibly for
$r \in \{11, 16, 36, 56, 66, 81, 96, 116\}.$
\end{theorem}
\begin{proof}
This is Theorem 12 of \cite{GriggsStokes2016} apart from $r = 86$, which is stated as a possible exception.

For $\adfPENT(5,86)$, we use Theorem~\ref{thm:GDD-basic} with the degenerate $\adfPENT(5,1)$ and a 5-GDD of type $10^{35}$.
The 5-GDD is constructed from the 5-GDD of type $2^{35}$ of Lemma~\ref{lem:5-GDDs-direct} by Lemma~\ref{lem:5-GDD-g-to-gh} with $h = 5$.
%%%%%%%%%%%%%%%%%%%%%%%%%%%%%%%%%%%%%%%%%%%%%%%%%%%%%%%%%%%%%%%%%%%%%%%%%%%%%%%%%%%%%%%%%%
\end{proof}

To deal with the other residue class, $r \equiv 0 \adfmod{5}$, we first present some new
pentagonal geometries with block size 5.
%%%%%%%%%%%%%%%%%%%%%%%%%%%%%%%%%%%%%%%%%%%%%%%%%%%%%%%%%%%%%%%%%%%%%%%%%%%%%%%%%%%%%%%%%%
%%%%%%%%%%%%%%%%%%%%%%%%%%%%%%%%%%%%%%%%%%%%%%%%%%%%%%%%%%%%%%%%%%%%%%%%%%%%%%%%%%%%%%%%%%
\begin{lemma}
\label{lem:PENT-5-r-direct}
There exist pentagonal geometries $\adfPENT(5,r)$ with no opposite line pairs for
$r \in $ \{$20$, $25$, $30$, $35$, $40$\}.
\end{lemma}
\begin{proof}

{\noindent\boldmath $\adfPENT(5, 20)$}~ With point set $Z_{86}$, the 344 lines are generated from

{\normalsize
%adfPENTstart:1:5:20:2:2:
 $\{20, 41, 65, 66, 67\}$, $\{20, 22, 33, 46, 55\}$, $\{46, 19, 35, 82, 75\}$,\adfsplit
 $\{50, 19, 1, 80, 13\}$, $\{4, 58, 62, 80, 46\}$, $\{49, 77, 34, 82, 85\}$,\adfsplit
 $\{0, 5, 8, 14, 71\}$, $\{0, 17, 27, 31, 69\}$
%adfPENTend
}

\noindent under the action of the mapping $x \mapsto x + 2 \adfmod{86}$.
The deficiency graph is connected and has girth 5.
% The opposite lines are generated by the first 2 blocks: True
% End of PENT(5, 20)
%%%%%%%%%%%%%%%%%%%%%%%%%%%%%%%%%%%%%%%%%%%%%%%%%%%%%%%%%%%%%%%%%%%%%%%%%%%%%%%%%%%%%%%%%%

{\noindent\boldmath $\adfPENT(5, 25)$}~ With point set $Z_{106}$, the 530 lines are generated from

{\normalsize
%adfPENTstart:1:5:25:2:2:
 $\{13, 37, 42, 64, 71\}$, $\{36, 53, 55, 70, 94\}$, $\{32, 33, 28, 43, 59\}$,\adfsplit
 $\{0, 3, 35, 41, 97\}$, $\{0, 25, 33, 44, 103\}$, $\{0, 21, 39, 61, 81\}$,\adfsplit
 $\{0, 2, 47, 51, 100\}$, $\{0, 9, 23, 30, 46\}$, $\{0, 12, 32, 75, 105\}$,\adfsplit
 $\{0, 10, 28, 66, 80\}$
%adfPENTend
}

\noindent under the action of the mapping $x \mapsto x + 2 \adfmod{106}$.
The deficiency graph is connected and has girth 6.
% The opposite lines are generated by the first 2 blocks: True
% End of PENT(5, 25)
%%%%%%%%%%%%%%%%%%%%%%%%%%%%%%%%%%%%%%%%%%%%%%%%%%%%%%%%%%%%%%%%%%%%%%%%%%%%%%%%%%%%%%%%%%

{\noindent\boldmath $\adfPENT(5, 30)$}~ With point set $Z_{126}$, the 756 lines are generated from

{\normalsize
%adfPENTstart:1:5:30:2:2:
 $\{29, 31, 36, 90, 117\}$, $\{10, 23, 96, 98, 105\}$, $\{36, 48, 56, 91, 81\}$,\adfsplit
 $\{68, 94, 39, 25, 50\}$, $\{0, 1, 4, 52, 111\}$, $\{0, 6, 17, 34, 102\}$,\adfsplit
 $\{0, 22, 61, 91, 125\}$, $\{0, 5, 32, 42, 112\}$, $\{0, 3, 50, 66, 87\}$,\adfsplit
 $\{0, 15, 23, 64, 113\}$, $\{0, 47, 73, 93, 105\}$, $\{1, 5, 53, 71, 77\}$
%adfPENTend
}

\noindent under the action of the mapping $x \mapsto x + 2 \adfmod{126}$.
The deficiency graph is connected and has girth 5.
% The opposite lines are generated by the first 2 blocks: True
% End of PENT(5, 30)
%%%%%%%%%%%%%%%%%%%%%%%%%%%%%%%%%%%%%%%%%%%%%%%%%%%%%%%%%%%%%%%%%%%%%%%%%%%%%%%%%%%%%%%%%%

{\noindent\boldmath $\adfPENT(5, 35)$}~ With point set $Z_{146}$, the 1022 lines are generated from

{\normalsize
%adfPENTstart:1:5:35:2:2:
 $\{34, 69, 112, 119, 133\}$, $\{14, 28, 73, 75, 78\}$, $\{33, 40, 93, 22, 62\}$,\adfsplit
 $\{77, 129, 111, 38, 41\}$, $\{0, 1, 2, 25, 127\}$, $\{0, 5, 9, 122, 138\}$,\adfsplit
 $\{0, 15, 37, 93, 121\}$, $\{0, 43, 60, 123, 135\}$, $\{0, 26, 67, 105, 113\}$,\adfsplit
 $\{0, 49, 55, 65, 97\}$, $\{0, 6, 38, 89, 115\}$, $\{0, 19, 28, 58, 70\}$,\adfsplit
 $\{0, 20, 74, 101, 131\}$, $\{0, 4, 48, 84, 94\}$
%adfPENTend
}

\noindent under the action of the mapping $x \mapsto x + 2 \adfmod{146}$.
The deficiency graph is connected and has girth 6.
% The opposite lines are generated by the first 2 blocks: True
% End of PENT(5, 35)
%%%%%%%%%%%%%%%%%%%%%%%%%%%%%%%%%%%%%%%%%%%%%%%%%%%%%%%%%%%%%%%%%%%%%%%%%%%%%%%%%%%%%%%%%%

{\noindent\boldmath $\adfPENT(5, 40)$}~ With point set $Z_{166}$, the 1328 lines are generated from

{\normalsize
%adfPENTstart:1:5:40:2:2:
 $\{42, 124, 155, 157, 163\}$, $\{4, 10, 12, 61, 107\}$, $\{141, 41, 148, 155, 51\}$,\adfsplit
 $\{35, 131, 151, 84, 147\}$, $\{121, 94, 53, 138, 154\}$, $\{0, 1, 26, 29, 30\}$,\adfsplit
 $\{0, 5, 23, 36, 45\}$, $\{0, 10, 35, 102, 119\}$, $\{0, 41, 53, 79, 114\}$,\adfsplit
 $\{0, 11, 43, 85, 129\}$, $\{0, 13, 37, 71, 101\}$, $\{0, 21, 77, 104, 132\}$,\adfsplit
 $\{0, 14, 70, 89, 161\}$, $\{0, 15, 38, 58, 145\}$, $\{0, 12, 73, 88, 112\}$,\adfsplit
 $\{0, 18, 40, 86, 134\}$
%adfPENTend
}

\noindent under the action of the mapping $x \mapsto x + 2 \adfmod{166}$.
The deficiency graph is connected and has girth 6.
% The opposite lines are generated by the first 2 blocks: True
% End of PENT(5, 40)
%%%%%%%%%%%%%%%%%%%%%%%%%%%%%%%%%%%%%%%%%%%%%%%%%%%%%%%%%%%%%%%%%%%%%%%%%%%%%%%%%%%%%%%%%%
\end{proof}
To employ constructions similar to those described in Section~\ref{sec:Block size 4}
we require suitable group divisible designs with block size 5 and type $g^u m^1$.
Unfortunately relatively few seem to be available.
Nevertheless, with 5-GDDs of type $g^u$
we can create infinite sets of pentagonal geometries albeit somewhat sparser than those of
Theorem~\ref{thm:PENT-5-1mod5-constructed}.
%%%%%%%%%%%%%%%%%%%%%%%%%%%%%%%%%%%%%%%%%%%%%%%%%%%%%%%%%%%%%%%%%%%%%%%%%%%%%%%%%%%%%%%%%%
%%%%%%%%%%%%%%%%%%%%%%%%%%%%%%%%%%%%%%%%%%%%%%%%%%%%%%%%%%%%%%%%%%%%%%%%%%%%%%%%%%%%%%%%%%
\begin{theorem}
\label{thm:PENT-5-constructed}
Suppose there exists a $\adfPENT(5, r)$ with $r \ge 20$.

$\mathrm{(i)}$ If $r \equiv 0 \adfmod{5}$, then for $t \ge 0$,
there exists a $\adfPENT(5, (10r + 15)t + r)$.

$\mathrm{(ii)}$ If $r \equiv 0 \adfmod{5}$, then for $t \ge 0$,
there exists a $\adfPENT(5, (10r + 15)t + 5r + 6)$ except possibly when
$r \le 1220$, $r \not\equiv 0 \adfmod{3}$ and $t = 1$.

$\mathrm{(iii)}$ If $r \equiv 1 \adfmod{5}$,
then for $t \ge 2$, there exists a $\adfPENT(5, (2r + 3)t + r)$.
\end{theorem}
\begin{proof}
Let $v = 4r+6$ and note that $v \ge 86$.
In each case we use Theorem~\ref{thm:GDD-basic} together with an appropriate 5-GDD of type $v^u$ for the existence of which
we refer to \cite[Theorem 2.25]{WeiGe2014} unless otherwise stated.

(i) Since $v \equiv 6 \adfmod{20}$ there exists a 5-GDD of type $v^{10t + 1}$ for $t \ge 0$.

(ii) Since $v \equiv 6 \adfmod{20}$ there exists a 5-GDD of type $v^{10t + 5}$ for $t \ge 0$
except possibly when $v \le 4886$, $\gcd(v/2, 30) = 1$ and $t = 1$,
i.e.\ when $r \le 1220$, $r \not\equiv 0 \adfmod{3}$ and $t = 1$.

(iii) Since $v \equiv 10 \adfmod{20}$ there exists a 5-GDD of type $v^{2t + 1}$ for $t \ge 2$.
This follows from \cite[Theorem 2.25]{WeiGe2014} when $t \neq 11$; 
otherwise the 5-GDD of type $v^{23}$ is constructed by Lemma~\ref{lem:5-GDD-g-to-gh} 
using the 5-GDD of type $10^{23}$ from Lemma~\ref{lem:5-GDDs-direct}.
%%%%%%%%%%%%%%%%%%%%%%%%%%%%%%%%%%%%%%%%%%%%%%%%%%%%%%%%%%%%%%%%%%%%%%%%%%%%%%%%%%%%%%%%%%
\end{proof}

%%%%%%%%%%%%%%%%%%%%%%%%%%%%%%%%%%%%%%%%%%%%%%%%%%%%%%%%%%%%%%%%%%%%%%%%%%%%%%%%%%%%%%%%%%
%%%%%%%%%%%%%%%%%%%%%%%%%%%%%%%%%%%%%%%%%%%%%%%%%%%%%%%%%%%%%%%%%%%%%%%%%%%%%%%%%%%%%%%%%%
\begin{theorem}
\label{thm:PENT-5-0OLP-constructed}
There exist pentagonal geometries $\adfPENT(5,r)$ with no opposite line pairs for the following $r$,
where in each case $t = 0$, $1$, \dots:
$$215 t + 20,~~~~~ 265 t + 25,~~~~~ 315 t + 30,~~~~~ 365 t + 35,~~~~~ 415 t + 40.$$
There exist pentagonal geometries $\adfPENT(5,r)$ with no opposite line pairs for the following $r$:
\begin{center}
\begin{tabular}{ll}
$215 t + 106,~~~ t \ge 0,~ t \neq 1$, & $265 t + 131,~~~  t \ge 0,~ t \neq 1$, \\
$315 t + 156,~~~ t \ge 0           $, & $365 t + 181,~~~  t \ge 0,~ t \neq 1$. \\
$415 t + 206,~~~ t \ge 0,~ t \neq 1$. &
\end{tabular}
\end{center}
\end{theorem}
\begin{proof}
This follows from Lemma~\ref{lem:PENT-5-r-direct} and Theorem~\ref{thm:PENT-5-constructed}.
%%%%%%%%%%%%%%%%%%%%%%%%%%%%%%%%%%%%%%%%%%%%%%%%%%%%%%%%%%%%%%%%%%%%%%%%%%%%%%%%%%%%%%%%%%
\end{proof}
Further iterations of the constructions described in
Theorem~\ref{thm:PENT-5-0OLP-constructed}
fail to yield pentagonal geometries with new parameters.
For example, a $\adfPENT(5, 215t + 20)$ together with a 5-GDD of type $(860t + 86)^{10u + 1}$
gives a $\adfPENT(5, 215 (10tu + t + u) + 20)$.

Now observe that the number of points in a $\adfPENT(5, r)$
quite often has the form $gq$, where $g$ is even and $q$ is a prime power that is not too small.
Then there exists a $(u + 1)$-GDD of type $q^{u + 1}$
(i.e.\ a transversal design TD$(u+1, q)$) for $5 \le u \le q$, \cite{AbelColbournDinitz2007},
and further pentagonal geometries with block size 5 are constructible.
The next lemma is an application of Wilson's Fundamental Construction, \cite[Theorem IV.2.5]{GreigMullen2007}.
%%%%%%%%%%%%%%%%%%%%%%%%%%%%%%%%%%%%%%%%%%%%%%%%%%%%%%%%%%%%%%%%%%%%%%%%%%%%%%%%%%%%%%%%%%
%%%%%%%%%%%%%%%%%%%%%%%%%%%%%%%%%%%%%%%%%%%%%%%%%%%%%%%%%%%%%%%%%%%%%%%%%%%%%%%%%%%%%%%%%%
\begin{lemma}
\label{lem:TD-construction-g^m}
Let $g$, $u$ and $q$ be positive integers such that $u \le q$.
Suppose $D$ is a set of non-negative integers such that there exists a $5$-$\mathrm{GDD}$ of type $g^u d$ for each $d \in D$.
Suppose also that there exists a $(u + 1)$-$\mathrm{GDD}$ of type $q^{u + 1}$.
Then there exists a $5$-$\mathrm{GDD}$ of type $(gq)^u m^1$ for each $m$ that is a sum of $q$ elements of $D$.
\end{lemma}
\begin{proof}
Take the $(u + 1)$-$\mathrm{GDD}$ of type $q^{u+1}$,
select a group and inflate its points by factors $d_1$, $d_2$, \dots, $d_q \in D$.
Inflate all other points by a factor of $g$.
Overlay the inflated blocks with 5-GDDs of types $g^u d_i^1$, $i = 1$, 2, \dots, $q$, as appropriate.
The result is a 5-GDD of type $(g q)^u (d_1 + d_2 + \dots + d_q)^1$.
%%%%%%%%%%%%%%%%%%%%%%%%%%%%%%%%%%%%%%%%%%%%%%%%%%%%%%%%%%%%%%%%%%%%%%%%%%%%%%%%%%%%%%%%%%
\end{proof}

Our next two lemmas are special cases of Lemma~\ref{lem:TD-construction-g^m} and exploit the known existence of certain 5-GDD types.
%%%%%%%%%%%%%%%%%%%%%%%%%%%%%%%%%%%%%%%%%%%%%%%%%%%%%%%%%%%%%%%%%%%%%%%%%%%%%%%%%%%%%%%%%%
%%%%%%%%%%%%%%%%%%%%%%%%%%%%%%%%%%%%%%%%%%%%%%%%%%%%%%%%%%%%%%%%%%%%%%%%%%%%%%%%%%%%%%%%%%
\begin{lemma}
\label{lem:5-GDD-g^40-m^1}
Let $g$ and $q$ be positive integers such that $g$ is even, $g \notin \{6, 12\}$,
$q \ge 40$ and there exist $39$ mutually orthogonal Latin squares of side $q$.
Then there exist
$5$-$\mathrm{GDD}$s of types $g^{40} (gd)^1$ for $d \in \{1, 3, 13\}$ as well as
$5$-$\mathrm{GDD}$s of types $(gq)^{40} m^1$ for $m \in M_{40}(g,q)$, where
\begin{equation}
% Table[i,{i,q,13q-40,2}]~Join~(13q-{36,34,32,30,24,22,20,12,10,0})
\begin{split}
\label{eqn:M40}
M_{40}(g,q) =~ & \{gj: j \in \{q, q + 2, \dots, 13q - 40\}\}  \\
               & \cup \{g(13q - j): j \in \{36,34,32,30,24,22,20,12,10,0\}\}.
\end{split}
\end{equation}
\end{lemma}
\begin{proof}
There exist 5-GDDs of the following types:
$2^{41}$ and $4^{41}$ from \cite[Theorem 2.25]{WeiGe2014};
$2^{40} 6^1$ and $4^{40} 12^1$ from Lemma~\ref{lem:5-GDDs-direct};
$2^{40} 26^1$ and $4^{40} 52^1$ obtained by Lemma~\ref{lem:4-RGDD-to-5-GDD} from
4-RGDDs of types $2^{40}$ and $4^{40}$, \cite{GeLing2005} or \cite[Theorem IV.5.44]{GeMiao2007}.
Except for $g \in \{6, 12\}$ we use Lemma~\ref{lem:5-GDD-g-to-gh} to construct
5-GDDs of types $g^{40} (gd)^1$ for $d \in \{1, 3, 13\}$ from the corresponding
5-GDDs of types $2^{40} (2d)^1$ and $4^{40} (4d)^1$.

For the 5-GDDs of type $(gq)^{40} m^1$, $m \in M_{40}(g,q)$, we use Lemma~\ref{lem:TD-construction-g^m} with $u = 40$ and $D = \{g, 3g, 13g\}$.
The existence of 39 MOLS of side $q$ guarantees the existence of a 41-GDD of type $q^{41}$.
The possible sums of $q$ elements taken from $D$ are given by (\ref{eqn:M40}).
%%%%%%%%%%%%%%%%%%%%%%%%%%%%%%%%%%%%%%%%%%%%%%%%%%%%%%%%%%%%%%%%%%%%%%%%%%%%%%%%%%%%%%%%%%
\end{proof}
%%%%%%%%%%%%%%%%%%%%%%%%%%%%%%%%%%%%%%%%%%%%%%%%%%%%%%%%%%%%%%%%%%%%%%%%%%%%%%%%%%%%%%%%%%
%%%%%%%%%%%%%%%%%%%%%%%%%%%%%%%%%%%%%%%%%%%%%%%%%%%%%%%%%%%%%%%%%%%%%%%%%%%%%%%%%%%%%%%%%%
\begin{lemma}
\label{lem:5-GDD-g^10-m^1}
Let $g$ and $q$ be positive integers such that $g$ is a multiple of $10$, $g \notin \{30, 60\}$,
$q \ge 10$ and there exist $9$ mutually orthogonal Latin squares of side $q$.
Then there exist
$5$-$\mathrm{GDD}$s of types $g^{10} (gd/5)^1$ for $d \in \{5, 9, 15\}$ as well as
$5$-$\mathrm{GDD}$s of types $(gq)^{10} m^1$ for $m \in M_{10}(g,q)$, where
\begin{equation}
% jjbase = Table[j, {j, 5 q, 15 q, 2}]
% jjex = {5 q + 2, 5 q + 6, 15 q - 8, 15 q - 4, 15 q - 2, 15 q - 14}
\begin{split}
\label{eqn:M10}
M_{10}(g,q) =~ & \{gj/5: j \in \{5q, 5q + 2, \dots, 15q\}  \\
     & ~~~~~ \setminus \{5 q + 2, 5 q + 6, 15 q - 14, 15 q - 8, 15 q - 4, 15 q - 2\}\}.
\end{split}
\end{equation}
\end{lemma}
\begin{proof}
The proof is similar to that of Lemma~\ref{lem:5-GDD-g^40-m^1}.
There exist 5-GDDs of the following types:
$10^{11}$ and $20^{11}$ from \cite[Theorem 2.25]{WeiGe2014};
$10^{10} 18^1$ and $20^{10} 36^1$ from Lemma~\ref{lem:5-GDDs-direct};
$10^{10} 30^1$ and $20^{10} 60^1$ by Lemma~\ref{lem:4-RGDD-to-5-GDD} from
4-RGDDs of types $10^{10}$ and $20^{10}$, \cite[Theorem 4.21]{WeiGe2014}.
Except for $g \in \{30, 60\}$ we use Lemma~\ref{lem:5-GDD-g-to-gh} to construct
5-GDDs of types $g^{10} (gd/5)^1$ for $d \in \{5, 9, 15\}$ from the corresponding
5-GDDs of types $10^{10} (2d)^1$ and $20^{10} (4d)^1$.

For the 5-GDDs of type $(gq)^{10} m^1$, $m \in M_{10}(g,q)$, we use Lemma~\ref{lem:TD-construction-g^m} with
$u = 10$ and $D = \{gj/5: j \in \{5, 9, 15\}\}$.
The possible sums of $q$ elements taken from $D$ are given by (\ref{eqn:M10}).
%%%%%%%%%%%%%%%%%%%%%%%%%%%%%%%%%%%%%%%%%%%%%%%%%%%%%%%%%%%%%%%%%%%%%%%%%%%%%%%%%%%%%%%%%%
\end{proof}

We give some examples to show how Lemmas~\ref{lem:5-GDD-g^40-m^1} and \ref{lem:5-GDD-g^10-m^1} can yield new pentagonal geometries.
%%%%%%%%%%%%%%%%%%%%%%%%%%%%%%%%%%%%%%%%%%%%%%%%%%%%%%%%%%%%%%%%%%%%%%%%%%%%%%%%%%%%%%%%%%
%%%%%%%%%%%%%%%%%%%%%%%%%%%%%%%%%%%%%%%%%%%%%%%%%%%%%%%%%%%%%%%%%%%%%%%%%%%%%%%%%%%%%%%%%%
\begin{construction}
\label{con:PENT-5-TD-construction-40}
{\rm
Start with a pentagonal geometry $\adfPENT(5, r)$, where $4r + 6 = g q$ such that $g$ is even, $g \notin \{6, 12\}$ and
there exist 39 MOLS of side $q$. Define $M_{40}(g,q)$ as in (\ref{eqn:M40}).
Use Lemma~\ref{lem:5-GDD-g^40-m^1} to construct 5-GDDs of types $(gq)^{40} m^1$, $m \in M_{40}(g,q)$.

Using Theorem~\ref{thm:GDD-basic} with 40 copies of the $\adfPENT(5, r)$,
for each $m \in M_{40}(g,q)$ such that there exists a $\adfPENT(5, (m - 6)/4)$
we can construct a $\adfPENT(5, 40r +(m-6)/4 + 60)$.
}
%%%%%%%%%%%%%%%%%%%%%%%%%%%%%%%%%%%%%%%%%%%%%%%%%%%%%%%%%%%%%%%%%%%%%%%%%%%%%%%%%%%%%%%%%%
\end{construction}
%%%%%%%%%%%%%%%%%%%%%%%%%%%%%%%%%%%%%%%%%%%%%%%%%%%%%%%%%%%%%%%%%%%%%%%%%%%%%%%%%%%%%%%%%%
%%%%%%%%%%%%%%%%%%%%%%%%%%%%%%%%%%%%%%%%%%%%%%%%%%%%%%%%%%%%%%%%%%%%%%%%%%%%%%%%%%%%%%%%%%
\begin{construction}
\label{con:PENT-5-TD-construction-10}
{\rm
Start with a pentagonal geometry $\adfPENT(5, r)$, where $4r + 6 = g q$ such that $g$ is divisible by 10, $g \notin \{30, 60\}$ and
there exist 9 MOLS of side $q$. Define $M_{10}(g,q)$ as in (\ref{eqn:M10}).
Use Lemma~\ref{lem:5-GDD-g^10-m^1} to construct 5-GDDs of types $(gq)^{10} m^1$, $m \in M_{10}(g,q)$.

Using Theorem~\ref{thm:GDD-basic} with 10 copies of the $\adfPENT(5, r)$,
for each $m \in M_{10}(g,q)$ such that there exists a $\adfPENT(5, (m - 6)/4)$
we can construct a $\adfPENT(5, 10r +(m-6)/4 + 15)$.
}
%%%%%%%%%%%%%%%%%%%%%%%%%%%%%%%%%%%%%%%%%%%%%%%%%%%%%%%%%%%%%%%%%%%%%%%%%%%%%%%%%%%%%%%%%%
\end{construction}

%
% PUT THIS TABLE HERE TO STOP IT TRESPASSING INTO THE REFERENCES.
%
\begin{table}[h]
{\scriptsize

\begin{tabular}{@{}r@{~}r@{~}r@{~}r|l@{}}
$g$ & $u$ & $q$ & $r_0$ & $\adfPENT(5,r)$ \\
\hline
% {2, 40, 43, 26, 20, {True}, {{86, 0, 20, 880}, {830, 31, 206, 1066}}}
2 & 40 & 43 & 20 & 880, 1066 \\
% {2, 40, 53, 26, 25, {True}, {{106, 0, 25, 1085}, {730, 26, 181, 1241}, {946, 35, 235, 1295}}}
2 & 40 & 53 & 25 & 1085, 1241, 1295 \\
% {2, 40, 73, 26, 35, {True}, {{146, 0, 35, 1495}, {530, 16, 131, 1591}, {1826, 70, 455, 1915}}}
2 & 40 & 73 & 35 & 1495, 1591, 1915 \\
% {2, 40, 83, 26, 40, {True}, {{166, 0, 40, 1700}, {430, 11, 106, 1766}, {1606, 60, 400, 2060}}}
2 & 40 & 83 & 40 & 1700, 1766, 2060 \\
% {2, 64, 73, 42, 35, {True}, {{146, 0, 35, 2371}, {946, 20, 235, 2571}, {1386, 31, 345, 2681}, {1826, 42, 455, 2791}, {2226, 52, 555, 2891}, {2666, 63, 665, 3001}, {3066, 73, 765, 3101}}}
2 & 64 & 73 & 35 & 2371, 2571, 2681, 2791, 2891, 3001, 3101 \\
% {2, 64, 83, 42, 40, {True}, {{166, 0, 40, 2696}, {1166, 25, 290, 2946}, {1606, 36, 400, 3056}, {1806, 41, 450, 3106}, {2646, 62, 660, 3316}, {3286, 78, 820, 3476}, {3486, 83, 870, 3526}}}
2 & 64 & 83 & 40 & 2696, 2946, 3056, 3106, 3316, 3476, 3526 \\
% {10, 10, 43, 30, 106, {True}, {{430, 0, 106, 1181}, {530, 5, 131, 1206}, {630, 10, 156, 1231}, {730, 15, 181, 1256}, {830, 20, 206, 1281}}}
10 & 10 & 43 & 106 & 1181, 1206, 1231, 1256, 1281 \\
% {10, 10, 53, 30, 131, {True}, {{530, 0, 131, 1456}, {630, 5, 156, 1481}, {730, 10, 181, 1506}, {830, 15, 206, 1531}}}
10 & 10 & 53 & 131 & 1456, 1481, 1506, 1531 \\
% {10, 10, 73, 30, 181, {True}, {{730, 0, 181, 2006}, {830, 5, 206, 2031}, {1890, 58, 471, 2296}, {2150, 71, 536, 2361}}}
10 & 10 & 73 & 181 & 2006, 2031, 2296, 2361 \\
% {10, 10, 83, 30, 206, {True}, {{830, 0, 206, 2281}, {1890, 53, 471, 2546}, {2150, 66, 536, 2611}}}
10 & 10 & 83 & 206 & 2281, 2546, 2611 \\
% {10, 16, 43, 50, 106, {True}, {{430, 0, 106, 1826}, {630, 5, 156, 1876}, {830, 10, 206, 1926}, {2150, 43, 536, 2256}}}
10 & 16 & 43 & 106 & 1826, 1876, 1926, 2256 \\
% {10, 16, 53, 50, 131, {True}, {{530, 0, 131, 2251}, {730, 5, 181, 2301}, {1890, 34, 471, 2591}, {2650, 53, 661, 2781}}}
10 & 16 & 53 & 131 & 2251, 2301, 2591, 2781 \\
% {10, 16, 73, 50, 181, {True}, {{730, 0, 181, 3101}, {1890, 29, 471, 3391}, {2650, 48, 661, 3581}, {3010, 57, 751, 3671}, {3650, 73, 911, 3831}}}
10 & 16 & 73 & 181 & 3101, 3391, 3581, 3671, 3831 \\
% {10, 16, 83, 50, 206, {True}, {{830, 0, 206, 3526}, {2150, 33, 536, 3856}, {3150, 58, 786, 4106}, {3710, 72, 926, 4246}, {3870, 76, 966, 4286}, {4150, 83, 1036, 4356}}}
10 & 16 & 83 & 206 & 3526, 3856, 4106, 4246, 4286, 4356 \\
% {10, 22, 43, 70, 106, {True}, {{430, 0, 106, 2471}, {730, 5, 181, 2546}, {2650, 37, 661, 3026}, {3010, 43, 751, 3116}}}
10 & 22 & 43 & 106 & 2471, 2546, 3026, 3116 \\
% {10, 22, 53, 70, 131, {True}, {{530, 0, 131, 3046}, {830, 5, 206, 3121}, {2150, 27, 536, 3451}, {3650, 52, 911, 3826}, {3710, 53, 926, 3841}}}
10 & 22 & 53 & 131 & 3046, 3121, 3451, 3826, 3841 \\
% {10, 22, 73, 70, 181, {True}, {{730, 0, 181, 4196}, {2650, 32, 661, 4676}, {3010, 38, 751, 4766}, {4150, 57, 1036, 5051}, {4270, 59, 1066, 5081}, {4930, 70, 1231, 5246}, {5110, 73, 1276, 5291}}}
10 & 22 & 73 & 181 & 4196, 4676, 4766, 5051, 5081, 5246, 5291 \\
% {10, 22, 83, 70, 206, {True}, {{830, 0, 206, 4771}, {2150, 22, 536, 5101}, {3650, 47, 911, 5476}, {3710, 48, 926, 5491}, {4730, 65, 1181, 5746}, {4970, 69, 1241, 5806}, {5030, 70, 1256, 5821}, {5810, 83, 1451, 6016}}}
10 & 22 & 83 & 206 & 4771, 5101, 5476, 5491, 5746, 5806, 5821, 6016 \\
% {10, 28, 43, 90, 106, {True}, {{430, 0, 106, 3116}, {830, 5, 206, 3216}, {3150, 34, 786, 3796}, {3710, 41, 926, 3936}, {3870, 43, 966, 3976}}}
10 & 28 & 43 & 106 & 3116, 3216, 3796, 3936, 3976 \\
% {10, 28, 53, 90, 131, {True}, {{530, 0, 131, 3841}, {1890, 17, 471, 4181}, {3010, 31, 751, 4461}, {3650, 39, 911, 4621}, {4770, 53, 1191, 4901}}}
10 & 28 & 53 & 131 & 3841, 4181, 4461, 4621, 4901 \\
% {10, 28, 73, 90, 181, {True}, {{730, 0, 181, 5291}, {2650, 24, 661, 5771}, {4410, 46, 1101, 6211}, {4730, 50, 1181, 6291}, {4970, 53, 1241, 6351}, {5130, 55, 1281, 6391}, {5930, 65, 1481, 6591}, {6570, 73, 1641, 6751}}}
10 & 28 & 73 & 181 & 5291, 5771, 6211, 6291, 6351, 6391, 6591, 6751 \\
% {10, 28, 83, 90, 206, {True}, {{830, 0, 206, 6016}, {3150, 29, 786, 6596}, {3710, 36, 926, 6736}, {3870, 38, 966, 6776}, {4270, 43, 1066, 6876}, {4830, 50, 1206, 7016}, {6030, 65, 1506, 7316}, {7070, 78, 1766, 7576}, {7310, 81, 1826, 7636}, {7470, 83, 1866, 7676}}}
10 & 28 & 83 & 206 & 6016, 6596, 6736, 6776, 6876, 7016, 7316, 7576, 7636, 7676 \\
% {10, 40, 43, 130, 106, {True}, {{430, 0, 106, 4406}, {4150, 31, 1036, 5336}, {4270, 32, 1066, 5366}, {5110, 39, 1276, 5576}, {5590, 43, 1396, 5696}}}
10 & 40 & 43 & 106 & 4406, 5336, 5366, 5576, 5696 \\
% {10, 40, 53, 130, 131, {True}, {{530, 0, 131, 5431}, {3650, 26, 911, 6211}, {4730, 35, 1181, 6481}, {4970, 37, 1241, 6541}, {5810, 44, 1451, 6751}, {5930, 45, 1481, 6781}, {6890, 53, 1721, 7021}}}
10 & 40 & 53 & 131 & 5431, 6211, 6481, 6541, 6751, 6781, 7021 \\
% {10, 40, 73, 130, 181, {True}, {{730, 0, 181, 7481}, {2650, 16, 661, 7961}}}
10 & 40 & 73 & 181 & 7481, 7961 \\
% {42, 10, 43, 126, 450, {True}, {{1806, 0, 450, 4965}, {1890, 1, 471, 4986}, {2226, 5, 555, 5070}, {2646, 10, 660, 5175}, {3066, 15, 765, 5280}, {3150, 16, 786, 5301}, {3486, 20, 870, 5385}, {3906, 25, 975, 5490}, {4410, 31, 1101, 5616}, {4830, 36, 1206, 5721}, {5166, 40, 1290, 5805}}}
42 & 10 & 43 & 450 & 4965, 4986, 5070, 5175, 5280, 5301, 5385, 5490, 5616, 5721, 5805 \\
% {42, 10, 53, 126, 555, {True}, {{2226, 0, 555, 6120}, {2646, 5, 660, 6225}, {3066, 10, 765, 6330}, {3150, 11, 786, 6351}, {3486, 15, 870, 6435}, {3906, 20, 975, 6540}, {4410, 26, 1101, 6666}, {4830, 31, 1206, 6771}, {5166, 35, 1290, 6855}, {5670, 41, 1416, 6981}, {6426, 50, 1605, 7170}}}
42 & 10 & 53 & 555 & 6120, 6225, 6330, 6351, 6435, 6540, 6666, 6771, 6855, 6981, 7170 \\
% {42, 14, 43, 182, 450, {True}, {{1806, 0, 450, 6771}, {2226, 3, 555, 6876}, {2646, 6, 660, 6981}, {3066, 9, 765, 7086}, {3486, 12, 870, 7191}, {3906, 15, 975, 7296}, {5166, 24, 1290, 7611}, {6426, 33, 1605, 7926}}}
42 & 14 & 43 & 450 & 6771, 6876, 6981, 7086, 7191, 7296, 7611, 7926 \\
% {50, 10, 43, 150, 536, {True}, {{2150, 0, 536, 5911}, {2650, 5, 661, 6036}, {3150, 10, 786, 6161}, {3650, 15, 911, 6286}, {4150, 20, 1036, 6411}, {6450, 43, 1611, 6986}}}
50 & 10 & 43 & 536 & 5911, 6036, 6161, 6286, 6411, 6986 \\
% {50, 10, 53, 150, 661, {True}, {{2650, 0, 661, 7286}, {3150, 5, 786, 7411}, {3650, 10, 911, 7536}, {4150, 15, 1036, 7661}}}
50 & 10 & 53 & 661 & 7286, 7411, 7536, 7661 \\
% {54, 10, 49, 162, 660, {True}, {{2646, 0, 660, 7275}, {5130, 23, 1281, 7896}}}
54 & 10 & 49 & 660 & 7275, 7896 \\
% {70, 10, 27, 210, 471, {True}, {{1890, 0, 471, 5196}, {3010, 8, 751, 5476}, {3150, 9, 786, 5511}, {3710, 13, 926, 5651}, {4270, 17, 1066, 5791}, {4410, 18, 1101, 5826}, {4830, 21, 1206, 5931}, {4970, 22, 1241, 5966}, {5110, 23, 1276, 6001}, {5670, 27, 1416, 6141}}}
70 & 10 & 27 & 471 & 5196, 5476, 5511, 5651, 5791, 5826, 5931, 5966, 6001, 6141 \\
% {90, 6, 7, 150, 156, {True}, {{630, 0, 156, 1101}}}
90 & 6 & 7 & 156 & 1101 \\
% {90, 6, 43, 150, 966, {True}, {{3870, 0, 966, 6771}, {4410, 9, 1101, 6906}, {4770, 15, 1191, 6996}, {4830, 16, 1206, 7011}, {5130, 21, 1281, 7086}, {5670, 30, 1416, 7221}, {6030, 36, 1506, 7311}, {6450, 43, 1611, 7416}}}
90 & 6 & 43 & 966 & 6771, 6906, 6996, 7011, 7086, 7221, 7311, 7416 \\
% {90, 6, 49, 150, 1101, {True}, {{4410, 0, 1101, 7716}, {4770, 6, 1191, 7806}, {4830, 7, 1206, 7821}, {5130, 12, 1281, 7896}}}
90 & 6 & 49 & 1101 & 7716, 7806, 7821, 7896 \\
% {98, 10, 27, 294, 660, {True}, {{2646, 0, 660, 7275}, {4410, 9, 1101, 7716}}}
98 & 10 & 27 & 660 & 7275, 7716 \\
% {126, 10, 11, 378, 345, {True}, {{1386, 0, 345, 3810}, {1890, 2, 471, 3936}, {2646, 5, 660, 4125}, {3150, 7, 786, 4251}, {3906, 10, 975, 4440}}}
126 & 10 & 11 & 345 & 3810, 3936, 4125, 4251, 4440 \\
% {146, 10, 11, 438, 400, {True}, {{1606, 0, 400, 4415}, {3066, 5, 765, 4780}, {3650, 7, 911, 4926}, {4526, 10, 1130, 5145}}}
146 & 10 & 11 & 400 & 4415, 4780, 4926, 5145 \\
% {166, 10, 11, 498, 455, {True}, {{1826, 0, 455, 5020}, {3486, 5, 870, 5435}, {4150, 7, 1036, 5601}, {5146, 10, 1285, 5850}}}
166 & 10 & 11 & 455 & 5020, 5435, 5601, 5850 \\
% {210, 6, 9, 350, 471, {True}, {{1890, 0, 471, 3306}, {3010, 8, 751, 3586}, {3150, 9, 786, 3621}}}
210 & 6 & 9 & 471 & 3306, 3586, 3621 \\
% {210, 8, 9, 490, 471, {True}, {{1890, 0, 471, 4251}, {3010, 4, 751, 4531}, {4410, 9, 1101, 4881}}}
210 & 8 & 9 & 471 & 4251, 4531, 4881 \\
% {270, 6, 7, 450, 471, {True}, {{1890, 0, 471, 3306}, {3150, 7, 786, 3621}}}
270 & 6 & 7 & 471 & 3306, 3621 \\
% {450, 6, 7, 750, 786, {True}, {{3150, 0, 786, 5511}}}
450 & 6 & 7 & 786 & 5511 \\
% {630, 6, 7, 1050, 1101, {True}, {{4410, 0, 1101, 7716}, {4830, 1, 1206, 7821}}}
630 & 6 & 7 & 1101 & 7716, 7821
% {217, 44}
\end{tabular}

}
\vskip 2mm
\caption{Construction~\ref{con:PENT-5-TD-construction-g-u-q}}
\label{tab:PENT-5-TD-construction-g-u-q}

\end{table}

To exploit these two constructions, first
let $R = \{20, 25, 35, 40\}$ and recall from Lemma~\ref{lem:PENT-5-r-direct} that
there exists a $\adfPENT(5,r)$ with no opposite line pair for
$r \in R \cup \{30\}$.
Let $r_0 \in R$ and $q = 2r_0 + 3$, half the number of points in the $\adfPENT(5, r_0)$.
Observe that $q$ is prime, $q \ge 40$ and that we cannot include 30 in $R$ because the corresponding $q$ is unsuitable.
Invoking Construction~\ref{con:PENT-5-TD-construction-40} with $g = 2$ gives
\begin{equation*}
\begin{split}
M_{40}(2,q) =~ & \{2q, 2q + 4, \dots, 26q - 80\}  \\
               & \cup \{26q - 2j: j \in \{36,34,32,30,24,22,20,12,10,0\}\},
\end{split}
\end{equation*}
and after eliminating those $m \in M_{40}(2,q)$ where
either a $\adfPENT(5, (m - 6)/4)$ does not exist, or a $\adfPENT(5, 40r_0 + (m - 6)/4 + 60)$ exists by
Theorem~\ref{thm:PENT-5-0OLP-constructed},
we obtain the following pentagonal geometries $\adfPENT(5, r)$, all with no opposite line pairs:
\begin{equation}
\label{eqn:PENT-5-885-et-al}
\begin{split}
r_0 &= 20: r = 885, 890, 895, 900, 991, 1016, 1041, 1066,\\
r_0 &= 25: r = 1090, 1100, 1166, 1216, 1241, 1266, 1295,\\
r_0 &= 35: r = 1500, 1566, 1591, 1616, 1666, 1695, 1750, 1805, 1910, 1915,\\
r_0 &= 40: r = 1766, 1791, 1816, 1841, 1895, 1950, 2005, 2060, 2110, 2131.
% {35, 4, 35, 4}
\end{split}
\end{equation}
For instance, $r = 2131$ arises from $r_0 = 40$, $q = 83$ and $s = 471$ corresponding to $4s + 6 = 1890 \in M_{40}(2,83)$.
Indeed, 1890 is expressible as the sum of $83$ elements from $\{2, 6, 26\}$: $1890 = 71 \cdot 26 + 5 \cdot 6 + 7 \cdot 2$, say.
The relevant 5-GDD type is $166^{40} 1890^1$.
Thus a $\adfPENT(5,2131)$ is constructed from existing pentagonal geometries with no opposite line pairs,
40 copies of $\adfPENT(5, 40)$ and one of $\adfPENT(5, 471)$.
The $\adfPENT(5, 471)$ is obtained by Theorem~\ref{thm:PENT-5-0OLP-constructed} from 15 copies of $\adfPENT(5, 30)$
and a 5-GDD of type $126^{15}$, which exists by \cite[Theorem 2.25]{WeiGe2014}.

Now let $R = \{106, 131, 181, 206\}$ and recall from Theorem~\ref{thm:PENT-5-0OLP-constructed} that
there exists a pentagonal geometry $\adfPENT(5,r)$ for each $r \in R \cup \{156\}$.
Let $r_0 \in R$ and $q = (4r_0 + 6)/10$.
Invoking Construction~\ref{con:PENT-5-TD-construction-10} with $g = 10$ gives
\begin{equation*}
\begin{split}
M_{10}(10,q) =~ & \{10q, 10q + 4, \dots, 30q\}  \\
                     & ~~~~~ \setminus \{10 q+4,10 q+12,30 q-28,30 q-16,30 q-8,30 q-4\},
\end{split}
\end{equation*}
and we obtain the following pentagonal geometries $\adfPENT(5, r)$, all with no opposite line pairs:
\begin{equation*}
\label{eqn:PENT-5-1206-et-al}
\begin{split}
r_0 &= 106: r = 1206, 1231, 1256, 1281, 1365,\\
r_0 &= 131: r = 1481, 1506, 1531, 1560, 1670,\\
r_0 &= 181: r = 2031, 2060, 2275, 2280, 2296,\\
r_0 &= 206: r = 2310, 2365, 2420, 2475, 2525, 2546, 2611, 2630.
% {23, 4, 23, 4}
\end{split}
\end{equation*}

Constructions~\ref{con:PENT-5-TD-construction-40} and \ref{con:PENT-5-TD-construction-10}
are of course constrained by the specific values of $u$, namely 40 and 10 respectively.
In the next construction $u$ is not so restricted.
However the current state of knowledge regarding the existence of non-uniform 5-GDDs
limits the range of $m$ for the constructing 5-GDDs of type $(gq)^u m^1$
by means of Lemma~\ref{lem:TD-construction-g^m}.
%%%%%%%%%%%%%%%%%%%%%%%%%%%%%%%%%%%%%%%%%%%%%%%%%%%%%%%%%%%%%%%%%%%%%%%%%%%%%%%%%%%%%%%%%%
%%%%%%%%%%%%%%%%%%%%%%%%%%%%%%%%%%%%%%%%%%%%%%%%%%%%%%%%%%%%%%%%%%%%%%%%%%%%%%%%%%%%%%%%%%
\begin{construction}
\label{con:PENT-5-TD-construction-g-u-q}
{\rm
Let $g$, $u$ and $q$ be positive integers such that
$u \ge 6$, $u$ is even, $gu \equiv 0 \adfmod{4}$, $g(u - 1) \equiv 0 \adfmod{3}$, $g^2 (u + 1)u \equiv 0 \adfmod{20}$,
$gq \equiv 6 \textrm{~or~} 10 \adfmod{20}$
and there exist $u - 1$ MOLS of side $q$.

Let $d = g(u - 1)/3$ and suppose that there exist 5-GDDs of types $g^{u + 1}$ and $g^u d^1$.
Then, by Lemma~\ref{lem:TD-construction-g^m}, there exists a 5-GDD of type $(gq)^u m^1$ for
each $m \in M$, where
$$M = \{j d + (q - j) g: j = 0, 1, \dots, q\}.$$

Let $r_0 = (gq - 6)/4$, let $S = \{(m - 6)/4: m \in M\}$, and
suppose there exists a $\adfPENT(5, r_0)$.
Then, by Theorem~\ref{thm:GDD-basic}, there exists a pentagonal geometry $\adfPENT(5, u r_0 + s + 3u/2)$
for each $s \in S$ where there exists a $\adfPENT(5, s)$.
}
%%%%%%%%%%%%%%%%%%%%%%%%%%%%%%%%%%%%%%%%%%%%%%%%%%%%%%%%%%%%%%%%%%%%%%%%%%%%%%%%%%%%%%%%%%
\end{construction}

We give some examples of Construction~\ref{con:PENT-5-TD-construction-g-u-q} in Table~\ref{tab:PENT-5-TD-construction-g-u-q}
with $q$ restricted to prime powers.
For each $r$ in the column headed $\adfPENT(5,r)$, let $s = r - u r_0 - 3u/2$, $m = 4s + 6$ and $d = g(u - 1)/3$.
Then one can verify that $m = j d + (q - j)g$ for some $j \in \{0, 1, \dots, q\}$.
The existence of the 5-GDD of type $g^{u + 1}$ follows from \cite[Theorem 2.25]{WeiGe2014} or Lemma~\ref{lem:5-GDDs-direct},
and that of type $g^u d^1$ from \cite[Theorem 2.41]{WeiGe2014} and Lemma~\ref{lem:4-RGDD-to-5-GDD}.
For the $\adfPENT(5,r_0)$ and the $\adfPENT(5,s)$, see
Theorem~\ref{thm:PENT-5-0OLP-constructed} or earlier $\adfPENT(5,r)$ entries in Table~\ref{tab:PENT-5-TD-construction-g-u-q}.

It seems possible that Construction~\ref{con:PENT-5-TD-construction-g-u-q} combined with
Lemma~\ref{lem:PENT-5-r-direct}, Theorem~\ref{thm:PENT-5-constructed} and the available 5-GDDs of type $g^u m^1$
will generate $\adfPENT(5,r)$ geometries without opposite line pairs for
all except a small number of $r \equiv 0 \textrm{~or~} 1 \adfmod{5}$.
A proof would be desirable.

%%%%%%%%%%%%%%%%%%%%%%%%%%%%%%%%%%%%%%%%%%%%%%%%%%%%%%%%%%%%%%%%%%%%%%%%%%%%%%%%%%%%%%%%%%
%%%%%%%%%%%%%%%%%%%%%%%%%%%%%%%%%%%%%%%%%%%%%%%%%%%%%%%%%%%%%%%%%%%%%%%%%%%%%%%%%%%%%%%%%%
%%%%%%%%%%%%%%%%%%%%%%%%%%%%%%%%%%%%%%%%%%%%%%%%%%%%%%%%%%%%%%%%%%%%%%%%%%%%%%%%%%%%%%%%%%
%%%%%%%%%%%%%%%%%%%%%%%%%%%%%%%%%%%%%%%%%%%%%%%%%%%%%%%%%%%%%%%%%%%%%%%%%%%%%%%%%%%%%%%%%%

%%%
%%% Stuff requested by journal. Remove for now.
%%%
% \nocite{*}
% \bibliographystyle{amsplain}
% \bibliography{}

%%%%%%%%%%%%%%%%%%%%%%%%%%%%%%%%%%%%%%%%%%%%%%%%%%%%%%%%%%%%%%%%%%%%%%%%%%%%%%%%%%%%%%%%%%
%%%%%%%%%%%%%%%%%%%%%%%%%%%%%%%%%%%%%%%%%%%%%%%%%%%%%%%%%%%%%%%%%%%%%%%%%%%%%%%%%%%%%%%%%%
%%%%%%%%%%%%%%%%%%%%%%%%%%%%%%%%%%%%%%%%%%%%%%%%%%%%%%%%%%%%%%%%%%%%%%%%%%%%%%%%%%%%%%%%%%
%%%%%%%%%%%%%%%%%%%%%%%%%%%%%%%%%%%%%%%%%%%%%%%%%%%%%%%%%%%%%%%%%%%%%%%%%%%%%%%%%%%%%%%%%%

\appendix

\section{$\adfPENT(4,r)$ geometries for Lemma~\ref{lem:PENT-4-direct-extra}}
\label{app:PENT-4-direct-extra}

{\noindent\boldmath $\adfPENT(4, 29)$}~ With point set $Z_{92}$, the 667 lines are generated from

{\scriptsize
%adfPENTstart:1:4:29:4:4:
 $\{12, 45, 73, 80\}$, $\{20, 48, 59, 62\}$, $\{10, 33, 67, 86\}$, $\{30, 37, 39, 59\}$,\adfsplit
 $\{64, 44, 70, 61\}$, $\{90, 31, 16, 59\}$, $\{90, 79, 10, 44\}$, $\{67, 64, 25, 48\}$,\adfsplit
 $\{11, 87, 16, 82\}$, $\{24, 62, 34, 86\}$, $\{26, 53, 41, 67\}$, $\{4, 31, 63, 25\}$,\adfsplit
 $\{6, 31, 7, 44\}$, $\{19, 58, 26, 28\}$, $\{0, 7, 51, 52\}$, $\{0, 2, 8, 49\}$,\adfsplit
 $\{0, 23, 31, 48\}$, $\{0, 1, 56, 88\}$, $\{0, 13, 50, 70\}$, $\{0, 18, 63, 81\}$,\adfsplit
 $\{0, 22, 65, 78\}$, $\{0, 9, 34, 82\}$, $\{0, 25, 29, 77\}$, $\{1, 6, 17, 83\}$,\adfsplit
 $\{1, 2, 73, 90\}$, $\{1, 42, 79, 91\}$, $\{1, 34, 47, 51\}$, $\{1, 9, 19, 71\}$,\adfsplit
 $\{1, 25, 55, 61\}$
%adfPENTend
}

\noindent under the action of the mapping $x \mapsto x + 4 \adfmod{92}$.
The deficiency graph is connected and has girth 6.
% The opposite lines are generated by the first 4 blocks: True
% End of PENT(4, 29)
%%%%%%%%%%%%%%%%%%%%%%%%%%%%%%%%%%%%%%%%%%%%%%%%%%%%%%%%%%%%%%%%%%%%%%%%%%%%%%%%%%%%%%%%%%

{\noindent\boldmath $\adfPENT(4, 33)$}~ With point set $Z_{104}$, the 858 lines are generated from

{\scriptsize
%adfPENTstart:1:4:33:8:8:
 $\{33, 39, 93, 99\}$, $\{28, 62, 72, 101\}$, $\{10, 46, 55, 98\}$, $\{8, 13, 43, 67\}$,\adfsplit
 $\{36, 76, 81, 94\}$, $\{9, 16, 86, 99\}$, $\{20, 29, 49, 66\}$, $\{47, 58, 71, 72\}$,\adfsplit
 $\{42, 22, 37, 66\}$, $\{83, 15, 31, 20\}$, $\{76, 70, 27, 100\}$, $\{58, 56, 81, 44\}$,\adfsplit
 $\{58, 26, 22, 19\}$, $\{35, 96, 93, 94\}$, $\{5, 101, 69, 76\}$, $\{9, 23, 59, 21\}$,\adfsplit
 $\{59, 91, 71, 2\}$, $\{90, 16, 26, 8\}$, $\{67, 75, 0, 42\}$, $\{38, 57, 18, 49\}$,\adfsplit
 $\{44, 97, 35, 100\}$, $\{94, 68, 90, 32\}$, $\{96, 24, 58, 65\}$, $\{87, 83, 68, 5\}$,\adfsplit
 $\{60, 87, 64, 44\}$, $\{5, 21, 19, 2\}$, $\{0, 1, 26, 44\}$, $\{0, 3, 4, 98\}$,\adfsplit
 $\{1, 5, 50, 92\}$, $\{1, 11, 36, 42\}$, $\{0, 31, 50, 52\}$, $\{1, 2, 17, 53\}$,\adfsplit
 $\{1, 10, 31, 58\}$, $\{2, 28, 29, 51\}$, $\{2, 7, 14, 76\}$, $\{1, 34, 45, 84\}$,\adfsplit
 $\{2, 30, 31, 68\}$, $\{2, 45, 93, 103\}$, $\{2, 3, 37, 78\}$, $\{2, 36, 43, 85\}$,\adfsplit
 $\{2, 39, 83, 94\}$, $\{0, 11, 46, 68\}$, $\{0, 12, 20, 71\}$, $\{1, 12, 83, 99\}$,\adfsplit
 $\{4, 7, 54, 85\}$, $\{4, 14, 61, 95\}$, $\{1, 20, 22, 95\}$, $\{0, 13, 28, 56\}$,\adfsplit
 $\{1, 27, 47, 75\}$, $\{1, 6, 41, 59\}$, $\{3, 21, 30, 95\}$, $\{0, 6, 22, 27\}$,\adfsplit
 $\{0, 19, 38, 86\}$, $\{0, 15, 87, 91\}$, $\{0, 17, 49, 51\}$, $\{5, 22, 30, 102\}$,\adfsplit
 $\{0, 14, 78, 81\}$, $\{0, 30, 47, 85\}$, $\{1, 21, 39, 54\}$, $\{0, 21, 45, 57\}$,\adfsplit
 $\{0, 16, 53, 80\}$, $\{0, 9, 54, 95\}$, $\{1, 30, 55, 63\}$, $\{0, 63, 65, 89\}$,\adfsplit
 $\{1, 23, 29, 57\}$, $\{0, 7, 55, 69\}$
%adfPENTend
}

\noindent under the action of the mapping $x \mapsto x + 8 \adfmod{104}$.
The deficiency graph is connected and has girth 6.
% The opposite lines are generated by the first 8 blocks: True
% End of PENT(4, 33)
%%%%%%%%%%%%%%%%%%%%%%%%%%%%%%%%%%%%%%%%%%%%%%%%%%%%%%%%%%%%%%%%%%%%%%%%%%%%%%%%%%%%%%%%%%

{\noindent\boldmath $\adfPENT(4, 37)$}~ With point set $Z_{116}$, the 1073 lines are generated from

{\scriptsize
%adfPENTstart:1:4:37:4:4:
 $\{24, 70, 92, 107\}$, $\{13, 14, 67, 105\}$, $\{30, 48, 90, 105\}$, $\{12, 23, 53, 99\}$,\adfsplit
 $\{108, 32, 91, 98\}$, $\{56, 26, 9, 81\}$, $\{71, 52, 46, 13\}$, $\{0, 1, 56, 58\}$,\adfsplit
 $\{0, 3, 84, 112\}$, $\{0, 5, 72, 108\}$, $\{0, 6, 10, 100\}$, $\{0, 9, 17, 104\}$,\adfsplit
 $\{0, 14, 52, 75\}$, $\{0, 18, 27, 96\}$, $\{0, 33, 37, 89\}$, $\{0, 30, 31, 109\}$,\adfsplit
 $\{0, 34, 39, 105\}$, $\{0, 43, 45, 51\}$, $\{0, 54, 62, 65\}$, $\{0, 50, 63, 97\}$,\adfsplit
 $\{0, 55, 67, 90\}$, $\{0, 53, 71, 102\}$, $\{0, 74, 93, 114\}$, $\{0, 91, 95, 113\}$,\adfsplit
 $\{0, 81, 101, 103\}$, $\{0, 73, 82, 115\}$, $\{0, 79, 85, 111\}$, $\{1, 17, 82, 94\}$,\adfsplit
 $\{1, 29, 63, 77\}$, $\{1, 6, 33, 86\}$, $\{1, 15, 31, 75\}$, $\{1, 11, 37, 110\}$,\adfsplit
 $\{1, 30, 71, 107\}$, $\{2, 18, 75, 86\}$, $\{2, 22, 51, 66\}$, $\{2, 47, 71, 99\}$,\adfsplit
 $\{2, 23, 26, 91\}$
%adfPENTend
}

\noindent under the action of the mapping $x \mapsto x + 4 \adfmod{116}$.
The deficiency graph is connected and has girth 6.
% The opposite lines are generated by the first 4 blocks: True
% End of PENT(4, 37)
%%%%%%%%%%%%%%%%%%%%%%%%%%%%%%%%%%%%%%%%%%%%%%%%%%%%%%%%%%%%%%%%%%%%%%%%%%%%%%%%%%%%%%%%%%

{\noindent\boldmath $\adfPENT(4, 40)$}~ With point set $Z_{125}$, the 1250 lines are generated from

{\scriptsize
%adfPENTstart:1:4:40:5:5:
 $\{18, 55, 70, 116\}$, $\{10, 14, 87, 92\}$, $\{36, 41, 69, 88\}$, $\{34, 42, 69, 110\}$,\adfsplit
 $\{62, 63, 98, 116\}$, $\{81, 107, 104, 41\}$, $\{62, 30, 119, 71\}$, $\{48, 57, 7, 38\}$,\adfsplit
 $\{100, 23, 27, 78\}$, $\{26, 71, 24, 63\}$, $\{100, 81, 95, 103\}$, $\{58, 15, 60, 102\}$,\adfsplit
 $\{21, 33, 36, 8\}$, $\{3, 114, 18, 32\}$, $\{98, 42, 27, 9\}$, $\{106, 10, 49, 89\}$,\adfsplit
 $\{20, 48, 42, 14\}$, $\{82, 2, 100, 72\}$, $\{124, 87, 51, 24\}$, $\{36, 60, 4, 53\}$,\adfsplit
 $\{92, 27, 25, 46\}$, $\{26, 19, 88, 115\}$, $\{10, 91, 122, 109\}$, $\{13, 72, 114, 92\}$,\adfsplit
 $\{2, 42, 86, 111\}$, $\{0, 1, 12, 102\}$, $\{0, 7, 37, 115\}$, $\{0, 68, 92, 117\}$,\adfsplit
 $\{1, 2, 9, 104\}$, $\{1, 62, 78, 83\}$, $\{1, 22, 36, 96\}$, $\{0, 34, 72, 76\}$,\adfsplit
 $\{1, 7, 33, 39\}$, $\{1, 43, 77, 89\}$, $\{2, 4, 13, 74\}$, $\{0, 26, 60, 122\}$,\adfsplit
 $\{2, 19, 63, 79\}$, $\{0, 6, 16, 108\}$, $\{0, 11, 66, 86\}$, $\{1, 3, 68, 88\}$,\adfsplit
 $\{0, 53, 54, 121\}$, $\{0, 44, 56, 59\}$, $\{0, 31, 38, 85\}$, $\{0, 51, 94, 104\}$,\adfsplit
 $\{0, 23, 74, 90\}$, $\{0, 13, 24, 25\}$, $\{0, 20, 50, 83\}$, $\{0, 14, 19, 64\}$,\adfsplit
 $\{0, 9, 93, 114\}$, $\{3, 29, 33, 83\}$
%adfPENTend
}

\noindent under the action of the mapping $x \mapsto x + 5 \adfmod{125}$.
The deficiency graph is connected and has girth 7.
% The opposite lines are generated by the first 5 blocks: True
% End of PENT(4, 40)
%%%%%%%%%%%%%%%%%%%%%%%%%%%%%%%%%%%%%%%%%%%%%%%%%%%%%%%%%%%%%%%%%%%%%%%%%%%%%%%%%%%%%%%%%%

{\noindent\boldmath $\adfPENT(4, 45)$}~ With point set $Z_{140}$, the 1575 lines are generated from

{\scriptsize
%adfPENTstart:1:4:45:4:4:
 $\{15, 66, 68, 72\}$, $\{37, 70, 103, 105\}$, $\{46, 73, 76, 98\}$, $\{41, 63, 83, 128\}$,\adfsplit
 $\{108, 89, 4, 110\}$, $\{47, 44, 133, 99\}$, $\{122, 40, 0, 102\}$, $\{80, 33, 134, 64\}$,\adfsplit
 $\{75, 130, 16, 80\}$, $\{110, 25, 16, 45\}$, $\{138, 85, 122, 109\}$, $\{99, 135, 51, 117\}$,\adfsplit
 $\{64, 52, 35, 29\}$, $\{83, 94, 70, 2\}$, $\{83, 120, 14, 113\}$, $\{8, 109, 92, 65\}$,\adfsplit
 $\{68, 27, 137, 25\}$, $\{79, 119, 47, 21\}$, $\{79, 58, 94, 52\}$, $\{1, 2, 93, 12\}$,\adfsplit
 $\{93, 81, 29, 74\}$, $\{45, 131, 127, 54\}$, $\{126, 119, 0, 23\}$, $\{48, 115, 139, 125\}$,\adfsplit
 $\{133, 72, 71, 90\}$, $\{0, 1, 5, 96\}$, $\{0, 13, 73, 107\}$, $\{1, 11, 33, 126\}$,\adfsplit
 $\{1, 17, 78, 101\}$, $\{0, 26, 31, 37\}$, $\{1, 63, 71, 110\}$, $\{0, 25, 98, 115\}$,\adfsplit
 $\{0, 14, 65, 80\}$, $\{0, 33, 41, 79\}$, $\{1, 18, 58, 115\}$, $\{0, 7, 21, 71\}$,\adfsplit
 $\{1, 6, 82, 138\}$, $\{0, 78, 90, 131\}$, $\{2, 6, 51, 67\}$, $\{0, 51, 63, 86\}$,\adfsplit
 $\{0, 10, 92, 112\}$, $\{0, 30, 39, 52\}$, $\{0, 19, 46, 47\}$, $\{0, 8, 32, 43\}$,\adfsplit
 $\{2, 31, 34, 62\}$
%adfPENTend
}

\noindent under the action of the mapping $x \mapsto x + 4 \adfmod{140}$.
The deficiency graph is connected and has girth 7.
% The opposite lines are generated by the first 4 blocks: True
% End of PENT(4, 45)
%%%%%%%%%%%%%%%%%%%%%%%%%%%%%%%%%%%%%%%%%%%%%%%%%%%%%%%%%%%%%%%%%%%%%%%%%%%%%%%%%%%%%%%%%%

{\noindent\boldmath $\adfPENT(4, 49)$}~ With point set $Z_{152}$, the 1862 lines are generated from

{\scriptsize
%adfPENTstart:1:4:49:8:8:
 $\{38, 43, 75, 86\}$, $\{50, 58, 100, 117\}$, $\{97, 105, 110, 140\}$, $\{29, 80, 110, 112\}$,\adfsplit
 $\{18, 57, 101, 119\}$, $\{15, 41, 60, 131\}$, $\{50, 51, 72, 120\}$, $\{44, 63, 103, 149\}$,\adfsplit
 $\{73, 28, 46, 139\}$, $\{10, 14, 78, 137\}$, $\{145, 53, 111, 103\}$, $\{112, 69, 140, 55\}$,\adfsplit
 $\{23, 128, 110, 12\}$, $\{113, 84, 146, 88\}$, $\{72, 33, 114, 69\}$, $\{92, 95, 68, 66\}$,\adfsplit
 $\{106, 124, 63, 14\}$, $\{138, 29, 118, 133\}$, $\{104, 94, 125, 113\}$, $\{9, 110, 133, 54\}$,\adfsplit
 $\{8, 52, 145, 84\}$, $\{114, 137, 46, 133\}$, $\{118, 150, 83, 27\}$, $\{16, 108, 133, 94\}$,\adfsplit
 $\{120, 39, 93, 38\}$, $\{77, 64, 115, 36\}$, $\{55, 133, 57, 35\}$, $\{147, 85, 106, 138\}$,\adfsplit
 $\{71, 88, 55, 90\}$, $\{102, 61, 69, 85\}$, $\{104, 57, 91, 9\}$, $\{116, 20, 107, 126\}$,\adfsplit
 $\{108, 5, 123, 105\}$, $\{0, 3, 45, 77\}$, $\{0, 1, 29, 85\}$, $\{0, 4, 61, 133\}$,\adfsplit
 $\{0, 18, 53, 140\}$, $\{0, 5, 7, 20\}$, $\{1, 3, 21, 133\}$, $\{1, 2, 5, 93\}$,\adfsplit
 $\{0, 15, 37, 94\}$, $\{0, 23, 46, 93\}$, $\{2, 7, 118, 125\}$, $\{1, 18, 69, 118\}$,\adfsplit
 $\{0, 14, 26, 141\}$, $\{1, 15, 141, 150\}$, $\{2, 12, 13, 14\}$, $\{1, 12, 28, 101\}$,\adfsplit
 $\{1, 27, 109, 116\}$, $\{2, 27, 29, 60\}$, $\{2, 79, 108, 141\}$, $\{3, 7, 13, 115\}$,\adfsplit
 $\{3, 11, 60, 69\}$, $\{2, 39, 77, 83\}$, $\{3, 79, 109, 148\}$, $\{3, 15, 125, 139\}$,\adfsplit
 $\{0, 6, 68, 108\}$, $\{0, 12, 17, 100\}$, $\{0, 35, 52, 60\}$, $\{1, 31, 52, 132\}$,\adfsplit
 $\{0, 50, 84, 132\}$, $\{0, 11, 110, 116\}$, $\{2, 19, 22, 148\}$, $\{1, 22, 43, 140\}$,\adfsplit
 $\{2, 23, 38, 132\}$, $\{1, 68, 103, 142\}$, $\{2, 30, 46, 116\}$, $\{1, 38, 76, 107\}$,\adfsplit
 $\{2, 15, 68, 107\}$, $\{2, 54, 76, 151\}$, $\{2, 95, 99, 100\}$, $\{0, 22, 31, 102\}$,\adfsplit
 $\{3, 30, 135, 142\}$, $\{1, 33, 86, 110\}$, $\{0, 118, 126, 151\}$, $\{0, 54, 111, 143\}$,\adfsplit
 $\{2, 51, 75, 126\}$, $\{1, 39, 70, 146\}$, $\{0, 62, 66, 122\}$, $\{2, 35, 87, 123\}$,\adfsplit
 $\{1, 23, 47, 95\}$, $\{0, 8, 87, 115\}$, $\{1, 87, 99, 147\}$, $\{0, 27, 88, 127\}$,\adfsplit
 $\{2, 55, 82, 143\}$, $\{1, 11, 25, 79\}$, $\{1, 7, 17, 89\}$, $\{0, 55, 73, 123\}$,\adfsplit
 $\{0, 34, 40, 103\}$, $\{0, 59, 74, 98\}$, $\{0, 81, 90, 121\}$, $\{1, 42, 90, 106\}$,\adfsplit
 $\{0, 33, 89, 147\}$, $\{0, 19, 99, 106\}$, $\{1, 75, 98, 138\}$, $\{0, 49, 56, 72\}$,\adfsplit
 $\{0, 10, 67, 97\}$, $\{0, 24, 65, 138\}$
%adfPENTend
}

\noindent under the action of the mapping $x \mapsto x + 8 \adfmod{152}$.
The deficiency graph is connected and has girth 6.
% The opposite lines are generated by the first 8 blocks: True
% End of PENT(4, 49)
%%%%%%%%%%%%%%%%%%%%%%%%%%%%%%%%%%%%%%%%%%%%%%%%%%%%%%%%%%%%%%%%%%%%%%%%%%%%%%%%%%%%%%%%%%

{\noindent\boldmath $\adfPENT(4, 52)$}~ With point set $Z_{161}$, the 2093 lines are generated from

{\scriptsize
%adfPENTstart:1:4:52:7:7:
 $\{14, 81, 143, 147\}$, $\{93, 136, 139, 149\}$, $\{4, 15, 71, 108\}$, $\{21, 29, 33, 58\}$,\adfsplit
 $\{2, 40, 84, 117\}$, $\{53, 90, 130, 136\}$, $\{29, 41, 82, 132\}$, $\{58, 125, 93, 144\}$,\adfsplit
 $\{67, 54, 105, 149\}$, $\{96, 107, 95, 125\}$, $\{4, 74, 92, 155\}$, $\{48, 26, 46, 69\}$,\adfsplit
 $\{17, 113, 36, 52\}$, $\{34, 141, 66, 38\}$, $\{39, 70, 24, 118\}$, $\{114, 27, 92, 57\}$,\adfsplit
 $\{149, 40, 22, 49\}$, $\{46, 88, 28, 23\}$, $\{131, 73, 139, 58\}$, $\{101, 3, 150, 126\}$,\adfsplit
 $\{158, 94, 7, 4\}$, $\{111, 13, 84, 28\}$, $\{83, 97, 44, 140\}$, $\{10, 37, 57, 132\}$,\adfsplit
 $\{155, 60, 8, 115\}$, $\{144, 160, 44, 143\}$, $\{156, 157, 130, 75\}$, $\{151, 101, 108, 79\}$,\adfsplit
 $\{7, 54, 132, 87\}$, $\{69, 142, 11, 5\}$, $\{79, 76, 32, 105\}$, $\{124, 142, 68, 73\}$,\adfsplit
 $\{149, 147, 140, 145\}$, $\{152, 57, 58, 7\}$, $\{18, 146, 140, 105\}$, $\{80, 14, 59, 156\}$,\adfsplit
 $\{55, 56, 19, 6\}$, $\{0, 1, 10, 20\}$, $\{0, 3, 34, 62\}$, $\{0, 4, 13, 69\}$,\adfsplit
 $\{0, 17, 101, 153\}$, $\{1, 3, 45, 146\}$, $\{0, 16, 52, 76\}$, $\{0, 15, 55, 139\}$,\adfsplit
 $\{0, 29, 97, 108\}$, $\{1, 4, 52, 160\}$, $\{1, 9, 111, 118\}$, $\{1, 29, 65, 90\}$,\adfsplit
 $\{1, 6, 33, 48\}$, $\{1, 16, 26, 153\}$, $\{1, 34, 40, 102\}$, $\{1, 72, 79, 83\}$,\adfsplit
 $\{2, 11, 19, 125\}$, $\{2, 5, 59, 146\}$, $\{1, 27, 75, 150\}$, $\{3, 19, 76, 130\}$,\adfsplit
 $\{3, 11, 32, 83\}$, $\{2, 33, 61, 132\}$, $\{4, 47, 138, 139\}$, $\{4, 18, 104, 116\}$,\adfsplit
 $\{0, 75, 122, 128\}$, $\{0, 59, 61, 150\}$, $\{0, 23, 73, 107\}$, $\{1, 44, 86, 157\}$,\adfsplit
 $\{2, 10, 51, 95\}$, $\{0, 30, 93, 121\}$, $\{1, 32, 135, 156\}$, $\{2, 80, 103, 124\}$,\adfsplit
 $\{0, 64, 94, 114\}$, $\{2, 16, 123, 158\}$, $\{0, 43, 77, 149\}$, $\{0, 53, 86, 131\}$,\adfsplit
 $\{1, 8, 121, 137\}$, $\{0, 26, 58, 120\}$, $\{0, 11, 31, 68\}$, $\{1, 24, 46, 117\}$,\adfsplit
 $\{3, 39, 68, 108\}$, $\{1, 88, 116, 129\}$, $\{1, 47, 73, 82\}$, $\{3, 103, 145, 152\}$,\adfsplit
 $\{1, 18, 123, 138\}$, $\{0, 40, 63, 109\}$, $\{1, 22, 60, 145\}$, $\{0, 32, 54, 85\}$,\adfsplit
 $\{0, 42, 144, 148\}$, $\{0, 21, 57, 99\}$, $\{0, 19, 91, 116\}$, $\{0, 103, 113, 137\}$,\adfsplit
 $\{0, 22, 92, 112\}$, $\{0, 82, 96, 155\}$, $\{1, 12, 110, 113\}$
%adfPENTend
}

\noindent under the action of the mapping $x \mapsto x + 7 \adfmod{161}$.
The deficiency graph is connected and has girth 7.
% The opposite lines are generated by the first 7 blocks: True
% End of PENT(4, 52)
%%%%%%%%%%%%%%%%%%%%%%%%%%%%%%%%%%%%%%%%%%%%%%%%%%%%%%%%%%%%%%%%%%%%%%%%%%%%%%%%%%%%%%%%%%

{\noindent\boldmath $\adfPENT(4, 53)$}~ With point set $Z_{164}$, the 2173 lines are generated from

{\scriptsize
%adfPENTstart:1:4:53:4:4:
 $\{2, 12, 121, 152\}$, $\{17, 44, 123, 149\}$, $\{0, 30, 138, 151\}$, $\{18, 45, 55, 115\}$,\adfsplit
 $\{136, 58, 22, 149\}$, $\{106, 148, 60, 25\}$, $\{82, 91, 17, 6\}$, $\{25, 69, 103, 83\}$,\adfsplit
 $\{35, 14, 62, 116\}$, $\{20, 112, 38, 122\}$, $\{78, 157, 102, 155\}$, $\{50, 49, 154, 91\}$,\adfsplit
 $\{16, 97, 64, 67\}$, $\{122, 126, 141, 67\}$, $\{128, 163, 88, 82\}$, $\{123, 76, 55, 39\}$,\adfsplit
 $\{79, 13, 25, 73\}$, $\{119, 2, 117, 125\}$, $\{74, 21, 82, 121\}$, $\{90, 112, 52, 48\}$,\adfsplit
 $\{6, 150, 35, 24\}$, $\{141, 75, 123, 140\}$, $\{86, 46, 14, 103\}$, $\{155, 31, 92, 62\}$,\adfsplit
 $\{128, 147, 154, 60\}$, $\{0, 5, 82, 98\}$, $\{0, 6, 7, 58\}$, $\{0, 9, 34, 130\}$,\adfsplit
 $\{1, 5, 34, 134\}$, $\{0, 31, 54, 66\}$, $\{0, 8, 69, 78\}$, $\{0, 15, 23, 114\}$,\adfsplit
 $\{0, 22, 25, 36\}$, $\{1, 6, 29, 69\}$, $\{1, 14, 19, 58\}$, $\{0, 55, 74, 125\}$,\adfsplit
 $\{0, 65, 141, 162\}$, $\{0, 95, 106, 131\}$, $\{1, 87, 90, 151\}$, $\{1, 39, 63, 158\}$,\adfsplit
 $\{1, 18, 51, 113\}$, $\{0, 27, 59, 115\}$, $\{0, 93, 119, 123\}$, $\{1, 25, 143, 155\}$,\adfsplit
 $\{0, 21, 43, 136\}$, $\{0, 39, 44, 89\}$, $\{0, 37, 117, 163\}$, $\{0, 67, 111, 139\}$,\adfsplit
 $\{0, 16, 107, 161\}$, $\{0, 17, 32, 84\}$, $\{0, 57, 77, 113\}$, $\{0, 20, 73, 155\}$,\adfsplit
 $\{0, 29, 101, 108\}$
%adfPENTend
}

\noindent under the action of the mapping $x \mapsto x + 4 \adfmod{164}$.
The deficiency graph is connected and has girth 7.
% The opposite lines are generated by the first 4 blocks: True
% End of PENT(4, 53)
%%%%%%%%%%%%%%%%%%%%%%%%%%%%%%%%%%%%%%%%%%%%%%%%%%%%%%%%%%%%%%%%%%%%%%%%%%%%%%%%%%%%%%%%%%

{\noindent\boldmath $\adfPENT(4, 60)$}~ With point set $Z_{185}$, the 2775 lines are generated from

{\scriptsize
%adfPENTstart:1:4:60:5:5:
 $\{21, 158, 167, 184\}$, $\{78, 98, 102, 165\}$, $\{20, 86, 108, 139\}$, $\{30, 82, 91, 111\}$,\adfsplit
 $\{5, 29, 52, 164\}$, $\{105, 5, 100, 59\}$, $\{156, 172, 23, 70\}$, $\{52, 11, 65, 57\}$,\adfsplit
 $\{171, 25, 114, 117\}$, $\{166, 38, 168, 21\}$, $\{88, 25, 37, 19\}$, $\{180, 31, 36, 119\}$,\adfsplit
 $\{43, 114, 11, 150\}$, $\{95, 143, 125, 151\}$, $\{96, 28, 172, 13\}$, $\{148, 137, 143, 14\}$,\adfsplit
 $\{119, 11, 109, 80\}$, $\{69, 77, 99, 132\}$, $\{29, 136, 103, 16\}$, $\{72, 93, 166, 104\}$,\adfsplit
 $\{78, 108, 49, 160\}$, $\{29, 82, 179, 124\}$, $\{40, 85, 150, 154\}$, $\{148, 75, 106, 169\}$,\adfsplit
 $\{135, 71, 143, 150\}$, $\{0, 1, 9, 115\}$, $\{0, 2, 6, 50\}$, $\{0, 3, 17, 175\}$,\adfsplit
 $\{0, 7, 80, 96\}$, $\{0, 14, 35, 165\}$, $\{0, 11, 40, 59\}$, $\{0, 22, 23, 25\}$,\adfsplit
 $\{0, 28, 60, 93\}$, $\{0, 32, 42, 99\}$, $\{0, 38, 44, 143\}$, $\{0, 37, 49, 83\}$,\adfsplit
 $\{0, 53, 72, 163\}$, $\{0, 43, 68, 168\}$, $\{0, 57, 64, 128\}$, $\{0, 67, 82, 174\}$,\adfsplit
 $\{0, 51, 84, 101\}$, $\{0, 76, 97, 113\}$, $\{0, 62, 109, 123\}$, $\{0, 58, 104, 148\}$,\adfsplit
 $\{0, 87, 107, 132\}$, $\{0, 77, 91, 152\}$, $\{0, 108, 126, 157\}$, $\{0, 127, 129, 161\}$,\adfsplit
 $\{0, 142, 169, 173\}$, $\{0, 117, 166, 176\}$, $\{0, 136, 151, 154\}$, $\{0, 111, 147, 171\}$,\adfsplit
 $\{0, 134, 162, 181\}$, $\{1, 2, 81, 87\}$, $\{1, 12, 82, 117\}$, $\{2, 62, 143, 144\}$,\adfsplit
 $\{2, 54, 67, 133\}$, $\{1, 52, 108, 142\}$, $\{2, 52, 129, 148\}$, $\{1, 27, 67, 143\}$,\adfsplit
 $\{1, 28, 56, 112\}$, $\{2, 32, 104, 149\}$, $\{2, 39, 88, 158\}$, $\{1, 8, 72, 183\}$,\adfsplit
 $\{1, 24, 63, 122\}$, $\{3, 44, 84, 138\}$, $\{1, 36, 83, 128\}$, $\{1, 13, 74, 163\}$,\adfsplit
 $\{1, 133, 149, 173\}$, $\{1, 39, 123, 159\}$, $\{1, 59, 68, 144\}$, $\{1, 26, 71, 101\}$,\adfsplit
 $\{1, 94, 114, 174\}$, $\{1, 69, 179, 184\}$, $\{1, 29, 44, 96\}$
%adfPENTend
}

\noindent under the action of the mapping $x \mapsto x + 5 \adfmod{185}$.
The deficiency graph is connected and has girth 7.
% The opposite lines are generated by the first 5 blocks: True
% End of PENT(4, 60)
%%%%%%%%%%%%%%%%%%%%%%%%%%%%%%%%%%%%%%%%%%%%%%%%%%%%%%%%%%%%%%%%%%%%%%%%%%%%%%%%%%%%%%%%%%

{\noindent\boldmath $\adfPENT(4, 61)$}~ With point set $Z_{188}$, the 2867 lines are generated from

{\scriptsize
%adfPENTstart:1:4:61:4:4:
 $\{27, 71, 74, 177\}$, $\{12, 21, 75, 169\}$, $\{26, 83, 116, 166\}$, $\{110, 117, 120, 164\}$,\adfsplit
 $\{39, 176, 16, 178\}$, $\{106, 123, 62, 19\}$, $\{144, 180, 135, 76\}$, $\{124, 172, 107, 9\}$,\adfsplit
 $\{66, 154, 37, 29\}$, $\{86, 47, 70, 179\}$, $\{163, 64, 158, 49\}$, $\{0, 1, 7, 127\}$,\adfsplit
 $\{0, 3, 4, 83\}$, $\{0, 5, 15, 47\}$, $\{0, 6, 31, 39\}$, $\{0, 8, 111, 115\}$,\adfsplit
 $\{0, 10, 55, 147\}$, $\{0, 11, 139, 159\}$, $\{0, 13, 75, 91\}$, $\{0, 14, 67, 183\}$,\adfsplit
 $\{0, 17, 135, 163\}$, $\{0, 19, 119, 131\}$, $\{0, 18, 151, 175\}$, $\{0, 29, 43, 167\}$,\adfsplit
 $\{0, 22, 35, 87\}$, $\{1, 2, 23, 71\}$, $\{1, 3, 5, 155\}$, $\{0, 49, 95, 154\}$,\adfsplit
 $\{1, 6, 18, 91\}$, $\{1, 10, 51, 62\}$, $\{1, 13, 31, 46\}$, $\{1, 14, 135, 161\}$,\adfsplit
 $\{1, 17, 53, 183\}$, $\{1, 22, 59, 141\}$, $\{1, 26, 35, 157\}$, $\{1, 25, 111, 125\}$,\adfsplit
 $\{1, 27, 57, 138\}$, $\{1, 54, 69, 171\}$, $\{1, 66, 150, 179\}$, $\{1, 142, 143, 178\}$,\adfsplit
 $\{1, 123, 130, 186\}$, $\{2, 10, 99, 126\}$, $\{0, 12, 33, 113\}$, $\{0, 16, 53, 145\}$,\adfsplit
 $\{0, 20, 61, 137\}$, $\{0, 24, 89, 149\}$, $\{0, 34, 77, 121\}$, $\{0, 57, 92, 161\}$,\adfsplit
 $\{1, 73, 166, 170\}$, $\{0, 54, 85, 174\}$, $\{0, 81, 130, 158\}$, $\{0, 45, 118, 124\}$,\adfsplit
 $\{0, 42, 97, 116\}$, $\{0, 70, 105, 112\}$, $\{0, 32, 90, 165\}$, $\{0, 26, 66, 93\}$,\adfsplit
 $\{0, 40, 122, 142\}$, $\{0, 38, 88, 150\}$, $\{0, 46, 126, 128\}$, $\{0, 78, 110, 170\}$,\adfsplit
 $\{0, 30, 52, 132\}$
%adfPENTend
}

\noindent under the action of the mapping $x \mapsto x + 4 \adfmod{188}$.
The deficiency graph is connected and has girth 7.
% The opposite lines are generated by the first 4 blocks: True
% End of PENT(4, 61)
%%%%%%%%%%%%%%%%%%%%%%%%%%%%%%%%%%%%%%%%%%%%%%%%%%%%%%%%%%%%%%%%%%%%%%%%%%%%%%%%%%%%%%%%%%

{\noindent\boldmath $\adfPENT(4, 65)$}~ With point set $Z_{200}$, the 3250 lines are generated from

{\scriptsize
%adfPENTstart:1:4:65:8:8:
 $\{42, 178, 188, 197\}$, $\{3, 81, 90, 121\}$, $\{24, 100, 113, 160\}$, $\{1, 23, 37, 143\}$,\adfsplit
 $\{6, 16, 63, 106\}$, $\{8, 37, 171, 173\}$, $\{4, 102, 110, 151\}$, $\{62, 67, 148, 187\}$,\adfsplit
 $\{48, 139, 113, 147\}$, $\{118, 46, 112, 166\}$, $\{51, 27, 36, 141\}$, $\{118, 81, 60, 195\}$,\adfsplit
 $\{56, 14, 41, 179\}$, $\{178, 57, 175, 62\}$, $\{30, 59, 189, 64\}$, $\{59, 46, 178, 147\}$,\adfsplit
 $\{87, 9, 21, 26\}$, $\{99, 103, 145, 141\}$, $\{63, 138, 69, 109\}$, $\{153, 142, 151, 2\}$,\adfsplit
 $\{34, 68, 161, 146\}$, $\{106, 188, 36, 152\}$, $\{97, 7, 42, 174\}$, $\{191, 45, 102, 58\}$,\adfsplit
 $\{166, 135, 48, 83\}$, $\{110, 134, 135, 117\}$, $\{67, 116, 140, 111\}$, $\{86, 100, 48, 108\}$,\adfsplit
 $\{70, 120, 135, 163\}$, $\{88, 157, 100, 29\}$, $\{180, 165, 18, 196\}$, $\{93, 88, 49, 20\}$,\adfsplit
 $\{40, 67, 195, 175\}$, $\{131, 108, 64, 125\}$, $\{154, 41, 51, 187\}$, $\{16, 193, 96, 25\}$,\adfsplit
 $\{30, 82, 75, 136\}$, $\{114, 42, 67, 122\}$, $\{106, 22, 7, 136\}$, $\{4, 45, 143, 100\}$,\adfsplit
 $\{61, 62, 152, 126\}$, $\{84, 114, 179, 116\}$, $\{22, 180, 173, 95\}$, $\{61, 2, 111, 119\}$,\adfsplit
 $\{109, 61, 117, 160\}$, $\{0, 1, 59, 75\}$, $\{0, 2, 51, 107\}$, $\{0, 4, 11, 115\}$,\adfsplit
 $\{0, 7, 19, 171\}$, $\{0, 10, 147, 187\}$, $\{0, 3, 14, 30\}$, $\{1, 2, 19, 187\}$,\adfsplit
 $\{0, 17, 83, 126\}$, $\{0, 46, 77, 179\}$, $\{0, 13, 130, 131\}$, $\{1, 4, 98, 107\}$,\adfsplit
 $\{1, 5, 51, 68\}$, $\{1, 9, 99, 109\}$, $\{1, 7, 165, 195\}$, $\{1, 20, 53, 179\}$,\adfsplit
 $\{1, 14, 125, 163\}$, $\{1, 22, 69, 131\}$, $\{1, 28, 29, 171\}$, $\{1, 12, 43, 102\}$,\adfsplit
 $\{2, 6, 18, 91\}$, $\{2, 7, 37, 59\}$, $\{2, 22, 39, 43\}$, $\{2, 69, 163, 182\}$,\adfsplit
 $\{2, 15, 123, 141\}$, $\{3, 4, 71, 142\}$, $\{2, 44, 62, 131\}$, $\{3, 28, 79, 189\}$,\adfsplit
 $\{3, 29, 31, 148\}$, $\{3, 6, 87, 150\}$, $\{3, 53, 63, 135\}$, $\{3, 36, 125, 149\}$,\adfsplit
 $\{3, 103, 124, 127\}$, $\{3, 119, 156, 167\}$, $\{3, 132, 159, 166\}$, $\{0, 16, 53, 133\}$,\adfsplit
 $\{0, 33, 173, 189\}$, $\{0, 22, 28, 93\}$, $\{1, 15, 94, 189\}$, $\{0, 58, 85, 181\}$,\adfsplit
 $\{1, 21, 30, 133\}$, $\{0, 20, 45, 182\}$, $\{1, 26, 62, 117\}$, $\{1, 39, 61, 86\}$,\adfsplit
 $\{1, 25, 158, 173\}$, $\{0, 21, 102, 106\}$, $\{2, 13, 143, 170\}$, $\{1, 77, 124, 198\}$,\adfsplit
 $\{1, 60, 95, 181\}$, $\{1, 85, 106, 118\}$, $\{2, 26, 52, 101\}$, $\{2, 5, 42, 98\}$,\adfsplit
 $\{1, 31, 54, 93\}$, $\{2, 30, 53, 76\}$, $\{4, 101, 135, 174\}$, $\{2, 78, 156, 166\}$,\adfsplit
 $\{1, 58, 150, 182\}$, $\{1, 100, 126, 166\}$, $\{0, 68, 148, 198\}$, $\{0, 18, 25, 70\}$,\adfsplit
 $\{2, 79, 92, 174\}$, $\{0, 78, 122, 175\}$, $\{0, 62, 116, 121\}$, $\{0, 34, 73, 142\}$,\adfsplit
 $\{0, 108, 114, 172\}$, $\{0, 55, 100, 156\}$, $\{0, 74, 92, 112\}$, $\{0, 124, 138, 186\}$,\adfsplit
 $\{1, 36, 76, 148\}$, $\{1, 44, 113, 178\}$, $\{0, 57, 105, 196\}$, $\{0, 66, 111, 137\}$,\adfsplit
 $\{1, 55, 87, 164\}$, $\{1, 42, 63, 156\}$, $\{0, 81, 164, 183\}$, $\{1, 52, 127, 137\}$,\adfsplit
 $\{1, 17, 73, 154\}$, $\{0, 26, 144, 193\}$, $\{0, 41, 145, 191\}$, $\{0, 79, 95, 113\}$,\adfsplit
 $\{0, 39, 103, 169\}$, $\{0, 31, 50, 143\}$, $\{2, 31, 87, 191\}$, $\{0, 63, 96, 194\}$,\adfsplit
 $\{0, 23, 24, 72\}$, $\{0, 8, 127, 168\}$
%adfPENTend
}

\noindent under the action of the mapping $x \mapsto x + 8 \adfmod{200}$.
The deficiency graph is connected and has girth 5.
% The opposite lines are generated by the first 8 blocks: True
% End of PENT(4, 65)
%%%%%%%%%%%%%%%%%%%%%%%%%%%%%%%%%%%%%%%%%%%%%%%%%%%%%%%%%%%%%%%%%%%%%%%%%%%%%%%%%%%%%%%%%%

{\noindent\boldmath $\adfPENT(4, 69)$}~ With point set $Z_{212}$, the 3657 lines are generated from

{\scriptsize
%adfPENTstart:1:4:69:4:4:
 $\{22, 117, 169, 195\}$, $\{26, 44, 75, 96\}$, $\{78, 138, 189, 192\}$, $\{20, 23, 141, 195\}$,\adfsplit
 $\{92, 13, 134, 119\}$, $\{79, 54, 195, 90\}$, $\{63, 119, 162, 125\}$, $\{138, 68, 24, 77\}$,\adfsplit
 $\{131, 189, 66, 69\}$, $\{140, 65, 60, 47\}$, $\{62, 178, 137, 130\}$, $\{53, 186, 93, 205\}$,\adfsplit
 $\{189, 153, 200, 7\}$, $\{69, 46, 101, 120\}$, $\{91, 139, 196, 110\}$, $\{19, 110, 123, 72\}$,\adfsplit
 $\{127, 42, 115, 38\}$, $\{171, 104, 89, 62\}$, $\{35, 108, 154, 82\}$, $\{129, 109, 33, 151\}$,\adfsplit
 $\{34, 134, 105, 140\}$, $\{89, 88, 138, 55\}$, $\{169, 30, 178, 54\}$, $\{119, 209, 20, 60\}$,\adfsplit
 $\{181, 41, 55, 122\}$, $\{95, 40, 56, 149\}$, $\{66, 10, 63, 113\}$, $\{1, 174, 176, 99\}$,\adfsplit
 $\{136, 13, 9, 210\}$, $\{0, 4, 17, 25\}$, $\{0, 6, 41, 205\}$, $\{0, 2, 33, 181\}$,\adfsplit
 $\{0, 7, 77, 185\}$, $\{0, 8, 57, 81\}$, $\{0, 15, 45, 125\}$, $\{0, 29, 84, 153\}$,\adfsplit
 $\{0, 18, 105, 173\}$, $\{1, 3, 157, 169\}$, $\{0, 12, 113, 141\}$, $\{1, 2, 54, 197\}$,\adfsplit
 $\{0, 23, 61, 145\}$, $\{0, 14, 97, 118\}$, $\{0, 63, 65, 178\}$, $\{1, 51, 106, 146\}$,\adfsplit
 $\{1, 6, 7, 86\}$, $\{1, 11, 19, 46\}$, $\{1, 39, 127, 150\}$, $\{1, 111, 115, 187\}$,\adfsplit
 $\{1, 131, 147, 191\}$, $\{1, 71, 171, 199\}$, $\{1, 135, 167, 203\}$, $\{1, 78, 139, 170\}$,\adfsplit
 $\{0, 11, 75, 176\}$, $\{0, 20, 91, 143\}$, $\{0, 82, 115, 207\}$, $\{0, 35, 86, 167\}$,\adfsplit
 $\{0, 66, 87, 150\}$, $\{0, 34, 79, 103\}$, $\{0, 10, 32, 163\}$, $\{0, 26, 43, 72\}$,\adfsplit
 $\{2, 11, 18, 186\}$, $\{0, 24, 88, 171\}$, $\{0, 60, 154, 211\}$, $\{0, 30, 62, 112\}$,\adfsplit
 $\{0, 76, 198, 203\}$, $\{0, 90, 102, 110\}$, $\{0, 54, 95, 120\}$, $\{0, 48, 104, 182\}$,\adfsplit
 $\{0, 58, 68, 96\}$
%adfPENTend
}

\noindent under the action of the mapping $x \mapsto x + 4 \adfmod{212}$.
The deficiency graph is connected and has girth 7.
% The opposite lines are generated by the first 4 blocks: True
% End of PENT(4, 69)
%%%%%%%%%%%%%%%%%%%%%%%%%%%%%%%%%%%%%%%%%%%%%%%%%%%%%%%%%%%%%%%%%%%%%%%%%%%%%%%%%%%%%%%%%%

{\noindent\boldmath $\adfPENT(4, 77)$}~ With point set $Z_{236}$, the 4543 lines are generated from

{\scriptsize
%adfPENTstart:1:4:77:4:4:
 $\{56, 113, 180, 211\}$, $\{2, 9, 124, 229\}$, $\{1, 22, 55, 218\}$, $\{28, 91, 151, 186\}$,\adfsplit
 $\{123, 1, 23, 198\}$, $\{37, 146, 205, 167\}$, $\{113, 220, 83, 26\}$, $\{95, 99, 106, 10\}$,\adfsplit
 $\{106, 103, 89, 17\}$, $\{216, 145, 132, 235\}$, $\{194, 20, 72, 64\}$, $\{34, 47, 163, 76\}$,\adfsplit
 $\{189, 193, 36, 18\}$, $\{143, 17, 91, 51\}$, $\{21, 182, 212, 11\}$, $\{172, 222, 132, 186\}$,\adfsplit
 $\{144, 219, 202, 98\}$, $\{230, 97, 148, 16\}$, $\{152, 164, 98, 145\}$, $\{221, 20, 93, 196\}$,\adfsplit
 $\{71, 150, 113, 33\}$, $\{222, 179, 41, 83\}$, $\{38, 96, 43, 87\}$, $\{44, 203, 50, 51\}$,\adfsplit
 $\{162, 131, 49, 143\}$, $\{133, 185, 102, 118\}$, $\{190, 134, 163, 18\}$, $\{66, 110, 78, 80\}$,\adfsplit
 $\{182, 51, 109, 129\}$, $\{192, 208, 42, 159\}$, $\{57, 20, 114, 22\}$, $\{116, 183, 121, 55\}$,\adfsplit
 $\{135, 92, 124, 24\}$, $\{81, 5, 144, 10\}$, $\{111, 20, 222, 137\}$, $\{10, 82, 13, 91\}$,\adfsplit
 $\{120, 97, 189, 107\}$, $\{0, 1, 78, 146\}$, $\{0, 3, 26, 106\}$, $\{0, 10, 198, 208\}$,\adfsplit
 $\{0, 15, 150, 210\}$, $\{0, 4, 66, 166\}$, $\{0, 9, 110, 138\}$, $\{0, 21, 118, 126\}$,\adfsplit
 $\{0, 18, 156, 230\}$, $\{0, 46, 71, 134\}$, $\{0, 34, 61, 186\}$, $\{1, 13, 42, 46\}$,\adfsplit
 $\{1, 14, 122, 174\}$, $\{1, 26, 138, 181\}$, $\{1, 19, 142, 153\}$, $\{1, 3, 50, 231\}$,\adfsplit
 $\{1, 7, 94, 163\}$, $\{0, 22, 116, 199\}$, $\{1, 158, 203, 235\}$, $\{0, 27, 139, 154\}$,\adfsplit
 $\{1, 115, 143, 214\}$, $\{1, 119, 135, 186\}$, $\{2, 39, 95, 199\}$, $\{0, 17, 127, 191\}$,\adfsplit
 $\{0, 20, 167, 235\}$, $\{0, 59, 107, 131\}$, $\{0, 135, 171, 221\}$, $\{0, 51, 141, 181\}$,\adfsplit
 $\{0, 49, 95, 115\}$, $\{0, 89, 119, 137\}$, $\{0, 55, 77, 109\}$, $\{0, 39, 53, 96\}$,\adfsplit
 $\{0, 93, 163, 209\}$, $\{0, 76, 161, 219\}$, $\{0, 33, 36, 144\}$, $\{0, 23, 64, 212\}$,\adfsplit
 $\{0, 65, 205, 231\}$, $\{0, 48, 197, 225\}$, $\{0, 29, 79, 164\}$, $\{0, 41, 145, 189\}$,\adfsplit
 $\{1, 25, 61, 125\}$
%adfPENTend
}

\noindent under the action of the mapping $x \mapsto x + 4 \adfmod{236}$.
The deficiency graph is connected and has girth 7.
% The opposite lines are generated by the first 4 blocks: True
% End of PENT(4, 77)
%%%%%%%%%%%%%%%%%%%%%%%%%%%%%%%%%%%%%%%%%%%%%%%%%%%%%%%%%%%%%%%%%%%%%%%%%%%%%%%%%%%%%%%%%%

{\noindent\boldmath $\adfPENT(4, 80)$}~ With point set $Z_{245}$, the 4900 lines are generated from

{\scriptsize
%adfPENTstart:1:4:80:5:5:
 $\{31, 43, 79, 178\}$, $\{16, 98, 215, 231\}$, $\{32, 54, 193, 217\}$, $\{57, 70, 151, 205\}$,\adfsplit
 $\{29, 170, 197, 224\}$, $\{192, 76, 48, 209\}$, $\{184, 93, 19, 126\}$, $\{32, 65, 79, 216\}$,\adfsplit
 $\{52, 58, 152, 49\}$, $\{146, 4, 128, 14\}$, $\{195, 126, 154, 49\}$, $\{52, 11, 206, 184\}$,\adfsplit
 $\{136, 156, 43, 130\}$, $\{5, 85, 63, 100\}$, $\{227, 55, 91, 146\}$, $\{126, 227, 46, 11\}$,\adfsplit
 $\{15, 77, 201, 84\}$, $\{222, 42, 33, 92\}$, $\{54, 96, 175, 8\}$, $\{107, 187, 40, 89\}$,\adfsplit
 $\{227, 107, 142, 150\}$, $\{174, 185, 229, 220\}$, $\{68, 234, 34, 136\}$, $\{157, 180, 123, 152\}$,\adfsplit
 $\{217, 10, 127, 111\}$, $\{179, 87, 195, 1\}$, $\{100, 51, 168, 95\}$, $\{85, 25, 33, 7\}$,\adfsplit
 $\{14, 21, 176, 19\}$, $\{29, 69, 213, 185\}$, $\{163, 61, 27, 58\}$, $\{175, 185, 109, 12\}$,\adfsplit
 $\{71, 128, 96, 59\}$, $\{112, 132, 186, 86\}$, $\{219, 138, 123, 223\}$, $\{147, 74, 104, 139\}$,\adfsplit
 $\{0, 1, 4, 164\}$, $\{0, 2, 39, 139\}$, $\{0, 3, 189, 209\}$, $\{0, 7, 19, 94\}$,\adfsplit
 $\{0, 11, 224, 239\}$, $\{0, 12, 84, 174\}$, $\{0, 13, 64, 214\}$, $\{0, 18, 24, 154\}$,\adfsplit
 $\{0, 17, 59, 119\}$, $\{0, 21, 34, 144\}$, $\{0, 22, 149, 219\}$, $\{0, 41, 74, 194\}$,\adfsplit
 $\{0, 23, 109, 142\}$, $\{0, 25, 159, 182\}$, $\{0, 29, 42, 242\}$, $\{0, 33, 141, 184\}$,\adfsplit
 $\{1, 2, 3, 219\}$, $\{1, 6, 12, 194\}$, $\{1, 8, 9, 186\}$, $\{1, 18, 121, 199\}$,\adfsplit
 $\{0, 37, 206, 244\}$, $\{1, 11, 48, 159\}$, $\{1, 22, 24, 43\}$, $\{1, 19, 201, 223\}$,\adfsplit
 $\{1, 28, 32, 94\}$, $\{0, 47, 61, 114\}$, $\{1, 63, 119, 181\}$, $\{1, 53, 74, 88\}$,\adfsplit
 $\{2, 12, 48, 189\}$, $\{2, 42, 83, 159\}$, $\{2, 13, 57, 169\}$, $\{2, 59, 113, 152\}$,\adfsplit
 $\{2, 68, 78, 194\}$, $\{3, 14, 188, 228\}$, $\{2, 98, 124, 148\}$, $\{2, 93, 109, 158\}$,\adfsplit
 $\{0, 30, 87, 192\}$, $\{0, 20, 127, 152\}$, $\{0, 32, 76, 116\}$, $\{0, 92, 131, 241\}$,\adfsplit
 $\{0, 102, 106, 181\}$, $\{0, 52, 121, 187\}$, $\{1, 71, 163, 227\}$, $\{1, 78, 168, 237\}$,\adfsplit
 $\{1, 112, 127, 233\}$, $\{1, 72, 123, 142\}$, $\{1, 147, 208, 222\}$, $\{0, 56, 112, 161\}$,\adfsplit
 $\{0, 86, 122, 236\}$, $\{1, 108, 187, 203\}$, $\{0, 63, 126, 177\}$, $\{2, 73, 118, 198\}$,\adfsplit
 $\{0, 93, 123, 197\}$, $\{0, 71, 156, 170\}$, $\{0, 45, 163, 216\}$, $\{0, 40, 88, 136\}$,\adfsplit
 $\{0, 66, 90, 238\}$, $\{0, 100, 203, 211\}$, $\{0, 53, 168, 226\}$, $\{0, 65, 173, 198\}$,\adfsplit
 $\{0, 38, 130, 228\}$, $\{0, 50, 120, 233\}$, $\{0, 55, 138, 140\}$, $\{0, 143, 213, 218\}$
%adfPENTend
}

\noindent under the action of the mapping $x \mapsto x + 5 \adfmod{245}$.
The deficiency graph is connected and has girth 7.
% The opposite lines are generated by the first 5 blocks: True
% End of PENT(4, 80)
%%%%%%%%%%%%%%%%%%%%%%%%%%%%%%%%%%%%%%%%%%%%%%%%%%%%%%%%%%%%%%%%%%%%%%%%%%%%%%%%%%%%%%%%%%

{\noindent\boldmath $\adfPENT(4, 81)$}~ With point set $Z_{248}$, the 5022 lines are generated from

{\scriptsize
%adfPENTstart:1:4:81:8:8:
 $\{88, 122, 160, 204\}$, $\{21, 102, 114, 164\}$, $\{29, 128, 137, 183\}$, $\{22, 67, 79, 187\}$,\adfsplit
 $\{23, 48, 54, 89\}$, $\{135, 226, 233, 246\}$, $\{13, 153, 204, 235\}$, $\{74, 125, 179, 236\}$,\adfsplit
 $\{234, 44, 202, 115\}$, $\{60, 194, 211, 133\}$, $\{83, 207, 157, 37\}$, $\{207, 190, 212, 242\}$,\adfsplit
 $\{79, 123, 46, 139\}$, $\{12, 161, 28, 127\}$, $\{220, 200, 5, 47\}$, $\{178, 173, 63, 231\}$,\adfsplit
 $\{24, 194, 85, 202\}$, $\{196, 149, 49, 141\}$, $\{166, 237, 90, 161\}$, $\{126, 58, 69, 32\}$,\adfsplit
 $\{171, 201, 139, 189\}$, $\{237, 76, 182, 141\}$, $\{227, 51, 209, 198\}$, $\{122, 224, 99, 74\}$,\adfsplit
 $\{173, 88, 170, 117\}$, $\{125, 19, 187, 16\}$, $\{57, 159, 244, 51\}$, $\{209, 187, 98, 177\}$,\adfsplit
 $\{59, 87, 152, 9\}$, $\{74, 23, 220, 149\}$, $\{102, 205, 70, 43\}$, $\{230, 46, 151, 115\}$,\adfsplit
 $\{8, 40, 203, 207\}$, $\{213, 44, 98, 239\}$, $\{232, 163, 3, 97\}$, $\{40, 172, 98, 133\}$,\adfsplit
 $\{238, 85, 31, 152\}$, $\{77, 105, 30, 67\}$, $\{0, 112, 229, 84\}$, $\{52, 163, 230, 159\}$,\adfsplit
 $\{221, 154, 4, 132\}$, $\{92, 218, 98, 97\}$, $\{107, 138, 199, 197\}$, $\{104, 235, 178, 179\}$,\adfsplit
 $\{5, 25, 37, 213\}$, $\{126, 92, 80, 245\}$, $\{75, 28, 170, 80\}$, $\{72, 210, 207, 226\}$,\adfsplit
 $\{179, 100, 6, 240\}$, $\{82, 45, 55, 111\}$, $\{229, 104, 223, 105\}$, $\{108, 62, 67, 247\}$,\adfsplit
 $\{52, 214, 160, 228\}$, $\{31, 68, 1, 38\}$, $\{21, 246, 137, 111\}$, $\{26, 194, 206, 89\}$,\adfsplit
 $\{77, 58, 221, 62\}$, $\{216, 69, 102, 121\}$, $\{133, 182, 231, 117\}$, $\{88, 1, 198, 236\}$,\adfsplit
 $\{80, 37, 9, 62\}$, $\{156, 92, 31, 132\}$, $\{0, 4, 21, 189\}$, $\{0, 2, 173, 221\}$,\adfsplit
 $\{0, 7, 45, 69\}$, $\{0, 11, 13, 77\}$, $\{0, 15, 157, 245\}$, $\{0, 5, 18, 141\}$,\adfsplit
 $\{0, 22, 197, 235\}$, $\{1, 3, 57, 61\}$, $\{0, 10, 114, 213\}$, $\{0, 23, 174, 237\}$,\adfsplit
 $\{0, 36, 133, 203\}$, $\{1, 4, 53, 203\}$, $\{1, 9, 45, 66\}$, $\{1, 7, 27, 181\}$,\adfsplit
 $\{1, 12, 99, 245\}$, $\{1, 17, 59, 173\}$, $\{2, 7, 109, 235\}$, $\{1, 22, 107, 149\}$,\adfsplit
 $\{1, 26, 126, 205\}$, $\{1, 23, 69, 226\}$, $\{1, 44, 71, 85\}$, $\{1, 34, 117, 153\}$,\adfsplit
 $\{1, 35, 36, 165\}$, $\{2, 20, 45, 212\}$, $\{3, 11, 63, 229\}$, $\{3, 6, 37, 203\}$,\adfsplit
 $\{3, 20, 28, 141\}$, $\{3, 12, 94, 237\}$, $\{3, 14, 175, 245\}$, $\{4, 6, 94, 245\}$,\adfsplit
 $\{4, 5, 142, 223\}$, $\{5, 166, 167, 214\}$, $\{4, 13, 14, 239\}$, $\{0, 8, 43, 147\}$,\adfsplit
 $\{0, 16, 83, 107\}$, $\{0, 42, 115, 227\}$, $\{0, 38, 59, 99\}$, $\{1, 20, 139, 235\}$,\adfsplit
 $\{0, 24, 51, 236\}$, $\{1, 10, 76, 179\}$, $\{0, 28, 81, 211\}$, $\{1, 15, 115, 228\}$,\adfsplit
 $\{0, 57, 106, 219\}$, $\{1, 39, 123, 162\}$, $\{1, 25, 132, 171\}$, $\{1, 42, 63, 75\}$,\adfsplit
 $\{2, 11, 111, 210\}$, $\{2, 12, 103, 139\}$, $\{2, 4, 99, 226\}$, $\{2, 15, 36, 179\}$,\adfsplit
 $\{2, 38, 140, 147\}$, $\{2, 30, 83, 156\}$, $\{2, 39, 94, 243\}$, $\{2, 43, 98, 150\}$,\adfsplit
 $\{2, 44, 67, 110\}$, $\{3, 126, 134, 199\}$, $\{3, 92, 135, 159\}$, $\{3, 71, 119, 236\}$,\adfsplit
 $\{2, 175, 190, 203\}$, $\{0, 17, 100, 188\}$, $\{0, 25, 52, 164\}$, $\{0, 48, 124, 228\}$,\adfsplit
 $\{0, 92, 103, 172\}$, $\{0, 60, 156, 215\}$, $\{1, 150, 156, 204\}$, $\{0, 66, 121, 244\}$,\adfsplit
 $\{2, 46, 66, 236\}$, $\{1, 92, 95, 154\}$, $\{1, 62, 180, 212\}$, $\{1, 46, 161, 220\}$,\adfsplit
 $\{2, 87, 132, 246\}$, $\{2, 126, 204, 222\}$, $\{4, 30, 46, 175\}$, $\{0, 47, 130, 218\}$,\adfsplit
 $\{0, 50, 64, 226\}$, $\{1, 55, 98, 210\}$, $\{2, 62, 71, 166\}$, $\{1, 82, 129, 222\}$,\adfsplit
 $\{1, 18, 74, 177\}$, $\{1, 86, 207, 218\}$, $\{0, 111, 150, 242\}$, $\{0, 70, 142, 186\}$,\adfsplit
 $\{0, 30, 194, 209\}$, $\{0, 126, 166, 202\}$, $\{0, 190, 193, 233\}$, $\{0, 78, 169, 198\}$,\adfsplit
 $\{0, 40, 222, 246\}$, $\{0, 102, 158, 238\}$, $\{1, 49, 113, 190\}$, $\{0, 56, 118, 185\}$,\adfsplit
 $\{0, 49, 214, 239\}$, $\{0, 65, 145, 231\}$, $\{0, 39, 89, 96\}$, $\{0, 87, 151, 225\}$,\adfsplit
 $\{0, 73, 120, 217\}$, $\{0, 63, 79, 168\}$, $\{0, 31, 104, 137\}$, $\{0, 71, 207, 247\}$,\adfsplit
 $\{1, 127, 159, 247\}$, $\{1, 135, 231, 239\}$
%adfPENTend
}

\noindent under the action of the mapping $x \mapsto x + 8 \adfmod{248}$.
The deficiency graph is connected and has girth 6.
% The opposite lines are generated by the first 8 blocks: True
% End of PENT(4, 81)
%%%%%%%%%%%%%%%%%%%%%%%%%%%%%%%%%%%%%%%%%%%%%%%%%%%%%%%%%%%%%%%%%%%%%%%%%%%%%%%%%%%%%%%%%%

{\noindent\boldmath $\adfPENT(4, 85)$}~ With point set $Z_{260}$, the 5525 lines are generated from

{\scriptsize
%adfPENTstart:1:4:85:4:4:
 $\{22, 60, 200, 226\}$, $\{74, 167, 179, 182\}$, $\{36, 81, 189, 240\}$, $\{85, 97, 119, 147\}$,\adfsplit
 $\{194, 114, 231, 117\}$, $\{78, 254, 200, 91\}$, $\{61, 124, 120, 157\}$, $\{27, 46, 161, 157\}$,\adfsplit
 $\{180, 103, 109, 177\}$, $\{227, 248, 138, 174\}$, $\{122, 137, 176, 190\}$, $\{115, 204, 32, 235\}$,\adfsplit
 $\{7, 37, 93, 240\}$, $\{26, 187, 89, 130\}$, $\{188, 109, 12, 54\}$, $\{85, 77, 190, 196\}$,\adfsplit
 $\{214, 21, 159, 223\}$, $\{73, 90, 116, 244\}$, $\{74, 0, 232, 247\}$, $\{209, 244, 196, 131\}$,\adfsplit
 $\{144, 82, 183, 251\}$, $\{107, 5, 133, 233\}$, $\{209, 29, 116, 15\}$, $\{0, 1, 21, 169\}$,\adfsplit
 $\{0, 2, 9, 53\}$, $\{0, 3, 5, 237\}$, $\{0, 6, 25, 41\}$, $\{0, 7, 17, 253\}$,\adfsplit
 $\{0, 8, 69, 241\}$, $\{0, 10, 105, 249\}$, $\{0, 11, 77, 117\}$, $\{0, 12, 121, 205\}$,\adfsplit
 $\{0, 18, 49, 185\}$, $\{0, 19, 73, 133\}$, $\{0, 16, 81, 145\}$, $\{0, 23, 125, 229\}$,\adfsplit
 $\{0, 34, 177, 213\}$, $\{1, 2, 3, 141\}$, $\{1, 6, 49, 58\}$, $\{0, 38, 85, 137\}$,\adfsplit
 $\{1, 7, 11, 185\}$, $\{1, 14, 19, 189\}$, $\{0, 30, 51, 161\}$, $\{0, 29, 44, 87\}$,\adfsplit
 $\{0, 46, 58, 165\}$, $\{1, 30, 47, 70\}$, $\{1, 31, 38, 134\}$, $\{1, 26, 34, 126\}$,\adfsplit
 $\{1, 39, 50, 98\}$, $\{0, 67, 118, 157\}$, $\{1, 71, 102, 143\}$, $\{1, 43, 90, 238\}$,\adfsplit
 $\{1, 46, 62, 178\}$, $\{1, 75, 95, 227\}$, $\{1, 66, 138, 202\}$, $\{1, 110, 170, 219\}$,\adfsplit
 $\{1, 79, 111, 174\}$, $\{1, 119, 211, 250\}$, $\{1, 107, 186, 190\}$, $\{1, 158, 191, 243\}$,\adfsplit
 $\{1, 163, 199, 234\}$, $\{1, 215, 223, 239\}$, $\{0, 20, 86, 162\}$, $\{0, 24, 94, 182\}$,\adfsplit
 $\{0, 35, 146, 174\}$, $\{0, 32, 110, 154\}$, $\{0, 36, 214, 238\}$, $\{0, 63, 210, 242\}$,\adfsplit
 $\{0, 114, 134, 144\}$, $\{2, 27, 142, 187\}$, $\{0, 91, 175, 218\}$, $\{0, 59, 139, 246\}$,\adfsplit
 $\{0, 71, 190, 219\}$, $\{0, 98, 167, 223\}$, $\{0, 90, 92, 156\}$, $\{0, 103, 143, 170\}$,\adfsplit
 $\{0, 50, 127, 187\}$, $\{0, 62, 123, 192\}$, $\{0, 76, 207, 255\}$, $\{0, 55, 100, 211\}$,\adfsplit
 $\{0, 47, 72, 208\}$, $\{0, 40, 115, 148\}$, $\{0, 95, 96, 231\}$, $\{0, 79, 155, 251\}$,\adfsplit
 $\{0, 80, 199, 243\}$
%adfPENTend
}

\noindent under the action of the mapping $x \mapsto x + 4 \adfmod{260}$.
The deficiency graph is connected and has girth 6.
% The opposite lines are generated by the first 4 blocks: True
% End of PENT(4, 85)
%%%%%%%%%%%%%%%%%%%%%%%%%%%%%%%%%%%%%%%%%%%%%%%%%%%%%%%%%%%%%%%%%%%%%%%%%%%%%%%%%%%%%%%%%%

{\noindent\boldmath $\adfPENT(4, 93)$}~ With point set $Z_{284}$, the 6603 lines are generated from

{\scriptsize
%adfPENTstart:1:4:93:4:4:
 $\{92, 122, 192, 202\}$, $\{53, 86, 127, 233\}$, $\{84, 87, 164, 201\}$, $\{139, 151, 161, 202\}$,\adfsplit
 $\{80, 19, 99, 152\}$, $\{29, 198, 247, 8\}$, $\{137, 0, 159, 268\}$, $\{223, 262, 251, 197\}$,\adfsplit
 $\{265, 52, 158, 31\}$, $\{83, 237, 116, 41\}$, $\{131, 154, 281, 218\}$, $\{215, 253, 77, 158\}$,\adfsplit
 $\{232, 191, 126, 80\}$, $\{177, 175, 44, 230\}$, $\{84, 101, 25, 46\}$, $\{12, 111, 215, 73\}$,\adfsplit
 $\{117, 66, 57, 101\}$, $\{176, 244, 253, 56\}$, $\{0, 1, 2, 49\}$, $\{0, 4, 33, 57\}$,\adfsplit
 $\{0, 5, 7, 221\}$, $\{0, 6, 13, 261\}$, $\{0, 8, 93, 149\}$, $\{0, 11, 41, 229\}$,\adfsplit
 $\{0, 12, 81, 181\}$, $\{0, 14, 25, 185\}$, $\{0, 15, 65, 217\}$, $\{0, 18, 45, 157\}$,\adfsplit
 $\{0, 20, 125, 165\}$, $\{0, 22, 97, 177\}$, $\{0, 23, 101, 245\}$, $\{0, 24, 273, 277\}$,\adfsplit
 $\{0, 26, 113, 233\}$, $\{0, 27, 129, 161\}$, $\{0, 42, 193, 265\}$, $\{0, 35, 173, 201\}$,\adfsplit
 $\{1, 6, 7, 201\}$, $\{1, 9, 19, 273\}$, $\{0, 31, 205, 269\}$, $\{0, 54, 189, 281\}$,\adfsplit
 $\{0, 39, 109, 237\}$, $\{1, 14, 137, 154\}$, $\{0, 50, 73, 241\}$, $\{0, 43, 47, 89\}$,\adfsplit
 $\{0, 51, 67, 257\}$, $\{1, 26, 35, 163\}$, $\{1, 46, 50, 195\}$, $\{1, 59, 99, 271\}$,\adfsplit
 $\{1, 38, 74, 227\}$, $\{1, 98, 115, 259\}$, $\{1, 15, 103, 203\}$, $\{1, 102, 107, 267\}$,\adfsplit
 $\{1, 70, 110, 191\}$, $\{1, 123, 126, 142\}$, $\{1, 118, 174, 251\}$, $\{1, 155, 187, 202\}$,\adfsplit
 $\{1, 106, 166, 199\}$, $\{1, 122, 194, 214\}$, $\{1, 159, 218, 226\}$, $\{1, 175, 211, 246\}$,\adfsplit
 $\{1, 186, 223, 231\}$, $\{1, 167, 242, 266\}$, $\{1, 182, 270, 282\}$, $\{1, 206, 254, 279\}$,\adfsplit
 $\{0, 28, 83, 135\}$, $\{0, 32, 91, 211\}$, $\{0, 34, 87, 143\}$, $\{0, 36, 115, 283\}$,\adfsplit
 $\{0, 38, 155, 215\}$, $\{0, 40, 187, 235\}$, $\{0, 44, 119, 139\}$, $\{0, 56, 183, 227\}$,\adfsplit
 $\{0, 62, 191, 255\}$, $\{2, 23, 91, 115\}$, $\{0, 71, 98, 271\}$, $\{0, 94, 123, 275\}$,\adfsplit
 $\{2, 71, 154, 279\}$, $\{0, 103, 182, 279\}$, $\{0, 90, 158, 259\}$, $\{0, 138, 199, 242\}$,\adfsplit
 $\{0, 52, 156, 219\}$, $\{0, 58, 163, 194\}$, $\{0, 126, 254, 267\}$, $\{0, 60, 162, 176\}$,\adfsplit
 $\{0, 82, 144, 258\}$, $\{0, 64, 210, 238\}$, $\{0, 66, 88, 230\}$, $\{0, 48, 124, 274\}$,\adfsplit
 $\{0, 70, 148, 278\}$, $\{0, 78, 112, 266\}$, $\{0, 74, 118, 234\}$, $\{0, 84, 170, 282\}$,\adfsplit
 $\{0, 134, 166, 218\}$
%adfPENTend
}

\noindent under the action of the mapping $x \mapsto x + 4 \adfmod{284}$.
The deficiency graph is connected and has girth 6.
% The opposite lines are generated by the first 4 blocks: True
% End of PENT(4, 93)
%%%%%%%%%%%%%%%%%%%%%%%%%%%%%%%%%%%%%%%%%%%%%%%%%%%%%%%%%%%%%%%%%%%%%%%%%%%%%%%%%%%%%%%%%%

{\noindent\boldmath $\adfPENT(4, 97)$}~ With point set $Z_{296}$, the 7178 lines are generated from

{\scriptsize
%adfPENTstart:1:4:97:8:8:
 $\{116, 121, 147, 236\}$, $\{12, 29, 176, 178\}$, $\{47, 121, 159, 227\}$, $\{74, 143, 152, 230\}$,\adfsplit
 $\{64, 184, 206, 289\}$, $\{133, 173, 273, 286\}$, $\{21, 75, 100, 135\}$, $\{146, 163, 174, 258\}$,\adfsplit
 $\{162, 102, 103, 128\}$, $\{274, 95, 128, 21\}$, $\{52, 109, 234, 78\}$, $\{166, 154, 44, 274\}$,\adfsplit
 $\{270, 224, 208, 189\}$, $\{25, 14, 126, 226\}$, $\{9, 117, 2, 181\}$, $\{91, 228, 130, 107\}$,\adfsplit
 $\{0, 221, 207, 36\}$, $\{61, 130, 284, 49\}$, $\{59, 163, 29, 149\}$, $\{51, 121, 28, 160\}$,\adfsplit
 $\{13, 81, 179, 63\}$, $\{121, 231, 206, 72\}$, $\{53, 135, 201, 148\}$, $\{90, 15, 18, 38\}$,\adfsplit
 $\{263, 31, 293, 4\}$, $\{182, 286, 45, 9\}$, $\{190, 192, 213, 246\}$, $\{116, 190, 52, 63\}$,\adfsplit
 $\{140, 242, 238, 180\}$, $\{245, 95, 239, 235\}$, $\{265, 139, 294, 219\}$, $\{32, 125, 284, 33\}$,\adfsplit
 $\{41, 192, 75, 158\}$, $\{23, 137, 191, 220\}$, $\{176, 100, 156, 279\}$, $\{213, 169, 147, 3\}$,\adfsplit
 $\{72, 27, 21, 295\}$, $\{238, 24, 147, 34\}$, $\{104, 14, 100, 247\}$, $\{277, 180, 165, 245\}$,\adfsplit
 $\{67, 244, 62, 262\}$, $\{96, 100, 99, 226\}$, $\{30, 20, 273, 115\}$, $\{202, 211, 83, 54\}$,\adfsplit
 $\{94, 282, 227, 257\}$, $\{199, 26, 189, 58\}$, $\{66, 76, 221, 118\}$, $\{147, 285, 66, 82\}$,\adfsplit
 $\{278, 21, 88, 271\}$, $\{167, 121, 244, 221\}$, $\{177, 238, 165, 280\}$, $\{118, 2, 52, 183\}$,\adfsplit
 $\{235, 294, 47, 128\}$, $\{260, 202, 284, 0\}$, $\{156, 175, 139, 99\}$, $\{295, 105, 57, 214\}$,\adfsplit
 $\{245, 159, 247, 176\}$, $\{114, 185, 289, 156\}$, $\{101, 77, 198, 30\}$, $\{245, 127, 249, 187\}$,\adfsplit
 $\{35, 255, 44, 140\}$, $\{179, 136, 12, 79\}$, $\{28, 111, 25, 98\}$, $\{88, 5, 30, 292\}$,\adfsplit
 $\{266, 99, 219, 92\}$, $\{130, 63, 229, 31\}$, $\{52, 36, 187, 130\}$, $\{116, 198, 233, 55\}$,\adfsplit
 $\{67, 245, 25, 227\}$, $\{173, 75, 64, 8\}$, $\{152, 31, 65, 283\}$, $\{111, 33, 183, 16\}$,\adfsplit
 $\{142, 19, 70, 280\}$, $\{204, 15, 226, 98\}$, $\{23, 293, 2, 49\}$, $\{66, 199, 264, 218\}$,\adfsplit
 $\{21, 157, 79, 81\}$, $\{162, 189, 119, 255\}$, $\{232, 201, 270, 2\}$, $\{83, 254, 147, 160\}$,\adfsplit
 $\{244, 197, 64, 134\}$, $\{63, 99, 263, 75\}$, $\{119, 68, 288, 220\}$, $\{151, 152, 220, 78\}$,\adfsplit
 $\{0, 5, 191, 199\}$, $\{0, 9, 15, 135\}$, $\{0, 6, 63, 119\}$, $\{0, 7, 18, 23\}$,\adfsplit
 $\{0, 13, 31, 55\}$, $\{0, 19, 39, 87\}$, $\{0, 14, 47, 151\}$, $\{1, 2, 31, 247\}$,\adfsplit
 $\{1, 3, 95, 135\}$, $\{0, 20, 79, 100\}$, $\{0, 30, 92, 247\}$, $\{0, 50, 52, 255\}$,\adfsplit
 $\{1, 6, 60, 103\}$, $\{1, 5, 143, 148\}$, $\{1, 9, 23, 61\}$, $\{1, 10, 63, 172\}$,\adfsplit
 $\{0, 27, 111, 253\}$, $\{1, 11, 215, 259\}$, $\{1, 18, 284, 287\}$, $\{1, 22, 207, 254\}$,\adfsplit
 $\{1, 46, 126, 167\}$, $\{1, 20, 150, 159\}$, $\{1, 66, 155, 255\}$, $\{1, 34, 71, 195\}$,\adfsplit
 $\{1, 36, 38, 199\}$, $\{2, 3, 127, 222\}$, $\{2, 46, 63, 211\}$, $\{2, 6, 38, 263\}$,\adfsplit
 $\{2, 21, 111, 230\}$, $\{2, 123, 151, 277\}$, $\{2, 79, 134, 253\}$, $\{3, 15, 35, 77\}$,\adfsplit
 $\{2, 103, 148, 187\}$, $\{2, 5, 167, 219\}$, $\{4, 95, 110, 198\}$, $\{3, 55, 149, 206\}$,\adfsplit
 $\{3, 30, 190, 295\}$, $\{4, 79, 269, 286\}$, $\{0, 25, 102, 246\}$, $\{0, 45, 110, 230\}$,\adfsplit
 $\{0, 24, 150, 198\}$, $\{0, 12, 83, 182\}$, $\{1, 17, 182, 206\}$, $\{1, 68, 230, 246\}$,\adfsplit
 $\{1, 69, 198, 238\}$, $\{0, 28, 118, 140\}$, $\{0, 26, 278, 291\}$, $\{2, 20, 278, 286\}$,\adfsplit
 $\{1, 51, 291, 294\}$, $\{0, 51, 139, 286\}$, $\{1, 52, 270, 276\}$, $\{1, 59, 85, 94\}$,\adfsplit
 $\{2, 36, 147, 262\}$, $\{1, 92, 141, 142\}$, $\{2, 13, 174, 195\}$, $\{2, 10, 70, 107\}$,\adfsplit
 $\{2, 53, 75, 94\}$, $\{2, 26, 126, 258\}$, $\{3, 53, 205, 254\}$, $\{3, 36, 115, 182\}$,\adfsplit
 $\{4, 37, 109, 254\}$, $\{2, 78, 85, 194\}$, $\{3, 46, 237, 253\}$, $\{0, 37, 108, 212\}$,\adfsplit
 $\{0, 48, 196, 259\}$, $\{1, 19, 212, 260\}$, $\{0, 32, 91, 188\}$, $\{0, 29, 84, 244\}$,\adfsplit
 $\{1, 33, 132, 140\}$, $\{0, 60, 61, 268\}$, $\{0, 44, 85, 275\}$, $\{1, 25, 147, 220\}$,\adfsplit
 $\{1, 28, 75, 156\}$, $\{1, 115, 187, 228\}$, $\{2, 43, 213, 252\}$, $\{1, 76, 91, 234\}$,\adfsplit
 $\{2, 76, 90, 285\}$, $\{0, 75, 124, 293\}$, $\{2, 77, 180, 212\}$, $\{2, 93, 141, 260\}$,\adfsplit
 $\{2, 92, 98, 235\}$, $\{2, 61, 244, 269\}$, $\{1, 125, 162, 188\}$, $\{0, 33, 99, 195\}$,\adfsplit
 $\{0, 53, 112, 227\}$, $\{0, 138, 163, 171\}$, $\{0, 117, 155, 170\}$, $\{0, 101, 203, 205\}$,\adfsplit
 $\{1, 73, 253, 261\}$, $\{1, 146, 213, 269\}$, $\{0, 125, 235, 289\}$, $\{1, 21, 201, 283\}$,\adfsplit
 $\{0, 65, 194, 243\}$, $\{0, 57, 267, 274\}$, $\{0, 173, 234, 269\}$, $\{0, 141, 161, 273\}$,\adfsplit
 $\{0, 88, 178, 285\}$, $\{0, 77, 169, 216\}$, $\{0, 42, 201, 290\}$, $\{1, 57, 210, 233\}$,\adfsplit
 $\{0, 82, 185, 242\}$, $\{0, 74, 113, 154\}$, $\{1, 50, 106, 161\}$, $\{0, 73, 153, 281\}$,\adfsplit
 $\{0, 97, 210, 241\}$, $\{0, 89, 152, 258\}$, $\{0, 8, 122, 137\}$, $\{0, 41, 81, 160\}$,\adfsplit
 $\{0, 40, 104, 266\}$, $\{0, 58, 72, 168\}$
%adfPENTend
}

\noindent under the action of the mapping $x \mapsto x + 8 \adfmod{296}$.
The deficiency graph is connected and has girth 6.
% The opposite lines are generated by the first 8 blocks: True
% End of PENT(4, 97)
%%%%%%%%%%%%%%%%%%%%%%%%%%%%%%%%%%%%%%%%%%%%%%%%%%%%%%%%%%%%%%%%%%%%%%%%%%%%%%%%%%%%%%%%%%

{\noindent\boldmath $\adfPENT(4, 100)$}~ With point set $Z_{305}$, the 7625 lines are generated from

{\scriptsize
%adfPENTstart:1:4:100:5:5:
 $\{95, 96, 162, 210\}$, $\{62, 103, 143, 210\}$, $\{84, 145, 233, 246\}$, $\{77, 84, 166, 206\}$,\adfsplit
 $\{24, 227, 228, 289\}$, $\{123, 282, 130, 176\}$, $\{176, 6, 20, 26\}$, $\{180, 44, 47, 148\}$,\adfsplit
 $\{236, 224, 119, 39\}$, $\{101, 1, 99, 97\}$, $\{110, 157, 93, 150\}$, $\{293, 245, 281, 18\}$,\adfsplit
 $\{40, 31, 270, 227\}$, $\{228, 96, 155, 149\}$, $\{6, 10, 169, 300\}$, $\{128, 242, 156, 164\}$,\adfsplit
 $\{12, 90, 71, 106\}$, $\{58, 260, 27, 292\}$, $\{251, 50, 27, 15\}$, $\{34, 19, 149, 63\}$,\adfsplit
 $\{106, 132, 82, 188\}$, $\{218, 224, 91, 66\}$, $\{301, 51, 198, 199\}$, $\{89, 270, 291, 22\}$,\adfsplit
 $\{127, 253, 141, 7\}$, $\{119, 130, 106, 76\}$, $\{270, 157, 119, 25\}$, $\{257, 129, 115, 29\}$,\adfsplit
 $\{80, 131, 214, 46\}$, $\{169, 38, 31, 106\}$, $\{211, 169, 32, 62\}$, $\{200, 156, 115, 148\}$,\adfsplit
 $\{83, 212, 117, 128\}$, $\{154, 38, 58, 42\}$, $\{301, 214, 35, 224\}$, $\{79, 285, 8, 127\}$,\adfsplit
 $\{29, 164, 161, 245\}$, $\{157, 243, 122, 81\}$, $\{113, 175, 192, 54\}$, $\{302, 137, 281, 81\}$,\adfsplit
 $\{0, 2, 31, 176\}$, $\{0, 3, 81, 211\}$, $\{0, 4, 26, 216\}$, $\{0, 5, 61, 131\}$,\adfsplit
 $\{0, 8, 116, 226\}$, $\{0, 9, 71, 196\}$, $\{0, 10, 161, 241\}$, $\{0, 13, 76, 86\}$,\adfsplit
 $\{0, 18, 106, 166\}$, $\{0, 20, 141, 206\}$, $\{1, 2, 291, 297\}$, $\{1, 3, 256, 261\}$,\adfsplit
 $\{0, 19, 91, 256\}$, $\{1, 18, 186, 208\}$, $\{0, 22, 181, 276\}$, $\{1, 12, 24, 216\}$,\adfsplit
 $\{0, 24, 87, 136\}$, $\{0, 37, 64, 111\}$, $\{1, 19, 28, 74\}$, $\{0, 39, 146, 297\}$,\adfsplit
 $\{1, 32, 49, 167\}$, $\{0, 44, 117, 171\}$, $\{1, 29, 34, 192\}$, $\{1, 47, 68, 237\}$,\adfsplit
 $\{1, 33, 89, 262\}$, $\{1, 58, 63, 112\}$, $\{1, 38, 73, 249\}$, $\{1, 39, 72, 269\}$,\adfsplit
 $\{1, 78, 117, 283\}$, $\{1, 79, 93, 227\}$, $\{1, 59, 124, 223\}$, $\{1, 92, 129, 274\}$,\adfsplit
 $\{1, 118, 137, 242\}$, $\{1, 108, 123, 289\}$, $\{1, 154, 178, 288\}$, $\{1, 173, 183, 268\}$,\adfsplit
 $\{1, 122, 188, 263\}$, $\{1, 187, 207, 213\}$, $\{1, 168, 182, 279\}$, $\{1, 142, 193, 287\}$,\adfsplit
 $\{1, 202, 248, 272\}$, $\{1, 209, 239, 273\}$, $\{1, 254, 258, 299\}$, $\{1, 214, 267, 303\}$,\adfsplit
 $\{0, 23, 54, 167\}$, $\{0, 25, 52, 267\}$, $\{0, 29, 42, 97\}$, $\{0, 30, 112, 137\}$,\adfsplit
 $\{0, 34, 62, 212\}$, $\{0, 53, 122, 207\}$, $\{0, 55, 232, 277\}$, $\{2, 7, 67, 82\}$,\adfsplit
 $\{0, 50, 127, 302\}$, $\{0, 63, 92, 272\}$, $\{0, 79, 182, 292\}$, $\{0, 59, 102, 202\}$,\adfsplit
 $\{0, 68, 84, 217\}$, $\{0, 104, 197, 269\}$, $\{2, 54, 124, 198\}$, $\{0, 229, 247, 279\}$,\adfsplit
 $\{2, 89, 208, 299\}$, $\{2, 113, 249, 298\}$, $\{2, 133, 154, 263\}$, $\{2, 44, 63, 203\}$,\adfsplit
 $\{2, 59, 119, 209\}$, $\{0, 98, 204, 287\}$, $\{0, 148, 214, 237\}$, $\{2, 24, 93, 219\}$,\adfsplit
 $\{0, 69, 175, 289\}$, $\{0, 109, 144, 165\}$, $\{0, 74, 105, 254\}$, $\{0, 158, 209, 234\}$,\adfsplit
 $\{0, 45, 153, 304\}$, $\{0, 164, 218, 278\}$, $\{0, 168, 194, 233\}$, $\{0, 119, 160, 303\}$,\adfsplit
 $\{0, 49, 133, 215\}$, $\{0, 110, 228, 239\}$, $\{0, 93, 173, 185\}$, $\{0, 113, 163, 268\}$,\adfsplit
 $\{0, 43, 178, 203\}$, $\{0, 38, 128, 150\}$, $\{0, 65, 188, 235\}$, $\{0, 83, 138, 208\}$,\adfsplit
 $\{0, 58, 100, 180\}$
%adfPENTend
}

\noindent under the action of the mapping $x \mapsto x + 5 \adfmod{305}$.
The deficiency graph is connected and has girth 6.
% The opposite lines are generated by the first 5 blocks: True
% End of PENT(4, 100)
%%%%%%%%%%%%%%%%%%%%%%%%%%%%%%%%%%%%%%%%%%%%%%%%%%%%%%%%%%%%%%%%%%%%%%%%%%%%%%%%%%%%%%%%%%

{\noindent\boldmath $\adfPENT(4, 101)$}~ With point set $Z_{308}$, the 7777 lines are generated from

{\scriptsize
%adfPENTstart:1:4:101:4:4:
 $\{132, 176, 273, 298\}$, $\{17, 36, 94, 293\}$, $\{12, 179, 217, 251\}$, $\{62, 107, 134, 207\}$,\adfsplit
 $\{19, 3, 104, 4\}$, $\{111, 176, 165, 219\}$, $\{84, 147, 88, 236\}$, $\{184, 97, 237, 87\}$,\adfsplit
 $\{178, 108, 242, 202\}$, $\{118, 37, 268, 144\}$, $\{219, 292, 216, 152\}$, $\{54, 247, 280, 194\}$,\adfsplit
 $\{240, 89, 106, 112\}$, $\{282, 163, 102, 214\}$, $\{293, 44, 54, 74\}$, $\{213, 46, 3, 50\}$,\adfsplit
 $\{192, 17, 56, 172\}$, $\{249, 78, 286, 24\}$, $\{299, 216, 203, 114\}$, $\{251, 205, 29, 202\}$,\adfsplit
 $\{149, 122, 216, 182\}$, $\{268, 254, 291, 75\}$, $\{158, 235, 1, 101\}$, $\{133, 219, 194, 51\}$,\adfsplit
 $\{0, 60, 191, 189\}$, $\{122, 243, 12, 60\}$, $\{159, 81, 103, 237\}$, $\{138, 64, 176, 139\}$,\adfsplit
 $\{134, 267, 39, 69\}$, $\{130, 95, 243, 292\}$, $\{210, 0, 161, 147\}$, $\{74, 220, 231, 132\}$,\adfsplit
 $\{285, 20, 57, 201\}$, $\{87, 57, 286, 185\}$, $\{22, 133, 107, 10\}$, $\{158, 68, 14, 283\}$,\adfsplit
 $\{151, 171, 41, 109\}$, $\{214, 292, 111, 126\}$, $\{6, 181, 126, 147\}$, $\{96, 103, 247, 259\}$,\adfsplit
 $\{135, 181, 116, 237\}$, $\{0, 1, 91, 119\}$, $\{0, 2, 71, 103\}$, $\{0, 5, 155, 195\}$,\adfsplit
 $\{0, 6, 87, 123\}$, $\{0, 8, 39, 263\}$, $\{0, 9, 35, 79\}$, $\{0, 12, 279, 303\}$,\adfsplit
 $\{0, 13, 27, 199\}$, $\{0, 17, 55, 143\}$, $\{0, 18, 47, 179\}$, $\{0, 21, 159, 287\}$,\adfsplit
 $\{0, 25, 107, 227\}$, $\{0, 14, 187, 299\}$, $\{1, 2, 7, 59\}$, $\{1, 5, 107, 167\}$,\adfsplit
 $\{1, 6, 239, 247\}$, $\{0, 34, 139, 203\}$, $\{0, 22, 171, 175\}$, $\{0, 42, 135, 251\}$,\adfsplit
 $\{1, 9, 75, 259\}$, $\{1, 51, 54, 291\}$, $\{0, 41, 51, 283\}$, $\{1, 25, 147, 195\}$,\adfsplit
 $\{0, 86, 95, 102\}$, $\{1, 10, 274, 307\}$, $\{0, 66, 105, 247\}$, $\{1, 41, 155, 246\}$,\adfsplit
 $\{1, 14, 215, 237\}$, $\{1, 19, 74, 286\}$, $\{1, 42, 179, 185\}$, $\{1, 30, 95, 182\}$,\adfsplit
 $\{1, 202, 219, 238\}$, $\{1, 170, 183, 262\}$, $\{2, 30, 227, 258\}$, $\{0, 16, 130, 306\}$,\adfsplit
 $\{0, 26, 33, 150\}$, $\{0, 46, 182, 256\}$, $\{0, 29, 50, 154\}$, $\{0, 73, 170, 286\}$,\adfsplit
 $\{0, 38, 238, 268\}$, $\{0, 106, 117, 266\}$, $\{0, 118, 126, 202\}$, $\{0, 85, 198, 246\}$,\adfsplit
 $\{1, 29, 222, 278\}$, $\{0, 194, 226, 245\}$, $\{0, 49, 69, 258\}$, $\{1, 53, 117, 294\}$,\adfsplit
 $\{0, 57, 178, 261\}$, $\{0, 93, 138, 213\}$, $\{1, 49, 154, 197\}$, $\{0, 61, 190, 281\}$,\adfsplit
 $\{0, 24, 56, 242\}$, $\{0, 101, 177, 237\}$, $\{0, 28, 120, 165\}$, $\{0, 81, 173, 204\}$,\adfsplit
 $\{0, 84, 209, 253\}$, $\{0, 80, 193, 229\}$, $\{0, 197, 293, 305\}$, $\{0, 89, 96, 164\}$,\adfsplit
 $\{0, 36, 145, 236\}$
%adfPENTend
}

\noindent under the action of the mapping $x \mapsto x + 4 \adfmod{308}$.
The deficiency graph is connected and has girth 6.
% The opposite lines are generated by the first 4 blocks: True
% End of PENT(4, 101)
%%%%%%%%%%%%%%%%%%%%%%%%%%%%%%%%%%%%%%%%%%%%%%%%%%%%%%%%%%%%%%%%%%%%%%%%%%%%%%%%%%%%%%%%%%

{\noindent\boldmath $\adfPENT(4, 108)$}~ With point set $Z_{329}$, the 8883 lines are generated from

{\scriptsize
%adfPENTstart:1:4:108:7:7:
 $\{151, 152, 269, 302\}$, $\{28, 137, 170, 262\}$, $\{162, 201, 271, 283\}$, $\{51, 63, 71, 285\}$,\adfsplit
 $\{182, 197, 279, 321\}$, $\{52, 65, 135, 182\}$, $\{18, 60, 104, 237\}$, $\{4, 91, 293, 199\}$,\adfsplit
 $\{276, 200, 25, 207\}$, $\{18, 185, 112, 274\}$, $\{319, 129, 219, 275\}$, $\{238, 235, 259, 140\}$,\adfsplit
 $\{250, 254, 62, 278\}$, $\{187, 238, 278, 193\}$, $\{259, 174, 112, 279\}$, $\{103, 121, 235, 274\}$,\adfsplit
 $\{40, 85, 162, 241\}$, $\{115, 252, 79, 96\}$, $\{56, 282, 172, 236\}$, $\{211, 228, 215, 156\}$,\adfsplit
 $\{254, 256, 40, 286\}$, $\{168, 233, 98, 59\}$, $\{108, 235, 2, 301\}$, $\{285, 34, 44, 128\}$,\adfsplit
 $\{94, 194, 158, 5\}$, $\{217, 203, 295, 194\}$, $\{77, 34, 1, 272\}$, $\{282, 10, 150, 313\}$,\adfsplit
 $\{232, 138, 69, 248\}$, $\{79, 142, 290, 280\}$, $\{251, 222, 288, 148\}$, $\{327, 248, 307, 39\}$,\adfsplit
 $\{27, 255, 2, 97\}$, $\{134, 318, 305, 124\}$, $\{231, 61, 257, 205\}$, $\{255, 296, 110, 171\}$,\adfsplit
 $\{175, 320, 155, 34\}$, $\{277, 71, 94, 261\}$, $\{52, 129, 71, 9\}$, $\{241, 22, 71, 115\}$,\adfsplit
 $\{39, 147, 35, 318\}$, $\{117, 152, 184, 99\}$, $\{165, 326, 180, 277\}$, $\{241, 306, 226, 295\}$,\adfsplit
 $\{166, 252, 270, 43\}$, $\{23, 21, 72, 27\}$, $\{104, 316, 78, 73\}$, $\{285, 178, 21, 67\}$,\adfsplit
 $\{21, 310, 125, 174\}$, $\{251, 66, 90, 92\}$, $\{144, 297, 89, 173\}$, $\{143, 251, 34, 158\}$,\adfsplit
 $\{288, 316, 169, 318\}$, $\{255, 155, 151, 59\}$, $\{274, 156, 58, 249\}$, $\{131, 257, 299, 245\}$,\adfsplit
 $\{279, 141, 172, 81\}$, $\{175, 122, 126, 174\}$, $\{2, 219, 252, 320\}$, $\{136, 199, 143, 172\}$,\adfsplit
 $\{203, 132, 247, 125\}$, $\{278, 49, 207, 228\}$, $\{239, 304, 167, 153\}$, $\{80, 191, 268, 127\}$,\adfsplit
 $\{306, 82, 76, 279\}$, $\{239, 24, 215, 96\}$, $\{130, 77, 290, 220\}$, $\{94, 116, 128, 69\}$,\adfsplit
 $\{266, 156, 4, 264\}$, $\{201, 125, 309, 324\}$, $\{328, 12, 181, 212\}$, $\{133, 116, 83, 22\}$,\adfsplit
 $\{214, 166, 223, 91\}$, $\{209, 134, 20, 128\}$, $\{285, 74, 93, 36\}$, $\{253, 46, 285, 299\}$,\adfsplit
 $\{105, 37, 195, 218\}$, $\{123, 100, 232, 91\}$, $\{65, 38, 176, 83\}$, $\{294, 103, 237, 265\}$,\adfsplit
 $\{35, 295, 316, 164\}$, $\{115, 196, 121, 181\}$, $\{171, 84, 40, 270\}$, $\{221, 79, 164, 176\}$,\adfsplit
 $\{0, 3, 5, 101\}$, $\{0, 1, 38, 194\}$, $\{0, 11, 31, 255\}$, $\{0, 13, 24, 178\}$,\adfsplit
 $\{0, 16, 17, 45\}$, $\{0, 19, 297, 318\}$, $\{0, 22, 150, 164\}$, $\{0, 27, 115, 227\}$,\adfsplit
 $\{0, 25, 80, 241\}$, $\{0, 37, 59, 206\}$, $\{0, 33, 66, 185\}$, $\{1, 2, 136, 227\}$,\adfsplit
 $\{0, 34, 52, 94\}$, $\{1, 6, 73, 108\}$, $\{0, 55, 262, 311\}$, $\{0, 39, 47, 122\}$,\adfsplit
 $\{1, 11, 223, 269\}$, $\{1, 20, 41, 213\}$, $\{1, 13, 255, 265\}$, $\{0, 61, 181, 304\}$,\adfsplit
 $\{1, 4, 131, 290\}$, $\{1, 52, 53, 69\}$, $\{1, 37, 164, 173\}$, $\{1, 26, 206, 243\}$,\adfsplit
 $\{1, 25, 145, 248\}$, $\{2, 10, 226, 265\}$, $\{1, 10, 30, 163\}$, $\{2, 12, 13, 66\}$,\adfsplit
 $\{2, 5, 153, 192\}$, $\{2, 48, 166, 234\}$, $\{2, 9, 76, 80\}$, $\{2, 11, 52, 68\}$,\adfsplit
 $\{2, 55, 269, 292\}$, $\{2, 40, 276, 314\}$, $\{2, 130, 256, 262\}$, $\{3, 61, 117, 292\}$,\adfsplit
 $\{3, 6, 11, 251\}$, $\{2, 53, 101, 242\}$, $\{3, 151, 235, 300\}$, $\{3, 158, 216, 327\}$,\adfsplit
 $\{3, 47, 83, 237\}$, $\{3, 202, 270, 298\}$, $\{0, 7, 131, 208\}$, $\{0, 23, 236, 299\}$,\adfsplit
 $\{0, 56, 166, 313\}$, $\{0, 29, 83, 117\}$, $\{0, 36, 96, 103\}$, $\{0, 41, 89, 160\}$,\adfsplit
 $\{0, 69, 189, 271\}$, $\{1, 8, 181, 278\}$, $\{1, 9, 82, 103\}$, $\{1, 48, 51, 201\}$,\adfsplit
 $\{1, 88, 202, 292\}$, $\{0, 165, 187, 293\}$, $\{2, 32, 173, 258\}$, $\{2, 137, 180, 229\}$,\adfsplit
 $\{1, 138, 153, 188\}$, $\{1, 67, 313, 321\}$, $\{2, 47, 109, 111\}$, $\{0, 146, 172, 250\}$,\adfsplit
 $\{2, 37, 257, 279\}$, $\{1, 111, 152, 155\}$, $\{1, 109, 165, 299\}$, $\{0, 118, 137, 292\}$,\adfsplit
 $\{1, 117, 144, 300\}$, $\{1, 96, 121, 319\}$, $\{1, 195, 257, 277\}$, $\{2, 89, 123, 293\}$,\adfsplit
 $\{2, 88, 128, 250\}$, $\{0, 91, 265, 321\}$, $\{0, 74, 144, 216\}$, $\{0, 58, 228, 307\}$,\adfsplit
 $\{1, 128, 251, 314\}$, $\{0, 60, 111, 240\}$, $\{0, 76, 177, 214\}$, $\{1, 43, 132, 289\}$,\adfsplit
 $\{0, 125, 176, 291\}$, $\{2, 60, 181, 326\}$, $\{2, 67, 167, 191\}$, $\{0, 121, 200, 223\}$,\adfsplit
 $\{1, 60, 97, 197\}$, $\{0, 102, 149, 237\}$, $\{1, 102, 298, 312\}$, $\{1, 36, 186, 284\}$,\adfsplit
 $\{1, 130, 235, 296\}$, $\{2, 16, 228, 263\}$, $\{0, 50, 130, 207\}$, $\{0, 43, 57, 298\}$,\adfsplit
 $\{0, 42, 126, 319\}$, $\{0, 81, 107, 197\}$, $\{0, 88, 155, 168\}$, $\{0, 72, 256, 310\}$,\adfsplit
 $\{0, 71, 114, 270\}$, $\{0, 86, 205, 247\}$, $\{0, 163, 184, 211\}$, $\{1, 72, 226, 254\}$,\adfsplit
 $\{0, 64, 148, 170\}$, $\{0, 85, 175, 268\}$, $\{1, 57, 135, 260\}$, $\{0, 100, 105, 274\}$,\adfsplit
 $\{0, 134, 233, 246\}$, $\{0, 77, 183, 275\}$, $\{0, 99, 133, 323\}$, $\{0, 127, 225, 288\}$,\adfsplit
 $\{0, 28, 232, 294\}$
%adfPENTend
}

\noindent under the action of the mapping $x \mapsto x + 7 \adfmod{329}$.
The deficiency graph is connected and has girth 5.
% The opposite lines are generated by the first 7 blocks: True
% End of PENT(4, 108)
%%%%%%%%%%%%%%%%%%%%%%%%%%%%%%%%%%%%%%%%%%%%%%%%%%%%%%%%%%%%%%%%%%%%%%%%%%%%%%%%%%%%%%%%%%

{\noindent\boldmath $\adfPENT(4, 109)$}~ With point set $Z_{332}$, the 9047 lines are generated from

{\scriptsize
%adfPENTstart:1:4:109:4:4:
 $\{64, 239, 268, 317\}$, $\{16, 53, 182, 281\}$, $\{98, 153, 155, 238\}$, $\{87, 96, 182, 251\}$,\adfsplit
 $\{187, 179, 207, 170\}$, $\{149, 255, 291, 308\}$, $\{63, 173, 202, 201\}$, $\{265, 70, 61, 42\}$,\adfsplit
 $\{240, 17, 4, 192\}$, $\{153, 58, 218, 145\}$, $\{78, 309, 76, 228\}$, $\{257, 164, 25, 290\}$,\adfsplit
 $\{174, 296, 199, 181\}$, $\{64, 13, 162, 183\}$, $\{200, 244, 231, 155\}$, $\{23, 210, 111, 20\}$,\adfsplit
 $\{149, 69, 209, 296\}$, $\{164, 145, 139, 2\}$, $\{268, 11, 82, 38\}$, $\{136, 249, 141, 126\}$,\adfsplit
 $\{171, 93, 45, 181\}$, $\{111, 84, 172, 116\}$, $\{71, 290, 302, 24\}$, $\{51, 13, 102, 178\}$,\adfsplit
 $\{87, 160, 188, 203\}$, $\{125, 179, 0, 90\}$, $\{152, 58, 255, 214\}$, $\{137, 201, 174, 161\}$,\adfsplit
 $\{133, 74, 155, 58\}$, $\{85, 223, 207, 6\}$, $\{165, 50, 241, 275\}$, $\{138, 135, 63, 327\}$,\adfsplit
 $\{79, 276, 93, 39\}$, $\{279, 5, 280, 60\}$, $\{30, 159, 241, 95\}$, $\{179, 40, 46, 120\}$,\adfsplit
 $\{1, 321, 108, 84\}$, $\{316, 73, 0, 184\}$, $\{213, 13, 243, 197\}$, $\{0, 1, 8, 85\}$,\adfsplit
 $\{0, 4, 21, 141\}$, $\{0, 7, 9, 181\}$, $\{0, 10, 117, 121\}$, $\{0, 11, 29, 209\}$,\adfsplit
 $\{0, 12, 45, 309\}$, $\{0, 14, 25, 301\}$, $\{0, 18, 41, 205\}$, $\{0, 19, 61, 133\}$,\adfsplit
 $\{0, 20, 189, 305\}$, $\{0, 22, 177, 197\}$, $\{0, 23, 53, 261\}$, $\{0, 39, 101, 145\}$,\adfsplit
 $\{0, 26, 69, 289\}$, $\{0, 30, 161, 257\}$, $\{0, 35, 201, 293\}$, $\{0, 34, 65, 97\}$,\adfsplit
 $\{0, 38, 165, 321\}$, $\{0, 55, 129, 273\}$, $\{1, 6, 7, 297\}$, $\{0, 63, 153, 251\}$,\adfsplit
 $\{0, 51, 58, 229\}$, $\{0, 54, 115, 329\}$, $\{0, 71, 123, 241\}$, $\{0, 67, 147, 269\}$,\adfsplit
 $\{1, 22, 26, 267\}$, $\{1, 15, 27, 151\}$, $\{1, 11, 54, 59\}$, $\{1, 51, 62, 235\}$,\adfsplit
 $\{1, 18, 63, 287\}$, $\{1, 71, 118, 191\}$, $\{1, 67, 82, 199\}$, $\{1, 98, 134, 283\}$,\adfsplit
 $\{1, 83, 183, 311\}$, $\{1, 103, 126, 214\}$, $\{1, 175, 179, 330\}$, $\{1, 147, 203, 299\}$,\adfsplit
 $\{1, 87, 154, 247\}$, $\{1, 114, 186, 239\}$, $\{1, 263, 282, 295\}$, $\{1, 91, 194, 262\}$,\adfsplit
 $\{1, 174, 207, 286\}$, $\{1, 170, 230, 250\}$, $\{1, 78, 242, 294\}$, $\{1, 94, 198, 307\}$,\adfsplit
 $\{0, 36, 143, 299\}$, $\{0, 40, 151, 211\}$, $\{0, 50, 79, 191\}$, $\{0, 52, 183, 275\}$,\adfsplit
 $\{0, 104, 267, 291\}$, $\{0, 76, 159, 203\}$, $\{0, 42, 120, 207\}$, $\{0, 68, 138, 295\}$,\adfsplit
 $\{0, 110, 195, 216\}$, $\{0, 215, 250, 274\}$, $\{0, 178, 255, 310\}$, $\{0, 142, 198, 247\}$,\adfsplit
 $\{0, 99, 130, 232\}$, $\{0, 114, 162, 283\}$, $\{0, 167, 206, 290\}$, $\{0, 134, 172, 330\}$,\adfsplit
 $\{0, 94, 218, 282\}$, $\{0, 60, 214, 306\}$, $\{0, 66, 176, 298\}$, $\{0, 78, 262, 302\}$,\adfsplit
 $\{0, 46, 108, 192\}$, $\{0, 82, 234, 240\}$, $\{0, 74, 194, 202\}$, $\{0, 118, 150, 164\}$,\adfsplit
 $\{0, 106, 124, 196\}$
%adfPENTend
}

\noindent under the action of the mapping $x \mapsto x + 4 \adfmod{332}$.
The deficiency graph is connected and has girth 7.
% The opposite lines are generated by the first 4 blocks: True
% End of PENT(4, 109)
%%%%%%%%%%%%%%%%%%%%%%%%%%%%%%%%%%%%%%%%%%%%%%%%%%%%%%%%%%%%%%%%%%%%%%%%%%%%%%%%%%%%%%%%%%

{\noindent\boldmath $\adfPENT(4, 117)$}~ With point set $Z_{356}$, the 10413 lines are generated from

{\scriptsize
%adfPENTstart:1:4:117:4:4:
 $\{100, 256, 285, 335\}$, $\{30, 72, 110, 139\}$, $\{114, 246, 249, 329\}$, $\{24, 43, 221, 319\}$,\adfsplit
 $\{176, 167, 119, 250\}$, $\{163, 287, 347, 329\}$, $\{130, 175, 260, 205\}$, $\{79, 94, 120, 301\}$,\adfsplit
 $\{329, 328, 217, 131\}$, $\{61, 1, 197, 76\}$, $\{62, 1, 322, 143\}$, $\{238, 121, 268, 187\}$,\adfsplit
 $\{245, 19, 8, 73\}$, $\{125, 166, 277, 251\}$, $\{322, 4, 6, 111\}$, $\{296, 232, 101, 344\}$,\adfsplit
 $\{101, 39, 111, 40\}$, $\{332, 169, 121, 2\}$, $\{69, 222, 319, 176\}$, $\{24, 140, 2, 265\}$,\adfsplit
 $\{234, 91, 26, 104\}$, $\{193, 155, 14, 47\}$, $\{117, 218, 187, 66\}$, $\{220, 62, 276, 103\}$,\adfsplit
 $\{175, 298, 182, 117\}$, $\{74, 122, 256, 3\}$, $\{294, 305, 224, 307\}$, $\{222, 175, 253, 197\}$,\adfsplit
 $\{63, 275, 147, 224\}$, $\{206, 252, 312, 88\}$, $\{125, 165, 129, 61\}$, $\{20, 47, 209, 50\}$,\adfsplit
 $\{51, 188, 183, 9\}$, $\{149, 311, 313, 198\}$, $\{332, 122, 53, 28\}$, $\{145, 63, 43, 253\}$,\adfsplit
 $\{177, 333, 1, 22\}$, $\{91, 327, 65, 318\}$, $\{273, 103, 202, 324\}$, $\{126, 67, 202, 193\}$,\adfsplit
 $\{54, 91, 120, 244\}$, $\{14, 138, 291, 139\}$, $\{268, 262, 258, 70\}$, $\{351, 290, 99, 263\}$,\adfsplit
 $\{259, 33, 282, 177\}$, $\{14, 42, 331, 30\}$, $\{309, 53, 322, 339\}$, $\{32, 53, 351, 129\}$,\adfsplit
 $\{58, 63, 218, 48\}$, $\{110, 44, 246, 273\}$, $\{110, 37, 162, 18\}$, $\{117, 72, 355, 239\}$,\adfsplit
 $\{0, 3, 9, 37\}$, $\{0, 4, 93, 105\}$, $\{0, 5, 13, 133\}$, $\{0, 6, 49, 137\}$,\adfsplit
 $\{0, 7, 17, 157\}$, $\{0, 8, 41, 273\}$, $\{0, 12, 85, 129\}$, $\{0, 14, 69, 201\}$,\adfsplit
 $\{0, 16, 221, 337\}$, $\{0, 18, 109, 317\}$, $\{0, 23, 165, 333\}$, $\{0, 20, 277, 349\}$,\adfsplit
 $\{0, 22, 149, 169\}$, $\{0, 31, 217, 269\}$, $\{0, 42, 57, 141\}$, $\{0, 28, 261, 353\}$,\adfsplit
 $\{0, 43, 313, 345\}$, $\{0, 55, 213, 309\}$, $\{0, 34, 253, 312\}$, $\{0, 47, 230, 289\}$,\adfsplit
 $\{0, 35, 126, 173\}$, $\{0, 39, 53, 134\}$, $\{1, 2, 23, 341\}$, $\{0, 58, 153, 227\}$,\adfsplit
 $\{0, 87, 98, 293\}$, $\{1, 6, 38, 294\}$, $\{1, 15, 78, 186\}$, $\{1, 34, 90, 318\}$,\adfsplit
 $\{1, 79, 98, 278\}$, $\{1, 35, 111, 334\}$, $\{1, 86, 122, 206\}$, $\{1, 7, 146, 166\}$,\adfsplit
 $\{1, 54, 114, 158\}$, $\{1, 234, 259, 311\}$, $\{1, 63, 174, 243\}$, $\{1, 107, 214, 307\}$,\adfsplit
 $\{1, 91, 183, 218\}$, $\{1, 155, 267, 291\}$, $\{1, 130, 203, 247\}$, $\{1, 134, 283, 339\}$,\adfsplit
 $\{1, 95, 235, 350\}$, $\{0, 24, 86, 175\}$, $\{0, 32, 114, 286\}$, $\{0, 36, 274, 282\}$,\adfsplit
 $\{0, 75, 270, 294\}$, $\{0, 80, 258, 322\}$, $\{0, 106, 194, 248\}$, $\{0, 154, 168, 354\}$,\adfsplit
 $\{0, 90, 139, 264\}$, $\{0, 50, 190, 243\}$, $\{0, 110, 128, 291\}$, $\{2, 255, 271, 283\}$,\adfsplit
 $\{0, 131, 148, 298\}$, $\{0, 54, 204, 272\}$, $\{0, 102, 196, 311\}$, $\{0, 88, 210, 260\}$,\adfsplit
 $\{0, 104, 223, 266\}$, $\{0, 123, 176, 331\}$, $\{0, 76, 171, 220\}$, $\{0, 111, 147, 216\}$,\adfsplit
 $\{0, 59, 91, 236\}$, $\{0, 40, 247, 255\}$, $\{0, 99, 199, 267\}$, $\{0, 72, 259, 263\}$,\adfsplit
 $\{0, 63, 127, 323\}$
%adfPENTend
}

\noindent under the action of the mapping $x \mapsto x + 4 \adfmod{356}$.
The deficiency graph is connected and has girth 6.
% The opposite lines are generated by the first 4 blocks: True
% End of PENT(4, 117)
%%%%%%%%%%%%%%%%%%%%%%%%%%%%%%%%%%%%%%%%%%%%%%%%%%%%%%%%%%%%%%%%%%%%%%%%%%%%%%%%%%%%%%%%%%

{\noindent\boldmath $\adfPENT(4, 120)$}~ With point set $Z_{365}$, the 10950 lines are generated from

{\scriptsize
%adfPENTstart:1:4:120:5:5:
 $\{82, 124, 276, 348\}$, $\{90, 144, 208, 263\}$, $\{12, 209, 285, 357\}$, $\{20, 106, 161, 179\}$,\adfsplit
 $\{162, 193, 226, 245\}$, $\{280, 135, 47, 114\}$, $\{328, 301, 174, 193\}$, $\{352, 281, 26, 17\}$,\adfsplit
 $\{341, 206, 361, 94\}$, $\{76, 199, 109, 54\}$, $\{166, 77, 120, 0\}$, $\{194, 286, 149, 202\}$,\adfsplit
 $\{259, 57, 71, 239\}$, $\{312, 45, 201, 247\}$, $\{128, 238, 138, 276\}$, $\{51, 104, 354, 323\}$,\adfsplit
 $\{284, 234, 22, 51\}$, $\{56, 312, 305, 248\}$, $\{252, 135, 66, 49\}$, $\{58, 187, 183, 243\}$,\adfsplit
 $\{55, 230, 18, 338\}$, $\{3, 328, 346, 168\}$, $\{117, 133, 4, 294\}$, $\{271, 55, 237, 260\}$,\adfsplit
 $\{142, 269, 273, 348\}$, $\{356, 48, 118, 19\}$, $\{332, 269, 317, 230\}$, $\{130, 188, 297, 117\}$,\adfsplit
 $\{19, 165, 359, 61\}$, $\{196, 312, 47, 272\}$, $\{137, 243, 33, 230\}$, $\{267, 98, 325, 147\}$,\adfsplit
 $\{151, 153, 114, 183\}$, $\{339, 305, 146, 94\}$, $\{100, 194, 301, 278\}$, $\{239, 178, 357, 231\}$,\adfsplit
 $\{113, 197, 296, 37\}$, $\{294, 6, 306, 330\}$, $\{134, 349, 35, 216\}$, $\{66, 165, 179, 31\}$,\adfsplit
 $\{351, 88, 201, 50\}$, $\{238, 50, 164, 350\}$, $\{269, 26, 148, 66\}$, $\{196, 100, 222, 363\}$,\adfsplit
 $\{29, 210, 36, 258\}$, $\{43, 322, 252, 171\}$, $\{164, 177, 174, 343\}$, $\{251, 32, 18, 23\}$,\adfsplit
 $\{214, 303, 314, 54\}$, $\{225, 254, 144, 15\}$, $\{178, 181, 175, 177\}$, $\{188, 264, 318, 40\}$,\adfsplit
 $\{111, 159, 50, 80\}$, $\{125, 256, 246, 259\}$, $\{71, 246, 251, 201\}$, $\{283, 10, 364, 349\}$,\adfsplit
 $\{224, 302, 60, 141\}$, $\{253, 227, 304, 12\}$, $\{186, 46, 364, 130\}$, $\{299, 77, 192, 228\}$,\adfsplit
 $\{109, 266, 287, 13\}$, $\{138, 224, 265, 150\}$, $\{304, 209, 43, 9\}$, $\{238, 127, 315, 171\}$,\adfsplit
 $\{102, 177, 127, 58\}$, $\{291, 328, 167, 208\}$, $\{181, 284, 300, 182\}$, $\{172, 64, 259, 104\}$,\adfsplit
 $\{329, 362, 66, 78\}$, $\{305, 314, 157, 291\}$, $\{0, 1, 59, 139\}$, $\{0, 4, 304, 334\}$,\adfsplit
 $\{0, 5, 89, 149\}$, $\{0, 8, 24, 264\}$, $\{0, 12, 19, 254\}$, $\{0, 15, 119, 204\}$,\adfsplit
 $\{0, 17, 69, 244\}$, $\{0, 16, 169, 309\}$, $\{0, 22, 44, 49\}$, $\{0, 18, 64, 249\}$,\adfsplit
 $\{0, 23, 229, 364\}$, $\{0, 25, 171, 299\}$, $\{0, 37, 227, 319\}$, $\{0, 10, 269, 327\}$,\adfsplit
 $\{0, 27, 174, 298\}$, $\{0, 42, 197, 209\}$, $\{0, 26, 262, 359\}$, $\{0, 62, 172, 294\}$,\adfsplit
 $\{1, 7, 12, 44\}$, $\{0, 63, 152, 214\}$, $\{1, 8, 164, 182\}$, $\{1, 16, 57, 174\}$,\adfsplit
 $\{1, 17, 63, 89\}$, $\{1, 26, 143, 364\}$, $\{1, 18, 39, 336\}$, $\{1, 32, 224, 317\}$,\adfsplit
 $\{1, 53, 197, 279\}$, $\{1, 52, 109, 276\}$, $\{1, 78, 137, 249\}$, $\{1, 64, 87, 108\}$,\adfsplit
 $\{1, 93, 198, 309\}$, $\{1, 97, 268, 269\}$, $\{2, 4, 63, 178\}$, $\{1, 79, 122, 338\}$,\adfsplit
 $\{1, 67, 126, 219\}$, $\{2, 8, 279, 307\}$, $\{2, 13, 92, 139\}$, $\{2, 98, 134, 183\}$,\adfsplit
 $\{2, 138, 233, 239\}$, $\{2, 103, 144, 242\}$, $\{2, 19, 293, 328\}$, $\{3, 53, 184, 193\}$,\adfsplit
 $\{2, 128, 329, 343\}$, $\{0, 20, 162, 332\}$, $\{0, 32, 35, 292\}$, $\{0, 45, 157, 192\}$,\adfsplit
 $\{0, 50, 302, 347\}$, $\{0, 21, 57, 222\}$, $\{0, 51, 207, 337\}$, $\{0, 47, 305, 357\}$,\adfsplit
 $\{0, 36, 287, 338\}$, $\{0, 78, 232, 286\}$, $\{1, 71, 247, 342\}$, $\{0, 66, 127, 271\}$,\adfsplit
 $\{2, 87, 203, 222\}$, $\{0, 67, 241, 295\}$, $\{1, 121, 293, 347\}$, $\{1, 86, 201, 292\}$,\adfsplit
 $\{0, 71, 237, 358\}$, $\{0, 183, 212, 316\}$, $\{1, 98, 178, 212\}$, $\{0, 97, 243, 331\}$,\adfsplit
 $\{0, 53, 176, 256\}$, $\{1, 76, 171, 323\}$, $\{0, 68, 76, 136\}$, $\{0, 91, 130, 356\}$,\adfsplit
 $\{0, 161, 248, 306\}$, $\{0, 40, 203, 251\}$, $\{0, 186, 291, 333\}$, $\{0, 85, 195, 321\}$,\adfsplit
 $\{0, 98, 111, 140\}$, $\{0, 43, 90, 196\}$, $\{0, 101, 105, 303\}$, $\{0, 83, 103, 240\}$,\adfsplit
 $\{0, 28, 93, 200\}$, $\{0, 75, 218, 233\}$, $\{0, 123, 268, 293\}$, $\{0, 88, 230, 363\}$,\adfsplit
 $\{0, 33, 185, 285\}$, $\{0, 55, 128, 150\}$
%adfPENTend
}

\noindent under the action of the mapping $x \mapsto x + 5 \adfmod{365}$.
The deficiency graph is connected and has girth 6.
% The opposite lines are generated by the first 5 blocks: True
% End of PENT(4, 120)
%%%%%%%%%%%%%%%%%%%%%%%%%%%%%%%%%%%%%%%%%%%%%%%%%%%%%%%%%%%%%%%%%%%%%%%%%%%%%%%%%%%%%%%%%%

{\noindent\boldmath $\adfPENT(4, 125)$}~ With point set $Z_{380}$, the 11875 lines are generated from

{\scriptsize
%adfPENTstart:1:4:125:4:4:
 $\{33, 230, 362, 373\}$, $\{8, 111, 226, 348\}$, $\{20, 91, 152, 157\}$, $\{175, 211, 273, 294\}$,\adfsplit
 $\{223, 70, 219, 226\}$, $\{54, 124, 355, 341\}$, $\{196, 7, 344, 358\}$, $\{166, 81, 257, 142\}$,\adfsplit
 $\{5, 108, 218, 308\}$, $\{255, 352, 292, 156\}$, $\{238, 365, 281, 146\}$, $\{330, 74, 86, 325\}$,\adfsplit
 $\{231, 67, 156, 181\}$, $\{236, 123, 293, 183\}$, $\{114, 131, 12, 10\}$, $\{18, 235, 347, 357\}$,\adfsplit
 $\{366, 80, 200, 154\}$, $\{193, 350, 8, 0\}$, $\{255, 61, 86, 299\}$, $\{127, 109, 274, 82\}$,\adfsplit
 $\{202, 167, 112, 346\}$, $\{52, 276, 160, 31\}$, $\{91, 117, 169, 154\}$, $\{94, 235, 28, 366\}$,\adfsplit
 $\{266, 198, 17, 276\}$, $\{61, 114, 377, 287\}$, $\{233, 183, 225, 20\}$, $\{371, 306, 123, 372\}$,\adfsplit
 $\{207, 5, 52, 33\}$, $\{198, 114, 194, 43\}$, $\{127, 95, 0, 147\}$, $\{138, 141, 310, 17\}$,\adfsplit
 $\{320, 134, 208, 115\}$, $\{335, 245, 324, 319\}$, $\{294, 319, 149, 366\}$, $\{271, 269, 284, 335\}$,\adfsplit
 $\{200, 138, 62, 159\}$, $\{17, 237, 263, 363\}$, $\{373, 237, 291, 85\}$, $\{124, 292, 29, 23\}$,\adfsplit
 $\{343, 336, 313, 195\}$, $\{173, 236, 15, 159\}$, $\{20, 166, 114, 146\}$, $\{69, 272, 281, 100\}$,\adfsplit
 $\{10, 181, 312, 111\}$, $\{38, 74, 280, 204\}$, $\{302, 40, 28, 375\}$, $\{186, 265, 160, 299\}$,\adfsplit
 $\{255, 68, 150, 279\}$, $\{360, 134, 1, 234\}$, $\{56, 345, 349, 289\}$, $\{148, 97, 338, 154\}$,\adfsplit
 $\{353, 373, 351, 95\}$, $\{156, 296, 314, 17\}$, $\{201, 75, 338, 147\}$, $\{0, 1, 34, 150\}$,\adfsplit
 $\{0, 2, 3, 62\}$, $\{0, 4, 182, 226\}$, $\{0, 10, 38, 278\}$, $\{0, 13, 282, 322\}$,\adfsplit
 $\{0, 15, 30, 210\}$, $\{0, 16, 114, 202\}$, $\{0, 17, 130, 294\}$, $\{0, 19, 86, 238\}$,\adfsplit
 $\{0, 22, 142, 352\}$, $\{0, 23, 270, 326\}$, $\{0, 20, 366, 374\}$, $\{0, 27, 70, 118\}$,\adfsplit
 $\{0, 31, 106, 358\}$, $\{0, 35, 58, 122\}$, $\{1, 7, 46, 194\}$, $\{1, 2, 18, 369\}$,\adfsplit
 $\{0, 42, 47, 246\}$, $\{0, 37, 46, 266\}$, $\{0, 29, 134, 169\}$, $\{1, 11, 66, 350\}$,\adfsplit
 $\{0, 32, 330, 363\}$, $\{0, 54, 63, 263\}$, $\{0, 39, 87, 198\}$, $\{0, 79, 206, 299\}$,\adfsplit
 $\{1, 17, 98, 165\}$, $\{1, 23, 126, 163\}$, $\{1, 25, 63, 202\}$, $\{1, 33, 102, 309\}$,\adfsplit
 $\{1, 37, 81, 342\}$, $\{1, 47, 269, 318\}$, $\{1, 59, 90, 249\}$, $\{1, 43, 209, 362\}$,\adfsplit
 $\{1, 49, 117, 374\}$, $\{1, 83, 95, 190\}$, $\{2, 59, 87, 127\}$, $\{1, 71, 206, 343\}$,\adfsplit
 $\{1, 187, 275, 286\}$, $\{1, 77, 186, 363\}$, $\{1, 119, 307, 334\}$, $\{1, 99, 274, 323\}$,\adfsplit
 $\{1, 74, 155, 231\}$, $\{1, 78, 101, 335\}$, $\{1, 87, 179, 322\}$, $\{1, 282, 295, 351\}$,\adfsplit
 $\{0, 24, 91, 243\}$, $\{0, 48, 171, 355\}$, $\{0, 107, 227, 311\}$, $\{0, 83, 199, 295\}$,\adfsplit
 $\{0, 73, 151, 255\}$, $\{1, 143, 151, 279\}$, $\{0, 97, 235, 315\}$, $\{0, 59, 192, 273\}$,\adfsplit
 $\{0, 89, 223, 280\}$, $\{0, 36, 203, 341\}$, $\{0, 69, 183, 288\}$, $\{0, 49, 195, 204\}$,\adfsplit
 $\{0, 61, 144, 259\}$, $\{0, 109, 215, 345\}$, $\{0, 80, 221, 351\}$, $\{0, 44, 197, 325\}$,\adfsplit
 $\{0, 93, 228, 353\}$, $\{0, 53, 96, 253\}$, $\{0, 64, 149, 237\}$, $\{0, 52, 165, 261\}$,\adfsplit
 $\{0, 45, 229, 296\}$, $\{0, 56, 128, 216\}$, $\{0, 65, 124, 257\}$, $\{0, 68, 269, 377\}$,\adfsplit
 $\{0, 41, 265, 276\}$
%adfPENTend
}

\noindent under the action of the mapping $x \mapsto x + 4 \adfmod{380}$.
The deficiency graph is connected and has girth 7.
% The opposite lines are generated by the first 4 blocks: True
% End of PENT(4, 125)
%%%%%%%%%%%%%%%%%%%%%%%%%%%%%%%%%%%%%%%%%%%%%%%%%%%%%%%%%%%%%%%%%%%%%%%%%%%%%%%%%%%%%%%%%%

{\noindent\boldmath $\adfPENT(4, 133)$}~ With point set $Z_{404}$, the 13433 lines are generated from

{\scriptsize
%adfPENTstart:1:4:133:4:4:
 $\{69, 83, 193, 202\}$, $\{203, 212, 336, 359\}$, $\{202, 204, 206, 235\}$, $\{49, 174, 205, 324\}$,\adfsplit
 $\{240, 376, 89, 63\}$, $\{197, 17, 60, 185\}$, $\{371, 155, 390, 318\}$, $\{37, 1, 97, 355\}$,\adfsplit
 $\{279, 330, 78, 213\}$, $\{51, 231, 31, 33\}$, $\{44, 368, 76, 401\}$, $\{110, 241, 84, 218\}$,\adfsplit
 $\{279, 36, 70, 262\}$, $\{318, 135, 88, 131\}$, $\{327, 43, 161, 263\}$, $\{280, 90, 228, 333\}$,\adfsplit
 $\{28, 238, 401, 159\}$, $\{224, 344, 352, 119\}$, $\{221, 68, 378, 142\}$, $\{280, 166, 391, 159\}$,\adfsplit
 $\{110, 105, 348, 326\}$, $\{325, 224, 357, 66\}$, $\{89, 93, 178, 110\}$, $\{165, 317, 176, 158\}$,\adfsplit
 $\{240, 298, 174, 353\}$, $\{201, 128, 11, 301\}$, $\{81, 219, 272, 394\}$, $\{274, 77, 84, 284\}$,\adfsplit
 $\{72, 398, 24, 335\}$, $\{77, 369, 15, 227\}$, $\{397, 12, 381, 289\}$, $\{185, 38, 23, 157\}$,\adfsplit
 $\{162, 292, 277, 26\}$, $\{276, 103, 316, 190\}$, $\{106, 366, 278, 231\}$, $\{352, 196, 74, 195\}$,\adfsplit
 $\{297, 11, 320, 326\}$, $\{69, 229, 335, 209\}$, $\{48, 394, 166, 225\}$, $\{363, 91, 309, 258\}$,\adfsplit
 $\{24, 183, 222, 170\}$, $\{258, 266, 77, 119\}$, $\{336, 274, 366, 184\}$, $\{308, 166, 72, 248\}$,\adfsplit
 $\{76, 197, 13, 122\}$, $\{346, 287, 266, 265\}$, $\{297, 130, 190, 16\}$, $\{70, 52, 56, 325\}$,\adfsplit
 $\{52, 119, 175, 297\}$, $\{322, 309, 100, 8\}$, $\{94, 170, 139, 255\}$, $\{193, 130, 372, 273\}$,\adfsplit
 $\{225, 346, 51, 160\}$, $\{325, 351, 273, 75\}$, $\{218, 175, 289, 265\}$, $\{57, 113, 345, 84\}$,\adfsplit
 $\{123, 33, 115, 137\}$, $\{202, 44, 227, 391\}$, $\{159, 49, 317, 308\}$, $\{265, 32, 353, 311\}$,\adfsplit
 $\{295, 353, 199, 165\}$, $\{372, 186, 122, 348\}$, $\{378, 107, 367, 337\}$, $\{274, 313, 109, 101\}$,\adfsplit
 $\{61, 119, 64, 0\}$, $\{268, 114, 354, 284\}$, $\{190, 398, 150, 199\}$, $\{126, 260, 93, 55\}$,\adfsplit
 $\{0, 1, 45, 165\}$, $\{0, 3, 37, 77\}$, $\{0, 7, 217, 265\}$, $\{0, 10, 21, 149\}$,\adfsplit
 $\{0, 5, 188, 201\}$, $\{0, 11, 17, 81\}$, $\{0, 15, 109, 185\}$, $\{0, 19, 97, 169\}$,\adfsplit
 $\{0, 22, 89, 349\}$, $\{0, 25, 27, 297\}$, $\{0, 20, 249, 333\}$, $\{0, 38, 189, 257\}$,\adfsplit
 $\{0, 36, 93, 241\}$, $\{1, 7, 229, 282\}$, $\{0, 35, 298, 317\}$, $\{0, 49, 98, 110\}$,\adfsplit
 $\{0, 12, 54, 365\}$, $\{0, 28, 103, 337\}$, $\{0, 50, 82, 85\}$, $\{0, 102, 289, 358\}$,\adfsplit
 $\{0, 41, 51, 116\}$, $\{0, 39, 99, 345\}$, $\{0, 78, 79, 181\}$, $\{0, 107, 293, 398\}$,\adfsplit
 $\{1, 38, 58, 74\}$, $\{1, 39, 71, 378\}$, $\{1, 26, 87, 210\}$, $\{1, 63, 98, 214\}$,\adfsplit
 $\{1, 46, 130, 310\}$, $\{1, 95, 206, 318\}$, $\{1, 75, 178, 202\}$, $\{1, 143, 234, 362\}$,\adfsplit
 $\{1, 179, 191, 322\}$, $\{1, 102, 279, 351\}$, $\{1, 138, 194, 275\}$, $\{1, 211, 278, 306\}$,\adfsplit
 $\{1, 66, 223, 239\}$, $\{1, 250, 294, 395\}$, $\{1, 299, 323, 375\}$, $\{1, 230, 307, 390\}$,\adfsplit
 $\{1, 142, 339, 387\}$, $\{0, 63, 66, 350\}$, $\{0, 62, 362, 367\}$, $\{0, 88, 194, 294\}$,\adfsplit
 $\{0, 68, 238, 370\}$, $\{0, 139, 330, 378\}$, $\{0, 130, 167, 180\}$, $\{2, 43, 98, 155\}$,\adfsplit
 $\{0, 142, 211, 240\}$, $\{2, 95, 195, 331\}$, $\{0, 59, 104, 390\}$, $\{0, 76, 184, 334\}$,\adfsplit
 $\{0, 259, 343, 366\}$, $\{2, 151, 291, 379\}$, $\{0, 144, 319, 363\}$, $\{0, 96, 196, 399\}$,\adfsplit
 $\{0, 132, 327, 355\}$, $\{0, 72, 215, 323\}$, $\{0, 44, 192, 379\}$, $\{0, 135, 232, 387\}$,\adfsplit
 $\{0, 71, 163, 267\}$, $\{0, 235, 315, 383\}$, $\{0, 87, 127, 160\}$, $\{0, 56, 207, 320\}$,\adfsplit
 $\{0, 115, 239, 275\}$
%adfPENTend
}

\noindent under the action of the mapping $x \mapsto x + 4 \adfmod{404}$.
The deficiency graph is connected and has girth 6.
% The opposite lines are generated by the first 4 blocks: True
% End of PENT(4, 133)
%%%%%%%%%%%%%%%%%%%%%%%%%%%%%%%%%%%%%%%%%%%%%%%%%%%%%%%%%%%%%%%%%%%%%%%%%%%%%%%%%%%%%%%%%%

{\noindent\boldmath $\adfPENT(4, 140)$}~ With point set $Z_{425}$, the 14875 lines are generated from

{\scriptsize
%adfPENTstart:1:4:140:5:5:
 $\{198, 249, 408, 424\}$, $\{42, 146, 252, 281\}$, $\{38, 176, 228, 386\}$, $\{20, 202, 230, 392\}$,\adfsplit
 $\{5, 159, 180, 274\}$, $\{287, 207, 228, 43\}$, $\{265, 164, 18, 76\}$, $\{161, 366, 78, 205\}$,\adfsplit
 $\{29, 420, 131, 253\}$, $\{194, 378, 365, 214\}$, $\{221, 289, 364, 399\}$, $\{63, 396, 330, 249\}$,\adfsplit
 $\{338, 234, 328, 368\}$, $\{256, 66, 119, 214\}$, $\{400, 164, 391, 188\}$, $\{138, 17, 25, 412\}$,\adfsplit
 $\{185, 204, 163, 311\}$, $\{169, 219, 319, 67\}$, $\{165, 382, 134, 54\}$, $\{393, 185, 81, 305\}$,\adfsplit
 $\{228, 252, 125, 19\}$, $\{361, 168, 164, 189\}$, $\{160, 168, 399, 2\}$, $\{418, 155, 22, 61\}$,\adfsplit
 $\{34, 342, 253, 248\}$, $\{394, 192, 113, 409\}$, $\{282, 423, 55, 411\}$, $\{214, 368, 196, 319\}$,\adfsplit
 $\{259, 395, 45, 169\}$, $\{141, 107, 383, 7\}$, $\{107, 198, 422, 231\}$, $\{405, 110, 62, 128\}$,\adfsplit
 $\{177, 25, 382, 126\}$, $\{181, 356, 97, 363\}$, $\{6, 3, 344, 205\}$, $\{58, 135, 43, 407\}$,\adfsplit
 $\{235, 378, 176, 316\}$, $\{301, 239, 36, 282\}$, $\{320, 5, 154, 415\}$, $\{240, 374, 178, 396\}$,\adfsplit
 $\{133, 367, 172, 77\}$, $\{0, 57, 200, 322\}$, $\{413, 69, 227, 381\}$, $\{201, 331, 352, 417\}$,\adfsplit
 $\{415, 417, 69, 247\}$, $\{402, 348, 390, 29\}$, $\{230, 143, 315, 278\}$, $\{397, 271, 255, 94\}$,\adfsplit
 $\{371, 30, 357, 82\}$, $\{69, 307, 88, 123\}$, $\{26, 113, 54, 294\}$, $\{202, 349, 131, 41\}$,\adfsplit
 $\{384, 311, 259, 202\}$, $\{297, 351, 181, 280\}$, $\{382, 69, 362, 294\}$, $\{314, 303, 196, 140\}$,\adfsplit
 $\{272, 126, 305, 388\}$, $\{123, 184, 248, 190\}$, $\{243, 75, 117, 120\}$, $\{369, 365, 299, 132\}$,\adfsplit
 $\{421, 1, 237, 197\}$, $\{403, 412, 107, 364\}$, $\{80, 416, 15, 324\}$, $\{24, 395, 62, 417\}$,\adfsplit
 $\{177, 39, 110, 221\}$, $\{377, 35, 314, 12\}$, $\{111, 207, 395, 208\}$, $\{238, 316, 105, 183\}$,\adfsplit
 $\{165, 239, 361, 324\}$, $\{206, 55, 39, 245\}$, $\{278, 380, 258, 423\}$, $\{179, 22, 95, 193\}$,\adfsplit
 $\{191, 318, 187, 114\}$, $\{401, 51, 310, 168\}$, $\{217, 393, 148, 48\}$, $\{118, 22, 72, 81\}$,\adfsplit
 $\{387, 222, 24, 240\}$, $\{166, 267, 209, 4\}$, $\{14, 21, 103, 332\}$, $\{315, 189, 402, 69\}$,\adfsplit
 $\{307, 140, 175, 208\}$, $\{23, 378, 273, 376\}$, $\{346, 138, 253, 53\}$, $\{148, 24, 353, 330\}$,\adfsplit
 $\{21, 206, 314, 130\}$, $\{274, 283, 387, 104\}$, $\{62, 68, 199, 192\}$, $\{169, 195, 123, 196\}$,\adfsplit
 $\{187, 46, 110, 265\}$, $\{403, 135, 411, 247\}$, $\{0, 5, 188, 318\}$, $\{0, 6, 73, 233\}$,\adfsplit
 $\{0, 9, 38, 412\}$, $\{0, 3, 26, 53\}$, $\{0, 11, 153, 248\}$, $\{0, 7, 93, 118\}$,\adfsplit
 $\{0, 14, 203, 373\}$, $\{0, 10, 228, 288\}$, $\{0, 20, 328, 418\}$, $\{0, 21, 63, 173\}$,\adfsplit
 $\{0, 24, 138, 293\}$, $\{0, 39, 273, 393\}$, $\{0, 29, 378, 423\}$, $\{1, 2, 18, 213\}$,\adfsplit
 $\{0, 47, 128, 193\}$, $\{1, 4, 103, 253\}$, $\{1, 7, 113, 398\}$, $\{1, 11, 308, 383\}$,\adfsplit
 $\{0, 28, 41, 384\}$, $\{1, 9, 148, 239\}$, $\{0, 27, 69, 343\}$, $\{1, 12, 23, 27\}$,\adfsplit
 $\{0, 108, 109, 302\}$, $\{1, 17, 157, 378\}$, $\{0, 89, 102, 413\}$, $\{1, 24, 64, 328\}$,\adfsplit
 $\{1, 14, 48, 174\}$, $\{1, 21, 99, 408\}$, $\{1, 16, 178, 184\}$, $\{1, 26, 189, 363\}$,\adfsplit
 $\{1, 37, 49, 198\}$, $\{1, 58, 177, 194\}$, $\{1, 31, 134, 303\}$, $\{1, 32, 318, 332\}$,\adfsplit
 $\{1, 47, 283, 369\}$, $\{1, 59, 67, 168\}$, $\{2, 12, 43, 252\}$, $\{1, 222, 248, 402\}$,\adfsplit
 $\{1, 122, 258, 277\}$, $\{2, 29, 73, 147\}$, $\{2, 37, 283, 319\}$, $\{2, 293, 349, 404\}$,\adfsplit
 $\{2, 84, 264, 343\}$, $\{1, 39, 84, 158\}$, $\{1, 188, 329, 362\}$, $\{2, 129, 383, 409\}$,\adfsplit
 $\{0, 40, 169, 379\}$, $\{0, 37, 334, 364\}$, $\{0, 44, 245, 304\}$, $\{0, 49, 51, 284\}$,\adfsplit
 $\{0, 36, 349, 414\}$, $\{0, 104, 116, 287\}$, $\{0, 229, 234, 369\}$, $\{1, 71, 224, 354\}$,\adfsplit
 $\{1, 129, 139, 357\}$, $\{0, 80, 194, 277\}$, $\{0, 281, 337, 374\}$, $\{0, 31, 107, 329\}$,\adfsplit
 $\{0, 144, 204, 297\}$, $\{2, 4, 7, 342\}$, $\{1, 34, 182, 287\}$, $\{0, 222, 247, 294\}$,\adfsplit
 $\{0, 181, 389, 406\}$, $\{1, 36, 121, 394\}$, $\{1, 56, 214, 267\}$, $\{0, 32, 99, 232\}$,\adfsplit
 $\{0, 157, 202, 301\}$, $\{0, 72, 187, 266\}$, $\{0, 105, 242, 317\}$, $\{0, 62, 117, 280\}$,\adfsplit
 $\{0, 96, 177, 251\}$, $\{0, 61, 115, 367\}$, $\{0, 70, 176, 382\}$, $\{1, 51, 112, 311\}$,\adfsplit
 $\{0, 97, 160, 391\}$, $\{0, 172, 221, 261\}$, $\{0, 25, 241, 332\}$, $\{1, 81, 181, 276\}$,\adfsplit
 $\{0, 55, 205, 351\}$, $\{0, 86, 191, 325\}$, $\{0, 311, 376, 421\}$, $\{0, 131, 256, 285\}$,\adfsplit
 $\{0, 30, 90, 260\}$, $\{0, 46, 125, 411\}$, $\{0, 50, 171, 290\}$
%adfPENTend
}

\noindent under the action of the mapping $x \mapsto x + 5 \adfmod{425}$.
The deficiency graph is connected and has girth 6.
% The opposite lines are generated by the first 5 blocks: True
% End of PENT(4, 140)
%%%%%%%%%%%%%%%%%%%%%%%%%%%%%%%%%%%%%%%%%%%%%%%%%%%%%%%%%%%%%%%%%%%%%%%%%%%%%%%%%%%%%%%%%%

{\noindent\boldmath $\adfPENT(4, 141)$}~ With point set $Z_{428}$, the 15087 lines are generated from

{\scriptsize
%adfPENTstart:1:4:141:4:4:
 $\{12, 97, 123, 416\}$, $\{117, 313, 318, 332\}$, $\{7, 22, 113, 410\}$, $\{215, 219, 308, 426\}$,\adfsplit
 $\{352, 408, 55, 222\}$, $\{42, 36, 149, 285\}$, $\{318, 399, 263, 316\}$, $\{328, 290, 170, 136\}$,\adfsplit
 $\{244, 345, 314, 315\}$, $\{383, 325, 318, 299\}$, $\{417, 50, 355, 130\}$, $\{158, 388, 63, 285\}$,\adfsplit
 $\{147, 92, 356, 353\}$, $\{246, 367, 318, 220\}$, $\{318, 253, 247, 44\}$, $\{185, 56, 222, 390\}$,\adfsplit
 $\{250, 94, 423, 394\}$, $\{66, 130, 216, 13\}$, $\{90, 410, 191, 216\}$, $\{102, 52, 160, 239\}$,\adfsplit
 $\{162, 168, 121, 217\}$, $\{51, 410, 278, 253\}$, $\{303, 380, 233, 80\}$, $\{258, 152, 13, 27\}$,\adfsplit
 $\{14, 3, 204, 264\}$, $\{118, 231, 314, 189\}$, $\{406, 314, 130, 419\}$, $\{249, 4, 400, 258\}$,\adfsplit
 $\{133, 146, 427, 335\}$, $\{200, 32, 159, 240\}$, $\{323, 110, 267, 424\}$, $\{238, 328, 24, 100\}$,\adfsplit
 $\{374, 418, 228, 165\}$, $\{348, 8, 7, 70\}$, $\{30, 406, 327, 376\}$, $\{346, 370, 235, 406\}$,\adfsplit
 $\{403, 223, 413, 117\}$, $\{294, 7, 198, 160\}$, $\{78, 97, 363, 180\}$, $\{143, 329, 140, 261\}$,\adfsplit
 $\{336, 303, 59, 361\}$, $\{264, 42, 312, 180\}$, $\{111, 248, 373, 49\}$, $\{154, 137, 85, 331\}$,\adfsplit
 $\{98, 25, 173, 183\}$, $\{337, 328, 52, 209\}$, $\{356, 391, 347, 327\}$, $\{57, 172, 238, 28\}$,\adfsplit
 $\{212, 290, 50, 394\}$, $\{1, 138, 305, 110\}$, $\{122, 358, 145, 161\}$, $\{69, 202, 413, 23\}$,\adfsplit
 $\{381, 255, 330, 25\}$, $\{135, 88, 322, 368\}$, $\{104, 387, 410, 371\}$, $\{204, 39, 147, 9\}$,\adfsplit
 $\{390, 108, 380, 263\}$, $\{61, 316, 68, 207\}$, $\{125, 3, 267, 383\}$, $\{227, 402, 19, 159\}$,\adfsplit
 $\{223, 377, 205, 423\}$, $\{0, 1, 86, 94\}$, $\{0, 4, 18, 226\}$, $\{0, 5, 54, 318\}$,\adfsplit
 $\{0, 7, 110, 402\}$, $\{0, 8, 130, 330\}$, $\{0, 11, 58, 346\}$, $\{0, 13, 42, 46\}$,\adfsplit
 $\{0, 15, 294, 406\}$, $\{0, 16, 366, 378\}$, $\{0, 17, 74, 162\}$, $\{0, 19, 170, 386\}$,\adfsplit
 $\{0, 21, 246, 314\}$, $\{0, 23, 186, 410\}$, $\{0, 22, 33, 418\}$, $\{0, 28, 202, 258\}$,\adfsplit
 $\{1, 2, 5, 102\}$, $\{0, 27, 250, 398\}$, $\{0, 45, 126, 374\}$, $\{0, 39, 90, 358\}$,\adfsplit
 $\{0, 41, 262, 310\}$, $\{1, 3, 162, 286\}$, $\{0, 99, 102, 218\}$, $\{1, 7, 114, 190\}$,\adfsplit
 $\{1, 9, 178, 350\}$, $\{1, 21, 366, 382\}$, $\{0, 142, 159, 394\}$, $\{1, 13, 90, 163\}$,\adfsplit
 $\{1, 22, 285, 359\}$, $\{0, 51, 89, 354\}$, $\{1, 23, 45, 282\}$, $\{0, 93, 253, 426\}$,\adfsplit
 $\{1, 33, 155, 402\}$, $\{1, 25, 154, 389\}$, $\{1, 35, 37, 310\}$, $\{1, 46, 55, 387\}$,\adfsplit
 $\{1, 79, 106, 183\}$, $\{1, 29, 150, 289\}$, $\{1, 51, 250, 265\}$, $\{1, 103, 191, 258\}$,\adfsplit
 $\{1, 99, 153, 394\}$, $\{1, 171, 230, 327\}$, $\{1, 49, 326, 347\}$, $\{2, 55, 91, 167\}$,\adfsplit
 $\{2, 119, 131, 399\}$, $\{2, 163, 395, 423\}$, $\{2, 235, 315, 387\}$, $\{0, 20, 191, 231\}$,\adfsplit
 $\{0, 31, 287, 392\}$, $\{0, 37, 119, 307\}$, $\{0, 43, 61, 343\}$, $\{0, 52, 299, 331\}$,\adfsplit
 $\{0, 57, 175, 411\}$, $\{0, 65, 179, 239\}$, $\{0, 44, 107, 359\}$, $\{0, 53, 251, 383\}$,\adfsplit
 $\{0, 73, 407, 415\}$, $\{0, 72, 163, 215\}$, $\{0, 95, 161, 423\}$, $\{1, 67, 181, 291\}$,\adfsplit
 $\{0, 243, 367, 417\}$, $\{1, 91, 211, 235\}$, $\{0, 59, 169, 353\}$, $\{0, 96, 201, 295\}$,\adfsplit
 $\{1, 101, 221, 351\}$, $\{0, 69, 255, 292\}$, $\{0, 177, 269, 355\}$, $\{0, 115, 265, 308\}$,\adfsplit
 $\{0, 207, 241, 297\}$, $\{0, 104, 363, 377\}$, $\{0, 100, 317, 405\}$, $\{0, 141, 221, 329\}$,\adfsplit
 $\{0, 64, 321, 397\}$, $\{0, 77, 212, 393\}$, $\{0, 133, 224, 361\}$, $\{0, 81, 237, 316\}$,\adfsplit
 $\{0, 80, 172, 312\}$, $\{0, 117, 176, 244\}$, $\{0, 149, 188, 373\}$, $\{0, 165, 341, 401\}$,\adfsplit
 $\{0, 145, 160, 357\}$
%adfPENTend
}

\noindent under the action of the mapping $x \mapsto x + 4 \adfmod{428}$.
The deficiency graph is connected and has girth 8.
% The opposite lines are generated by the first 4 blocks: True
% End of PENT(4, 141)
%%%%%%%%%%%%%%%%%%%%%%%%%%%%%%%%%%%%%%%%%%%%%%%%%%%%%%%%%%%%%%%%%%%%%%%%%%%%%%%%%%%%%%%%%%

{\noindent\boldmath $\adfPENT(4, 149)$}~ With point set $Z_{452}$, the 16837 lines are generated from

{\scriptsize
%adfPENTstart:1:4:149:4:4:
 $\{44, 345, 355, 408\}$, $\{3, 108, 358, 387\}$, $\{50, 97, 406, 419\}$, $\{1, 38, 69, 100\}$,\adfsplit
 $\{28, 395, 309, 144\}$, $\{24, 332, 195, 81\}$, $\{223, 411, 374, 131\}$, $\{358, 263, 102, 154\}$,\adfsplit
 $\{293, 73, 366, 48\}$, $\{196, 326, 136, 342\}$, $\{400, 214, 79, 345\}$, $\{400, 418, 139, 318\}$,\adfsplit
 $\{357, 242, 257, 144\}$, $\{135, 373, 113, 46\}$, $\{445, 115, 363, 224\}$, $\{312, 59, 225, 317\}$,\adfsplit
 $\{276, 256, 390, 448\}$, $\{176, 59, 160, 378\}$, $\{290, 245, 409, 191\}$, $\{438, 210, 318, 425\}$,\adfsplit
 $\{271, 27, 440, 312\}$, $\{150, 199, 327, 171\}$, $\{52, 87, 326, 309\}$, $\{21, 114, 413, 311\}$,\adfsplit
 $\{19, 68, 82, 155\}$, $\{77, 171, 64, 401\}$, $\{371, 422, 328, 23\}$, $\{260, 232, 384, 42\}$,\adfsplit
 $\{293, 259, 317, 124\}$, $\{288, 47, 425, 380\}$, $\{345, 349, 274, 118\}$, $\{411, 309, 29, 3\}$,\adfsplit
 $\{5, 46, 34, 126\}$, $\{115, 193, 334, 17\}$, $\{113, 258, 405, 404\}$, $\{109, 168, 13, 20\}$,\adfsplit
 $\{281, 334, 300, 327\}$, $\{198, 96, 348, 79\}$, $\{393, 337, 195, 264\}$, $\{103, 84, 423, 155\}$,\adfsplit
 $\{312, 262, 226, 392\}$, $\{109, 427, 8, 184\}$, $\{285, 49, 262, 32\}$, $\{388, 102, 307, 26\}$,\adfsplit
 $\{319, 419, 243, 193\}$, $\{46, 114, 53, 89\}$, $\{149, 347, 3, 13\}$, $\{7, 86, 212, 375\}$,\adfsplit
 $\{69, 40, 435, 173\}$, $\{95, 369, 220, 320\}$, $\{44, 14, 127, 339\}$, $\{87, 216, 268, 83\}$,\adfsplit
 $\{156, 153, 215, 36\}$, $\{37, 436, 187, 243\}$, $\{277, 303, 0, 231\}$, $\{391, 159, 337, 141\}$,\adfsplit
 $\{386, 83, 279, 154\}$, $\{2, 40, 414, 365\}$, $\{19, 414, 376, 122\}$, $\{445, 233, 376, 272\}$,\adfsplit
 $\{19, 447, 289, 354\}$, $\{408, 144, 30, 109\}$, $\{324, 396, 186, 12\}$, $\{245, 251, 169, 213\}$,\adfsplit
 $\{37, 438, 109, 70\}$, $\{314, 395, 401, 160\}$, $\{387, 105, 425, 362\}$, $\{263, 104, 62, 145\}$,\adfsplit
 $\{421, 269, 316, 21\}$, $\{369, 82, 61, 392\}$, $\{364, 359, 18, 281\}$, $\{49, 353, 128, 58\}$,\adfsplit
 $\{77, 254, 409, 315\}$, $\{354, 242, 269, 84\}$, $\{20, 333, 96, 207\}$, $\{443, 427, 134, 188\}$,\adfsplit
 $\{278, 377, 79, 38\}$, $\{234, 404, 67, 3\}$, $\{153, 241, 284, 42\}$, $\{0, 2, 21, 357\}$,\adfsplit
 $\{0, 3, 189, 197\}$, $\{0, 4, 37, 289\}$, $\{0, 6, 177, 401\}$, $\{0, 7, 81, 349\}$,\adfsplit
 $\{0, 9, 23, 441\}$, $\{0, 10, 65, 361\}$, $\{0, 11, 125, 293\}$, $\{0, 15, 265, 329\}$,\adfsplit
 $\{0, 26, 217, 233\}$, $\{0, 12, 109, 437\}$, $\{0, 24, 229, 341\}$, $\{0, 30, 209, 249\}$,\adfsplit
 $\{0, 50, 305, 333\}$, $\{0, 58, 61, 269\}$, $\{0, 63, 273, 381\}$, $\{0, 78, 181, 261\}$,\adfsplit
 $\{1, 2, 189, 194\}$, $\{1, 31, 273, 315\}$, $\{1, 59, 78, 369\}$, $\{1, 13, 82, 186\}$,\adfsplit
 $\{1, 91, 106, 141\}$, $\{0, 31, 93, 141\}$, $\{0, 66, 99, 385\}$, $\{0, 85, 151, 275\}$,\adfsplit
 $\{0, 77, 186, 423\}$, $\{0, 103, 145, 443\}$, $\{1, 71, 107, 118\}$, $\{1, 131, 218, 219\}$,\adfsplit
 $\{1, 138, 202, 279\}$, $\{1, 150, 210, 331\}$, $\{1, 182, 326, 431\}$, $\{1, 130, 278, 302\}$,\adfsplit
 $\{1, 175, 271, 314\}$, $\{1, 206, 250, 435\}$, $\{1, 163, 303, 451\}$, $\{1, 230, 286, 423\}$,\adfsplit
 $\{1, 355, 363, 394\}$, $\{1, 290, 295, 362\}$, $\{1, 247, 439, 442\}$, $\{0, 8, 299, 415\}$,\adfsplit
 $\{0, 46, 215, 439\}$, $\{0, 47, 79, 127\}$, $\{0, 42, 195, 207\}$, $\{0, 67, 219, 296\}$,\adfsplit
 $\{0, 54, 123, 431\}$, $\{0, 122, 387, 427\}$, $\{2, 47, 67, 339\}$, $\{0, 143, 263, 438\}$,\adfsplit
 $\{2, 131, 191, 314\}$, $\{0, 70, 155, 391\}$, $\{0, 75, 168, 426\}$, $\{2, 22, 247, 338\}$,\adfsplit
 $\{0, 138, 204, 379\}$, $\{0, 91, 180, 212\}$, $\{2, 11, 246, 263\}$, $\{0, 55, 56, 220\}$,\adfsplit
 $\{2, 134, 286, 431\}$, $\{0, 307, 334, 362\}$, $\{0, 22, 196, 236\}$, $\{0, 226, 350, 354\}$,\adfsplit
 $\{0, 194, 358, 446\}$, $\{0, 118, 294, 316\}$, $\{0, 170, 434, 442\}$, $\{0, 82, 244, 356\}$,\adfsplit
 $\{0, 48, 132, 346\}$, $\{0, 36, 162, 378\}$, $\{0, 150, 182, 184\}$, $\{0, 64, 224, 310\}$,\adfsplit
 $\{0, 62, 108, 330\}$
%adfPENTend
}

\noindent under the action of the mapping $x \mapsto x + 4 \adfmod{452}$.
The deficiency graph is connected and has girth 7.
% The opposite lines are generated by the first 4 blocks: True
% End of PENT(4, 149)
%%%%%%%%%%%%%%%%%%%%%%%%%%%%%%%%%%%%%%%%%%%%%%%%%%%%%%%%%%%%%%%%%%%%%%%%%%%%%%%%%%%%%%%%%%

{\noindent\boldmath $\adfPENT(4, 157)$}~ With point set $Z_{476}$, the 18683 lines are generated from

{\scriptsize
%adfPENTstart:1:4:157:4:4:
 $\{162, 291, 421, 459\}$, $\{13, 56, 375, 465\}$, $\{50, 316, 323, 430\}$, $\{20, 105, 158, 188\}$,\adfsplit
 $\{466, 103, 25, 340\}$, $\{234, 316, 22, 382\}$, $\{260, 336, 131, 68\}$, $\{51, 289, 378, 270\}$,\adfsplit
 $\{411, 371, 89, 444\}$, $\{307, 81, 82, 45\}$, $\{438, 108, 388, 165\}$, $\{64, 74, 198, 459\}$,\adfsplit
 $\{328, 139, 370, 24\}$, $\{430, 65, 27, 145\}$, $\{84, 393, 93, 379\}$, $\{204, 343, 282, 54\}$,\adfsplit
 $\{64, 436, 327, 48\}$, $\{278, 155, 176, 132\}$, $\{91, 211, 4, 227\}$, $\{7, 9, 193, 264\}$,\adfsplit
 $\{163, 36, 470, 343\}$, $\{144, 193, 46, 297\}$, $\{383, 230, 201, 251\}$, $\{195, 204, 416, 70\}$,\adfsplit
 $\{343, 121, 392, 367\}$, $\{415, 180, 141, 450\}$, $\{347, 83, 403, 426\}$, $\{440, 37, 354, 140\}$,\adfsplit
 $\{331, 335, 306, 265\}$, $\{132, 78, 405, 383\}$, $\{290, 25, 241, 131\}$, $\{169, 215, 140, 321\}$,\adfsplit
 $\{49, 441, 79, 434\}$, $\{142, 153, 84, 120\}$, $\{302, 214, 276, 365\}$, $\{207, 34, 371, 205\}$,\adfsplit
 $\{203, 151, 108, 98\}$, $\{285, 401, 137, 79\}$, $\{415, 430, 77, 257\}$, $\{311, 382, 272, 181\}$,\adfsplit
 $\{157, 340, 139, 327\}$, $\{350, 91, 372, 396\}$, $\{200, 303, 195, 438\}$, $\{275, 436, 341, 118\}$,\adfsplit
 $\{327, 182, 78, 148\}$, $\{225, 163, 98, 3\}$, $\{399, 442, 152, 462\}$, $\{417, 134, 331, 444\}$,\adfsplit
 $\{46, 121, 352, 320\}$, $\{87, 61, 126, 118\}$, $\{25, 273, 394, 15\}$, $\{54, 55, 283, 56\}$,\adfsplit
 $\{455, 253, 26, 304\}$, $\{33, 25, 240, 144\}$, $\{253, 416, 272, 124\}$, $\{447, 258, 303, 453\}$,\adfsplit
 $\{147, 125, 243, 66\}$, $\{131, 236, 347, 255\}$, $\{130, 267, 105, 119\}$, $\{79, 232, 315, 466\}$,\adfsplit
 $\{193, 198, 154, 272\}$, $\{123, 88, 76, 436\}$, $\{38, 346, 452, 402\}$, $\{42, 100, 82, 236\}$,\adfsplit
 $\{451, 313, 109, 69\}$, $\{162, 44, 156, 320\}$, $\{133, 465, 378, 383\}$, $\{250, 28, 76, 330\}$,\adfsplit
 $\{37, 150, 95, 41\}$, $\{109, 35, 226, 151\}$, $\{156, 250, 388, 109\}$, $\{108, 190, 241, 100\}$,\adfsplit
 $\{412, 88, 167, 265\}$, $\{332, 74, 134, 276\}$, $\{87, 146, 311, 260\}$, $\{34, 213, 166, 350\}$,\adfsplit
 $\{380, 360, 69, 200\}$, $\{339, 29, 12, 295\}$, $\{144, 185, 45, 359\}$, $\{152, 450, 463, 243\}$,\adfsplit
 $\{202, 322, 439, 412\}$, $\{421, 101, 98, 104\}$, $\{0, 1, 187, 259\}$, $\{0, 2, 59, 143\}$,\adfsplit
 $\{0, 3, 15, 415\}$, $\{0, 4, 203, 343\}$, $\{0, 5, 67, 439\}$, $\{0, 11, 359, 407\}$,\adfsplit
 $\{0, 13, 31, 419\}$, $\{0, 14, 107, 391\}$, $\{0, 18, 119, 411\}$, $\{0, 21, 123, 447\}$,\adfsplit
 $\{0, 30, 211, 387\}$, $\{0, 25, 99, 379\}$, $\{0, 37, 335, 355\}$, $\{0, 45, 55, 167\}$,\adfsplit
 $\{0, 65, 299, 399\}$, $\{0, 93, 175, 243\}$, $\{0, 52, 183, 383\}$, $\{1, 7, 35, 61\}$,\adfsplit
 $\{0, 122, 155, 423\}$, $\{1, 10, 87, 95\}$, $\{0, 142, 191, 435\}$, $\{1, 22, 243, 303\}$,\adfsplit
 $\{1, 17, 195, 231\}$, $\{1, 70, 399, 431\}$, $\{0, 71, 98, 375\}$, $\{0, 38, 106, 147\}$,\adfsplit
 $\{0, 159, 178, 431\}$, $\{0, 61, 267, 392\}$, $\{0, 135, 169, 186\}$, $\{1, 21, 190, 351\}$,\adfsplit
 $\{1, 14, 57, 199\}$, $\{1, 45, 135, 401\}$, $\{1, 29, 127, 290\}$, $\{1, 82, 86, 295\}$,\adfsplit
 $\{1, 106, 307, 314\}$, $\{1, 78, 147, 234\}$, $\{1, 94, 111, 222\}$, $\{1, 182, 206, 447\}$,\adfsplit
 $\{1, 242, 270, 279\}$, $\{1, 191, 258, 462\}$, $\{2, 14, 154, 347\}$, $\{2, 66, 375, 378\}$,\adfsplit
 $\{0, 53, 154, 230\}$, $\{0, 54, 77, 438\}$, $\{0, 70, 86, 306\}$, $\{0, 46, 113, 350\}$,\adfsplit
 $\{0, 68, 314, 450\}$, $\{0, 74, 101, 262\}$, $\{0, 137, 294, 366\}$, $\{1, 46, 98, 378\}$,\adfsplit
 $\{1, 134, 358, 394\}$, $\{0, 194, 249, 374\}$, $\{1, 130, 322, 406\}$, $\{1, 49, 214, 446\}$,\adfsplit
 $\{0, 226, 258, 401\}$, $\{1, 34, 210, 313\}$, $\{0, 209, 281, 354\}$, $\{0, 189, 250, 389\}$,\adfsplit
 $\{0, 150, 285, 349\}$, $\{0, 64, 220, 462\}$, $\{0, 97, 252, 442\}$, $\{0, 80, 286, 352\}$,\adfsplit
 $\{1, 89, 197, 382\}$, $\{0, 60, 297, 404\}$, $\{0, 193, 289, 461\}$, $\{0, 28, 216, 445\}$,\adfsplit
 $\{0, 157, 257, 325\}$, $\{0, 109, 301, 333\}$, $\{0, 81, 92, 241\}$, $\{0, 225, 361, 413\}$,\adfsplit
 $\{0, 108, 233, 356\}$, $\{0, 40, 213, 305\}$, $\{0, 117, 337, 469\}$, $\{0, 105, 217, 453\}$,\adfsplit
 $\{0, 100, 240, 441\}$
%adfPENTend
}

\noindent under the action of the mapping $x \mapsto x + 4 \adfmod{476}$.
The deficiency graph is connected and has girth 7.
% The opposite lines are generated by the first 4 blocks: True
% End of PENT(4, 157)
%%%%%%%%%%%%%%%%%%%%%%%%%%%%%%%%%%%%%%%%%%%%%%%%%%%%%%%%%%%%%%%%%%%%%%%%%%%%%%%%%%%%%%%%%%

{\noindent\boldmath $\adfPENT(4, 160)$}~ With point set $Z_{485}$, the 19400 lines are generated from

{\scriptsize
%adfPENTstart:1:4:160:5:5:
 $\{86, 153, 379, 427\}$, $\{23, 76, 400, 411\}$, $\{60, 117, 372, 388\}$, $\{19, 102, 335, 466\}$,\adfsplit
 $\{24, 110, 469, 473\}$, $\{417, 282, 94, 239\}$, $\{381, 459, 384, 254\}$, $\{439, 334, 134, 313\}$,\adfsplit
 $\{77, 106, 24, 422\}$, $\{301, 337, 33, 263\}$, $\{126, 349, 75, 323\}$, $\{333, 183, 462, 360\}$,\adfsplit
 $\{108, 390, 481, 475\}$, $\{141, 142, 63, 330\}$, $\{150, 413, 417, 32\}$, $\{394, 69, 55, 288\}$,\adfsplit
 $\{425, 178, 98, 196\}$, $\{82, 227, 137, 294\}$, $\{7, 186, 354, 260\}$, $\{68, 317, 159, 383\}$,\adfsplit
 $\{302, 113, 306, 361\}$, $\{255, 163, 462, 361\}$, $\{90, 125, 282, 271\}$, $\{377, 336, 91, 297\}$,\adfsplit
 $\{77, 213, 366, 62\}$, $\{466, 300, 328, 83\}$, $\{334, 0, 242, 83\}$, $\{392, 259, 174, 297\}$,\adfsplit
 $\{204, 277, 206, 482\}$, $\{18, 197, 462, 438\}$, $\{137, 303, 459, 216\}$, $\{328, 235, 395, 134\}$,\adfsplit
 $\{263, 269, 339, 384\}$, $\{438, 260, 85, 107\}$, $\{235, 137, 298, 234\}$, $\{113, 88, 449, 3\}$,\adfsplit
 $\{113, 442, 157, 475\}$, $\{39, 97, 141, 267\}$, $\{377, 356, 143, 12\}$, $\{62, 30, 173, 70\}$,\adfsplit
 $\{293, 156, 363, 312\}$, $\{412, 129, 415, 303\}$, $\{414, 60, 7, 249\}$, $\{260, 212, 297, 258\}$,\adfsplit
 $\{301, 334, 370, 168\}$, $\{382, 113, 202, 384\}$, $\{147, 340, 101, 223\}$, $\{372, 403, 55, 151\}$,\adfsplit
 $\{122, 162, 281, 285\}$, $\{463, 127, 385, 162\}$, $\{179, 463, 272, 41\}$, $\{1, 318, 170, 406\}$,\adfsplit
 $\{322, 372, 346, 204\}$, $\{259, 31, 122, 246\}$, $\{253, 184, 174, 203\}$, $\{443, 385, 370, 173\}$,\adfsplit
 $\{22, 419, 271, 363\}$, $\{85, 476, 435, 3\}$, $\{231, 175, 425, 91\}$, $\{188, 126, 196, 144\}$,\adfsplit
 $\{68, 455, 299, 199\}$, $\{301, 57, 106, 229\}$, $\{396, 462, 481, 78\}$, $\{397, 379, 226, 337\}$,\adfsplit
 $\{123, 35, 474, 365\}$, $\{212, 395, 192, 407\}$, $\{26, 160, 6, 355\}$, $\{314, 307, 409, 335\}$,\adfsplit
 $\{81, 232, 417, 205\}$, $\{156, 285, 109, 356\}$, $\{14, 301, 328, 462\}$, $\{220, 479, 169, 282\}$,\adfsplit
 $\{133, 353, 228, 124\}$, $\{201, 25, 361, 94\}$, $\{16, 326, 174, 456\}$, $\{116, 292, 424, 55\}$,\adfsplit
 $\{151, 472, 356, 364\}$, $\{83, 198, 144, 63\}$, $\{149, 290, 419, 372\}$, $\{57, 435, 406, 95\}$,\adfsplit
 $\{266, 153, 410, 443\}$, $\{89, 7, 447, 88\}$, $\{251, 21, 140, 429\}$, $\{201, 83, 238, 38\}$,\adfsplit
 $\{404, 375, 270, 223\}$, $\{338, 116, 4, 193\}$, $\{423, 65, 276, 475\}$, $\{286, 11, 83, 469\}$,\adfsplit
 $\{243, 63, 316, 159\}$, $\{177, 391, 15, 107\}$, $\{139, 242, 347, 372\}$, $\{404, 342, 471, 170\}$,\adfsplit
 $\{109, 45, 271, 32\}$, $\{302, 435, 150, 297\}$, $\{124, 297, 393, 369\}$, $\{28, 58, 399, 249\}$,\adfsplit
 $\{168, 326, 422, 358\}$, $\{392, 384, 242, 188\}$, $\{159, 298, 398, 100\}$, $\{167, 308, 313, 181\}$,\adfsplit
 $\{21, 126, 271, 184\}$, $\{379, 348, 293, 125\}$, $\{256, 28, 356, 136\}$, $\{51, 116, 479, 447\}$,\adfsplit
 $\{329, 371, 228, 180\}$, $\{0, 1, 7, 417\}$, $\{0, 2, 20, 23\}$, $\{0, 4, 173, 188\}$,\adfsplit
 $\{0, 5, 52, 343\}$, $\{0, 8, 128, 453\}$, $\{0, 9, 43, 72\}$, $\{0, 10, 187, 408\}$,\adfsplit
 $\{0, 13, 16, 293\}$, $\{0, 21, 68, 423\}$, $\{0, 17, 278, 383\}$, $\{0, 18, 25, 413\}$,\adfsplit
 $\{0, 26, 213, 327\}$, $\{0, 24, 113, 122\}$, $\{0, 31, 138, 422\}$, $\{0, 34, 217, 473\}$,\adfsplit
 $\{0, 84, 193, 428\}$, $\{0, 38, 132, 448\}$, $\{0, 54, 303, 468\}$, $\{0, 97, 183, 443\}$,\adfsplit
 $\{0, 36, 182, 208\}$, $\{0, 44, 323, 463\}$, $\{0, 79, 112, 378\}$, $\{1, 3, 303, 313\}$,\adfsplit
 $\{0, 76, 258, 318\}$, $\{0, 42, 163, 197\}$, $\{0, 154, 222, 313\}$, $\{0, 77, 133, 202\}$,\adfsplit
 $\{1, 6, 58, 268\}$, $\{0, 89, 253, 337\}$, $\{1, 11, 213, 412\}$, $\{1, 8, 31, 458\}$,\adfsplit
 $\{1, 44, 57, 163\}$, $\{1, 24, 262, 298\}$, $\{1, 17, 257, 453\}$, $\{1, 18, 192, 461\}$,\adfsplit
 $\{1, 32, 308, 436\}$, $\{1, 29, 227, 473\}$, $\{1, 77, 89, 438\}$, $\{1, 36, 167, 228\}$,\adfsplit
 $\{1, 94, 338, 352\}$, $\{1, 96, 248, 394\}$, $\{1, 69, 84, 393\}$, $\{1, 267, 289, 383\}$,\adfsplit
 $\{1, 74, 301, 443\}$, $\{1, 209, 264, 418\}$, $\{1, 187, 363, 422\}$, $\{2, 29, 154, 228\}$,\adfsplit
 $\{2, 8, 59, 422\}$, $\{2, 114, 173, 284\}$, $\{2, 3, 167, 477\}$, $\{3, 284, 439, 474\}$,\adfsplit
 $\{1, 118, 129, 474\}$, $\{0, 19, 287, 397\}$, $\{0, 66, 117, 377\}$, $\{0, 65, 342, 372\}$,\adfsplit
 $\{0, 104, 127, 402\}$, $\{0, 67, 170, 257\}$, $\{0, 81, 142, 251\}$, $\{0, 30, 231, 442\}$,\adfsplit
 $\{0, 99, 247, 407\}$, $\{0, 70, 171, 462\}$, $\{1, 104, 107, 339\}$, $\{1, 119, 254, 382\}$,\adfsplit
 $\{1, 54, 261, 347\}$, $\{0, 272, 461, 476\}$, $\{1, 91, 284, 477\}$, $\{1, 134, 424, 452\}$,\adfsplit
 $\{0, 137, 284, 469\}$, $\{1, 109, 252, 369\}$, $\{0, 144, 357, 409\}$, $\{1, 61, 432, 449\}$,\adfsplit
 $\{0, 45, 261, 371\}$, $\{0, 46, 90, 396\}$, $\{0, 80, 276, 466\}$, $\{1, 41, 354, 361\}$,\adfsplit
 $\{0, 139, 321, 436\}$, $\{0, 116, 150, 421\}$, $\{0, 126, 281, 394\}$, $\{0, 194, 219, 431\}$,\adfsplit
 $\{0, 244, 364, 381\}$, $\{0, 60, 266, 280\}$, $\{0, 374, 406, 454\}$, $\{0, 125, 346, 444\}$,\adfsplit
 $\{0, 121, 264, 294\}$, $\{0, 141, 180, 459\}$, $\{0, 55, 241, 304\}$, $\{0, 174, 404, 426\}$,\adfsplit
 $\{0, 120, 299, 389\}$, $\{0, 164, 214, 225\}$, $\{0, 39, 210, 414\}$, $\{0, 119, 124, 184\}$,\adfsplit
 $\{0, 95, 209, 419\}$, $\{0, 100, 230, 345\}$, $\{0, 159, 165, 215\}$, $\{0, 49, 185, 295\}$
%adfPENTend
}

\noindent under the action of the mapping $x \mapsto x + 5 \adfmod{485}$.
The deficiency graph is connected and has girth 7.
% The opposite lines are generated by the first 5 blocks: True
% End of PENT(4, 160)
%%%%%%%%%%%%%%%%%%%%%%%%%%%%%%%%%%%%%%%%%%%%%%%%%%%%%%%%%%%%%%%%%%%%%%%%%%%%%%%%%%%%%%%%%%

{\noindent\boldmath $\adfPENT(4, 165)$}~ With point set $Z_{500}$, the 20625 lines are generated from

{\scriptsize
%adfPENTstart:1:4:165:4:4:
 $\{281, 389, 406, 451\}$, $\{112, 190, 220, 422\}$, $\{81, 96, 313, 351\}$, $\{52, 143, 154, 363\}$,\adfsplit
 $\{67, 68, 319, 451\}$, $\{476, 99, 234, 397\}$, $\{377, 489, 290, 108\}$, $\{309, 268, 304, 243\}$,\adfsplit
 $\{405, 7, 269, 218\}$, $\{83, 82, 61, 106\}$, $\{333, 342, 33, 90\}$, $\{86, 240, 43, 214\}$,\adfsplit
 $\{407, 388, 451, 216\}$, $\{285, 111, 129, 178\}$, $\{19, 316, 123, 53\}$, $\{475, 390, 404, 276\}$,\adfsplit
 $\{368, 406, 364, 38\}$, $\{259, 297, 471, 403\}$, $\{179, 457, 98, 339\}$, $\{342, 170, 491, 430\}$,\adfsplit
 $\{26, 482, 463, 129\}$, $\{113, 271, 448, 310\}$, $\{311, 347, 38, 473\}$, $\{25, 397, 453, 314\}$,\adfsplit
 $\{278, 431, 36, 459\}$, $\{345, 265, 4, 382\}$, $\{102, 402, 133, 193\}$, $\{169, 300, 317, 246\}$,\adfsplit
 $\{384, 458, 127, 312\}$, $\{39, 411, 400, 414\}$, $\{322, 139, 432, 42\}$, $\{71, 238, 80, 156\}$,\adfsplit
 $\{60, 122, 339, 29\}$, $\{267, 445, 360, 217\}$, $\{138, 313, 254, 130\}$, $\{320, 251, 172, 253\}$,\adfsplit
 $\{17, 120, 349, 91\}$, $\{314, 165, 108, 89\}$, $\{337, 388, 435, 127\}$, $\{254, 283, 121, 163\}$,\adfsplit
 $\{407, 444, 212, 76\}$, $\{114, 116, 386, 472\}$, $\{202, 54, 104, 332\}$, $\{355, 136, 388, 420\}$,\adfsplit
 $\{311, 371, 182, 484\}$, $\{158, 446, 138, 457\}$, $\{47, 399, 155, 434\}$, $\{156, 280, 67, 182\}$,\adfsplit
 $\{482, 416, 263, 3\}$, $\{311, 27, 58, 406\}$, $\{37, 405, 488, 317\}$, $\{146, 168, 415, 17\}$,\adfsplit
 $\{405, 23, 429, 95\}$, $\{436, 22, 32, 136\}$, $\{239, 490, 426, 212\}$, $\{409, 176, 245, 342\}$,\adfsplit
 $\{415, 406, 337, 121\}$, $\{414, 118, 334, 116\}$, $\{260, 310, 487, 394\}$, $\{161, 417, 427, 410\}$,\adfsplit
 $\{197, 86, 11, 205\}$, $\{126, 174, 142, 387\}$, $\{76, 333, 291, 37\}$, $\{201, 5, 437, 442\}$,\adfsplit
 $\{170, 252, 66, 332\}$, $\{299, 24, 481, 233\}$, $\{427, 434, 159, 491\}$, $\{355, 332, 431, 339\}$,\adfsplit
 $\{70, 280, 467, 74\}$, $\{487, 38, 43, 11\}$, $\{262, 45, 41, 275\}$, $\{437, 483, 104, 120\}$,\adfsplit
 $\{287, 184, 238, 177\}$, $\{12, 392, 213, 155\}$, $\{270, 455, 113, 257\}$, $\{132, 458, 308, 321\}$,\adfsplit
 $\{491, 175, 39, 122\}$, $\{317, 190, 67, 370\}$, $\{389, 83, 208, 142\}$, $\{345, 205, 351, 184\}$,\adfsplit
 $\{457, 62, 244, 80\}$, $\{219, 154, 118, 269\}$, $\{284, 299, 213, 378\}$, $\{459, 471, 100, 434\}$,\adfsplit
 $\{190, 409, 374, 266\}$, $\{0, 1, 178, 230\}$, $\{0, 3, 18, 162\}$, $\{0, 6, 12, 366\}$,\adfsplit
 $\{0, 8, 246, 358\}$, $\{0, 9, 10, 286\}$, $\{0, 20, 126, 458\}$, $\{0, 22, 34, 210\}$,\adfsplit
 $\{0, 28, 254, 410\}$, $\{0, 24, 274, 330\}$, $\{0, 29, 70, 190\}$, $\{0, 25, 58, 194\}$,\adfsplit
 $\{0, 30, 33, 266\}$, $\{0, 31, 130, 422\}$, $\{0, 35, 90, 394\}$, $\{0, 39, 186, 282\}$,\adfsplit
 $\{0, 37, 342, 442\}$, $\{0, 43, 426, 466\}$, $\{1, 3, 206, 274\}$, $\{0, 55, 262, 322\}$,\adfsplit
 $\{1, 13, 86, 114\}$, $\{0, 46, 65, 302\}$, $\{1, 21, 174, 266\}$, $\{0, 89, 118, 454\}$,\adfsplit
 $\{1, 26, 53, 454\}$, $\{0, 51, 122, 462\}$, $\{0, 138, 171, 271\}$, $\{0, 45, 231, 374\}$,\adfsplit
 $\{0, 53, 135, 430\}$, $\{0, 67, 154, 391\}$, $\{1, 19, 146, 347\}$, $\{0, 59, 223, 398\}$,\adfsplit
 $\{1, 15, 211, 446\}$, $\{1, 29, 183, 322\}$, $\{1, 31, 85, 426\}$, $\{1, 35, 117, 462\}$,\adfsplit
 $\{1, 27, 439, 486\}$, $\{1, 37, 261, 370\}$, $\{1, 90, 293, 427\}$, $\{1, 65, 186, 487\}$,\adfsplit
 $\{1, 82, 227, 321\}$, $\{1, 143, 367, 478\}$, $\{1, 194, 387, 471\}$, $\{2, 175, 287, 383\}$,\adfsplit
 $\{1, 263, 303, 458\}$, $\{1, 214, 351, 371\}$, $\{0, 40, 335, 459\}$, $\{0, 44, 127, 447\}$,\adfsplit
 $\{0, 48, 159, 211\}$, $\{0, 61, 183, 339\}$, $\{0, 60, 155, 343\}$, $\{0, 77, 479, 487\}$,\adfsplit
 $\{0, 117, 267, 495\}$, $\{0, 93, 147, 299\}$, $\{0, 101, 319, 399\}$, $\{0, 151, 173, 351\}$,\adfsplit
 $\{0, 73, 131, 427\}$, $\{0, 119, 355, 361\}$, $\{1, 33, 391, 395\}$, $\{0, 169, 264, 443\}$,\adfsplit
 $\{1, 41, 161, 287\}$, $\{0, 177, 291, 301\}$, $\{0, 277, 367, 453\}$, $\{0, 107, 152, 393\}$,\adfsplit
 $\{0, 115, 141, 185\}$, $\{0, 129, 259, 477\}$, $\{0, 75, 121, 437\}$, $\{0, 153, 308, 483\}$,\adfsplit
 $\{0, 100, 305, 409\}$, $\{0, 197, 293, 385\}$, $\{0, 125, 337, 340\}$, $\{0, 116, 249, 265\}$,\adfsplit
 $\{0, 145, 280, 473\}$, $\{0, 97, 289, 344\}$, $\{0, 52, 325, 425\}$, $\{0, 88, 225, 312\}$,\adfsplit
 $\{0, 68, 208, 313\}$, $\{0, 109, 157, 168\}$, $\{0, 56, 240, 353\}$, $\{0, 84, 196, 288\}$,\adfsplit
 $\{0, 64, 244, 465\}$
%adfPENTend
}

\noindent under the action of the mapping $x \mapsto x + 4 \adfmod{500}$.
The deficiency graph is connected and has girth 8.
% The opposite lines are generated by the first 4 blocks: True
% End of PENT(4, 165)
%%%%%%%%%%%%%%%%%%%%%%%%%%%%%%%%%%%%%%%%%%%%%%%%%%%%%%%%%%%%%%%%%%%%%%%%%%%%%%%%%%%%%%%%%%

{\noindent\boldmath $\adfPENT(4, 173)$}~ With point set $Z_{524}$, the 22663 lines are generated from

{\scriptsize
%adfPENTstart:1:4:173:4:4:
 $\{230, 240, 284, 447\}$, $\{209, 299, 317, 355\}$, $\{50, 203, 296, 478\}$, $\{80, 173, 229, 326\}$,\adfsplit
 $\{219, 290, 85, 402\}$, $\{147, 307, 151, 497\}$, $\{422, 314, 170, 203\}$, $\{389, 117, 518, 205\}$,\adfsplit
 $\{227, 233, 242, 103\}$, $\{140, 219, 505, 142\}$, $\{361, 385, 225, 335\}$, $\{192, 86, 126, 197\}$,\adfsplit
 $\{193, 447, 89, 162\}$, $\{235, 290, 362, 364\}$, $\{89, 119, 323, 45\}$, $\{178, 53, 264, 205\}$,\adfsplit
 $\{422, 198, 53, 475\}$, $\{523, 473, 329, 204\}$, $\{435, 102, 145, 324\}$, $\{24, 316, 133, 290\}$,\adfsplit
 $\{37, 76, 340, 43\}$, $\{77, 328, 414, 65\}$, $\{113, 446, 470, 7\}$, $\{82, 448, 51, 329\}$,\adfsplit
 $\{187, 15, 36, 364\}$, $\{219, 5, 326, 373\}$, $\{419, 339, 328, 63\}$, $\{14, 382, 97, 7\}$,\adfsplit
 $\{407, 269, 90, 172\}$, $\{6, 120, 480, 493\}$, $\{14, 390, 194, 410\}$, $\{6, 45, 208, 177\}$,\adfsplit
 $\{466, 274, 106, 381\}$, $\{216, 151, 58, 115\}$, $\{134, 516, 137, 267\}$, $\{126, 58, 331, 64\}$,\adfsplit
 $\{36, 1, 138, 283\}$, $\{207, 161, 444, 49\}$, $\{257, 450, 43, 110\}$, $\{418, 55, 390, 218\}$,\adfsplit
 $\{271, 231, 109, 429\}$, $\{210, 70, 405, 112\}$, $\{304, 31, 97, 263\}$, $\{381, 321, 462, 404\}$,\adfsplit
 $\{513, 124, 165, 476\}$, $\{329, 91, 387, 172\}$, $\{318, 199, 252, 91\}$, $\{184, 465, 247, 510\}$,\adfsplit
 $\{234, 241, 289, 31\}$, $\{38, 190, 309, 54\}$, $\{373, 401, 500, 391\}$, $\{82, 142, 131, 25\}$,\adfsplit
 $\{276, 504, 10, 40\}$, $\{358, 254, 140, 357\}$, $\{81, 128, 143, 373\}$, $\{401, 418, 367, 499\}$,\adfsplit
 $\{59, 511, 280, 105\}$, $\{67, 70, 38, 154\}$, $\{382, 148, 328, 181\}$, $\{42, 216, 514, 59\}$,\adfsplit
 $\{283, 158, 366, 403\}$, $\{233, 208, 308, 159\}$, $\{402, 309, 189, 491\}$, $\{317, 121, 244, 58\}$,\adfsplit
 $\{443, 186, 60, 125\}$, $\{424, 62, 351, 306\}$, $\{495, 435, 76, 479\}$, $\{46, 97, 113, 167\}$,\adfsplit
 $\{350, 500, 57, 175\}$, $\{259, 118, 256, 443\}$, $\{210, 59, 245, 395\}$, $\{78, 479, 379, 340\}$,\adfsplit
 $\{225, 145, 446, 115\}$, $\{103, 480, 416, 458\}$, $\{369, 435, 192, 107\}$, $\{91, 518, 3, 264\}$,\adfsplit
 $\{189, 178, 102, 11\}$, $\{228, 238, 395, 437\}$, $\{132, 411, 397, 76\}$, $\{55, 308, 335, 200\}$,\adfsplit
 $\{203, 418, 88, 212\}$, $\{462, 175, 508, 463\}$, $\{485, 468, 183, 173\}$, $\{395, 422, 130, 95\}$,\adfsplit
 $\{345, 484, 91, 316\}$, $\{514, 55, 475, 480\}$, $\{131, 64, 387, 395\}$, $\{395, 393, 2, 372\}$,\adfsplit
 $\{481, 34, 215, 238\}$, $\{291, 417, 249, 446\}$, $\{49, 101, 216, 197\}$, $\{233, 282, 338, 112\}$,\adfsplit
 $\{91, 48, 204, 375\}$, $\{0, 4, 55, 183\}$, $\{0, 1, 35, 467\}$, $\{0, 6, 47, 427\}$,\adfsplit
 $\{0, 7, 8, 463\}$, $\{0, 9, 195, 219\}$, $\{0, 12, 119, 295\}$, $\{0, 19, 31, 95\}$,\adfsplit
 $\{0, 18, 155, 371\}$, $\{0, 20, 123, 507\}$, $\{0, 16, 75, 407\}$, $\{0, 14, 343, 495\}$,\adfsplit
 $\{0, 70, 311, 339\}$, $\{0, 87, 141, 223\}$, $\{1, 9, 207, 227\}$, $\{0, 49, 71, 291\}$,\adfsplit
 $\{0, 74, 355, 387\}$, $\{0, 22, 315, 399\}$, $\{0, 26, 255, 435\}$, $\{0, 57, 83, 511\}$,\adfsplit
 $\{1, 5, 379, 431\}$, $\{0, 134, 159, 307\}$, $\{1, 22, 127, 487\}$, $\{1, 14, 191, 403\}$,\adfsplit
 $\{1, 21, 363, 411\}$, $\{1, 102, 115, 439\}$, $\{1, 143, 238, 351\}$, $\{1, 42, 275, 387\}$,\adfsplit
 $\{1, 87, 101, 203\}$, $\{0, 94, 203, 362\}$, $\{1, 26, 251, 503\}$, $\{0, 143, 225, 454\}$,\adfsplit
 $\{0, 290, 370, 499\}$, $\{1, 166, 331, 430\}$, $\{1, 90, 95, 366\}$, $\{1, 77, 231, 310\}$,\adfsplit
 $\{1, 6, 79, 141\}$, $\{1, 150, 171, 173\}$, $\{1, 65, 447, 510\}$, $\{2, 103, 178, 406\}$,\adfsplit
 $\{1, 323, 342, 466\}$, $\{0, 24, 202, 422\}$, $\{0, 28, 106, 346\}$, $\{0, 32, 122, 434\}$,\adfsplit
 $\{0, 36, 146, 306\}$, $\{0, 38, 52, 274\}$, $\{0, 40, 190, 282\}$, $\{0, 46, 382, 440\}$,\adfsplit
 $\{0, 45, 342, 430\}$, $\{0, 53, 166, 174\}$, $\{0, 69, 310, 314\}$, $\{0, 77, 186, 250\}$,\adfsplit
 $\{0, 82, 116, 506\}$, $\{0, 294, 393, 426\}$, $\{0, 117, 334, 378\}$, $\{0, 358, 394, 509\}$,\adfsplit
 $\{0, 198, 210, 248\}$, $\{0, 169, 238, 369\}$, $\{0, 48, 285, 462\}$, $\{0, 118, 193, 457\}$,\adfsplit
 $\{1, 37, 225, 434\}$, $\{0, 194, 413, 445\}$, $\{1, 66, 85, 333\}$, $\{0, 286, 429, 521\}$,\adfsplit
 $\{0, 76, 214, 513\}$, $\{0, 92, 277, 446\}$, $\{1, 54, 217, 397\}$, $\{0, 101, 173, 337\}$,\adfsplit
 $\{0, 85, 153, 420\}$, $\{0, 89, 132, 333\}$, $\{0, 181, 297, 421\}$, $\{0, 61, 289, 329\}$,\adfsplit
 $\{0, 128, 325, 380\}$, $\{0, 72, 233, 453\}$, $\{0, 68, 148, 372\}$, $\{0, 113, 204, 388\}$,\adfsplit
 $\{0, 133, 140, 348\}$, $\{0, 105, 212, 324\}$, $\{0, 96, 301, 364\}$, $\{0, 88, 253, 280\}$,\adfsplit
 $\{0, 137, 188, 308\}$
%adfPENTend
}

\noindent under the action of the mapping $x \mapsto x + 4 \adfmod{524}$.
The deficiency graph is connected and has girth 6.
% The opposite lines are generated by the first 4 blocks: True
% End of PENT(4, 173)
%%%%%%%%%%%%%%%%%%%%%%%%%%%%%%%%%%%%%%%%%%%%%%%%%%%%%%%%%%%%%%%%%%%%%%%%%%%%%%%%%%%%%%%%%%

{\noindent\boldmath $\adfPENT(4, 180)$}~ With point set $Z_{545}$, the 24525 lines are generated from

{\scriptsize
%adfPENTstart:1:4:180:5:5:
 $\{100, 118, 126, 445\}$, $\{29, 64, 420, 524\}$, $\{242, 307, 324, 333\}$, $\{108, 217, 430, 443\}$,\adfsplit
 $\{26, 227, 486, 521\}$, $\{13, 147, 196, 24\}$, $\{258, 308, 345, 273\}$, $\{369, 406, 479, 506\}$,\adfsplit
 $\{532, 277, 97, 91\}$, $\{208, 538, 36, 29\}$, $\{541, 179, 189, 27\}$, $\{83, 165, 278, 11\}$,\adfsplit
 $\{102, 319, 281, 451\}$, $\{467, 544, 368, 427\}$, $\{422, 541, 293, 153\}$, $\{185, 432, 139, 378\}$,\adfsplit
 $\{259, 440, 160, 419\}$, $\{442, 359, 462, 285\}$, $\{39, 330, 294, 45\}$, $\{418, 395, 381, 334\}$,\adfsplit
 $\{281, 327, 201, 47\}$, $\{282, 46, 500, 93\}$, $\{157, 122, 27, 4\}$, $\{492, 242, 506, 117\}$,\adfsplit
 $\{216, 423, 20, 543\}$, $\{479, 185, 340, 38\}$, $\{352, 143, 147, 46\}$, $\{106, 77, 429, 541\}$,\adfsplit
 $\{466, 99, 124, 404\}$, $\{170, 140, 146, 440\}$, $\{446, 359, 414, 429\}$, $\{320, 341, 472, 488\}$,\adfsplit
 $\{296, 388, 243, 418\}$, $\{249, 26, 318, 131\}$, $\{171, 160, 531, 224\}$, $\{178, 251, 149, 525\}$,\adfsplit
 $\{80, 116, 396, 372\}$, $\{493, 289, 258, 146\}$, $\{254, 217, 441, 140\}$, $\{442, 143, 498, 57\}$,\adfsplit
 $\{325, 276, 449, 447\}$, $\{346, 198, 258, 117\}$, $\{136, 190, 445, 352\}$, $\{285, 274, 178, 140\}$,\adfsplit
 $\{42, 155, 354, 169\}$, $\{301, 296, 309, 207\}$, $\{266, 310, 382, 241\}$, $\{278, 340, 189, 506\}$,\adfsplit
 $\{401, 306, 113, 118\}$, $\{363, 72, 111, 433\}$, $\{195, 22, 179, 303\}$, $\{447, 459, 174, 303\}$,\adfsplit
 $\{204, 145, 384, 148\}$, $\{147, 314, 181, 413\}$, $\{232, 459, 426, 214\}$, $\{261, 206, 100, 116\}$,\adfsplit
 $\{421, 94, 393, 268\}$, $\{12, 544, 78, 53\}$, $\{459, 487, 354, 0\}$, $\{366, 26, 223, 292\}$,\adfsplit
 $\{226, 432, 265, 296\}$, $\{293, 292, 195, 199\}$, $\{54, 131, 7, 119\}$, $\{259, 527, 24, 62\}$,\adfsplit
 $\{103, 336, 258, 418\}$, $\{133, 236, 466, 320\}$, $\{479, 135, 413, 130\}$, $\{194, 411, 452, 431\}$,\adfsplit
 $\{107, 199, 242, 44\}$, $\{14, 462, 223, 161\}$, $\{305, 181, 424, 421\}$, $\{418, 196, 198, 510\}$,\adfsplit
 $\{157, 245, 321, 288\}$, $\{211, 54, 325, 435\}$, $\{458, 169, 25, 153\}$, $\{363, 261, 94, 260\}$,\adfsplit
 $\{276, 205, 180, 148\}$, $\{6, 197, 155, 434\}$, $\{403, 198, 152, 162\}$, $\{216, 201, 486, 15\}$,\adfsplit
 $\{169, 151, 202, 83\}$, $\{68, 37, 230, 267\}$, $\{209, 128, 109, 261\}$, $\{59, 503, 131, 261\}$,\adfsplit
 $\{505, 432, 486, 135\}$, $\{348, 198, 427, 149\}$, $\{455, 482, 119, 35\}$, $\{75, 37, 374, 524\}$,\adfsplit
 $\{119, 95, 348, 289\}$, $\{153, 208, 335, 400\}$, $\{60, 522, 252, 245\}$, $\{57, 226, 185, 452\}$,\adfsplit
 $\{379, 334, 358, 480\}$, $\{414, 30, 454, 306\}$, $\{233, 16, 137, 398\}$, $\{90, 298, 98, 77\}$,\adfsplit
 $\{90, 438, 159, 362\}$, $\{474, 248, 60, 70\}$, $\{100, 348, 179, 11\}$, $\{79, 299, 412, 57\}$,\adfsplit
 $\{123, 393, 516, 405\}$, $\{404, 172, 48, 115\}$, $\{536, 287, 68, 317\}$, $\{337, 120, 336, 235\}$,\adfsplit
 $\{474, 172, 380, 216\}$, $\{0, 2, 66, 401\}$, $\{0, 9, 51, 266\}$, $\{0, 12, 151, 286\}$,\adfsplit
 $\{0, 15, 61, 221\}$, $\{0, 17, 86, 411\}$, $\{0, 19, 141, 181\}$, $\{0, 20, 191, 436\}$,\adfsplit
 $\{0, 22, 121, 356\}$, $\{0, 28, 326, 446\}$, $\{0, 29, 281, 341\}$, $\{0, 32, 176, 241\}$,\adfsplit
 $\{0, 33, 56, 246\}$, $\{0, 34, 131, 331\}$, $\{0, 44, 291, 541\}$, $\{0, 47, 361, 516\}$,\adfsplit
 $\{0, 40, 346, 376\}$, $\{0, 54, 261, 441\}$, $\{0, 39, 311, 386\}$, $\{0, 62, 81, 511\}$,\adfsplit
 $\{1, 8, 17, 536\}$, $\{1, 28, 67, 371\}$, $\{1, 12, 126, 138\}$, $\{0, 48, 156, 536\}$,\adfsplit
 $\{0, 53, 256, 481\}$, $\{0, 82, 406, 451\}$, $\{0, 52, 136, 486\}$, $\{1, 23, 141, 291\}$,\adfsplit
 $\{0, 58, 85, 296\}$, $\{1, 24, 44, 288\}$, $\{0, 49, 184, 466\}$, $\{0, 73, 252, 426\}$,\adfsplit
 $\{0, 93, 224, 366\}$, $\{1, 37, 88, 378\}$, $\{1, 59, 139, 388\}$, $\{1, 33, 43, 154\}$,\adfsplit
 $\{1, 53, 72, 143\}$, $\{1, 79, 84, 87\}$, $\{1, 62, 69, 284\}$, $\{1, 68, 94, 209\}$,\adfsplit
 $\{1, 57, 89, 433\}$, $\{1, 92, 164, 238\}$, $\{1, 104, 108, 279\}$, $\{1, 99, 277, 479\}$,\adfsplit
 $\{1, 107, 168, 364\}$, $\{1, 124, 163, 354\}$, $\{1, 162, 247, 369\}$, $\{1, 178, 342, 489\}$,\adfsplit
 $\{1, 167, 303, 488\}$, $\{1, 214, 228, 387\}$, $\{1, 189, 262, 439\}$, $\{1, 114, 304, 368\}$,\adfsplit
 $\{1, 182, 269, 483\}$, $\{1, 133, 314, 453\}$, $\{1, 254, 397, 408\}$, $\{1, 183, 448, 454\}$,\adfsplit
 $\{1, 292, 347, 498\}$, $\{1, 319, 417, 467\}$, $\{1, 412, 437, 464\}$, $\{1, 349, 463, 487\}$,\adfsplit
 $\{1, 362, 537, 543\}$, $\{1, 283, 332, 533\}$, $\{1, 357, 419, 544\}$, $\{1, 508, 528, 542\}$,\adfsplit
 $\{0, 35, 339, 469\}$, $\{0, 45, 219, 314\}$, $\{0, 50, 284, 489\}$, $\{0, 55, 319, 464\}$,\adfsplit
 $\{0, 83, 229, 369\}$, $\{0, 63, 179, 389\}$, $\{0, 78, 120, 479\}$, $\{0, 88, 89, 149\}$,\adfsplit
 $\{0, 67, 454, 544\}$, $\{0, 87, 273, 504\}$, $\{0, 68, 109, 309\}$, $\{0, 112, 374, 494\}$,\adfsplit
 $\{0, 90, 334, 468\}$, $\{0, 137, 204, 429\}$, $\{2, 59, 227, 334\}$, $\{0, 129, 163, 324\}$,\adfsplit
 $\{0, 117, 428, 474\}$, $\{2, 118, 439, 469\}$, $\{2, 123, 174, 249\}$, $\{0, 214, 258, 302\}$,\adfsplit
 $\{0, 282, 419, 538\}$, $\{0, 153, 355, 514\}$, $\{0, 202, 308, 519\}$, $\{3, 48, 254, 413\}$,\adfsplit
 $\{2, 158, 238, 309\}$, $\{2, 178, 288, 364\}$, $\{3, 78, 169, 453\}$, $\{2, 193, 299, 347\}$,\adfsplit
 $\{0, 329, 382, 493\}$, $\{2, 62, 194, 518\}$, $\{3, 139, 248, 288\}$, $\{2, 47, 187, 329\}$,\adfsplit
 $\{0, 80, 203, 318\}$, $\{0, 140, 313, 443\}$, $\{0, 183, 268, 315\}$, $\{0, 228, 328, 502\}$,\adfsplit
 $\{0, 222, 323, 388\}$, $\{2, 102, 222, 453\}$, $\{0, 213, 310, 542\}$, $\{0, 143, 227, 395\}$,\adfsplit
 $\{0, 105, 393, 467\}$, $\{0, 172, 333, 350\}$, $\{0, 158, 307, 335\}$, $\{0, 148, 262, 522\}$,\adfsplit
 $\{0, 133, 197, 305\}$, $\{0, 70, 408, 410\}$, $\{2, 17, 213, 352\}$, $\{0, 95, 237, 537\}$,\adfsplit
 $\{0, 92, 207, 325\}$, $\{0, 407, 412, 482\}$, $\{0, 132, 287, 320\}$, $\{0, 180, 392, 497\}$,\adfsplit
 $\{0, 107, 160, 330\}$, $\{0, 60, 187, 402\}$, $\{0, 75, 257, 422\}$, $\{0, 147, 165, 295\}$,\adfsplit
 $\{0, 242, 387, 477\}$
%adfPENTend
}

\noindent under the action of the mapping $x \mapsto x + 5 \adfmod{545}$.
The deficiency graph is connected and has girth 6.
% The opposite lines are generated by the first 5 blocks: True
% End of PENT(4, 180)
%%%%%%%%%%%%%%%%%%%%%%%%%%%%%%%%%%%%%%%%%%%%%%%%%%%%%%%%%%%%%%%%%%%%%%%%%%%%%%%%%%%%%%%%%%

%%%%%%%%%%%%%%%%%%%%%%%%%%%%%%%%%%%%%%%%%%%%%%%%%%%%%%%%%%%%%%%%%%%%%%%%%%%%%%%%%%%%%%%%%%
%%%%%%%%%%%%%%%%%%%%%%%%%%%%%%%%%%%%%%%%%%%%%%%%%%%%%%%%%%%%%%%%%%%%%%%%%%%%%%%%%%%%%%%%%%
%%%%%%%%%%%%%%%%%%%%%%%%%%%%%%%%%%%%%%%%%%%%%%%%%%%%%%%%%%%%%%%%%%%%%%%%%%%%%%%%%%%%%%%%%%
%%%%%%%%%%%%%%%%%%%%%%%%%%%%%%%%%%%%%%%%%%%%%%%%%%%%%%%%%%%%%%%%%%%%%%%%%%%%%%%%%%%%%%%%%%

\end{document}